\documentclass[12pt, a4paper]{article}
\usepackage{amsmath}
\usepackage{graphicx}
\usepackage{enumerate}
\usepackage{natbib}
\usepackage{url}
\usepackage{bm, bbm, enumitem, mathtools, booktabs, multirow, amsthm, amsfonts, amssymb}
\usepackage{hyperref}
\usepackage{algorithm, algpseudocode}

\hypersetup{
	colorlinks=true,
	linkcolor=black,
	citecolor=black,
	urlcolor=blue
}

\numberwithin{equation}{section}

\theoremstyle{plain}
\newtheorem{thm}{Theorem}[section]
\newtheorem{cor}{Corollary}[section]
\newtheorem{lemma}{Lemma}[section]

\theoremstyle{remark}
\newtheorem*{remark}{Remark}

\newcommand{\yvec}{\bm{y}}
\newcommand{\xvec}{\bm{x}}
\newcommand{\xmat}{\bm{X}}

\newcommand{\bvec}{\bm{\beta}}
\newcommand{\evec}{\bm{\eta}}
\newcommand{\nvec}{\bm{\nu}}

\newcommand{\Cmat}{\mathbf{C}}
\newcommand{\Mset}{\mathcal{M}}
\newcommand{\Sset}{\mathcal{S}}
\newcommand{\Aset}{\mathcal{A}}
\newcommand{\Eset}{\mathcal{E}}

\newcommand{\Kna}{\widehat{\bm K}_{n,a}}
\newcommand{\Kno}{\widehat{\bm K}_{n,0}}

\newcommand{\reals}{\mathbb{R}}
\newcommand{\ind}{\mathbbm{1}}

\newcommand{\defeq}{\vcentcolon =}
\newcommand{\norm}[1]{\left\lVert#1\right\rVert}
\newcommand{\shortnorm}[1]{\lVert#1\rVert}

\DeclareMathOperator*{\argmin}{arg\,min}
\DeclareMathOperator{\oracle}{oracle}
\DeclareMathOperator{\initial}{initial}
\DeclareMathOperator{\lasso}{lasso}
\DeclareMathOperator{\sgn}{sgn}
\DeclareMathOperator{\E}{E}
\DeclareMathOperator{\Cov}{Cov}
\DeclareMathOperator{\diag}{diag}
\DeclareMathOperator{\tr}{tr}

\allowdisplaybreaks

\begin{document}

\def\spacingset#1{\renewcommand{\baselinestretch}%
	{#1}\small\normalsize} \spacingset{1}

\title{Strong Oracle Guarantees for Partial Penalized Tests of High Dimensional Generalized Linear Models}
\author{Tate Jacobson \hspace{.2cm}\\
	Department of Statistics, Oregon State University}
	
\date{}

\maketitle

\begin{abstract}
	Partial penalized tests provide flexible approaches to testing linear hypotheses in high dimensional generalized linear models.
	However, because the estimators used in these tests are local minimizers of potentially non-convex folded-concave penalized objectives, the solutions one computes in practice may not coincide with the unknown local minima for which we have nice theoretical guarantees.
	To close this gap between theory and computation, we introduce local linear approximation (LLA) algorithms to compute the full and reduced model estimators for these tests and develop theory specifically for the LLA solutions.
	We prove that our LLA algorithms converge to oracle estimators for the full and reduced models in two steps with overwhelming probability.
	We then leverage this strong oracle result and the asymptotic properties of the oracle estimators to show that the partial penalized test statistics evaluated at the two-step LLA solutions are approximately chi-square in large samples, giving us guarantees for the tests using specific computed solutions and thereby closing the theoretical gap.
	We conduct simulation studies to assess the finite-sample performance of our testing procedures, finding that partial penalized tests using the LLA solutions agree with tests using the oracle estimators, and demonstrate our testing procedures in a real data application.
\end{abstract}

\noindent%
\vfill

\newpage
\spacingset{1.45}

\section{Introduction}

In applications from genomics to economics, data containing more features than observations have become increasingly common.
The dimensionality of these data creates unique challenges for statistical inference.
Suppose we want to test a low dimensional subvector $\bvec_{\Mset}^*$ of the true parameter vector $\bvec^*$, which we have selected \textit{a priori}.
Penalization methods, such as the lasso \citep{Tibshirani1996}, enable us to estimate $\bvec^*$, but do so at the cost of introducing bias and inducing sparsity, both of which complicate inference.
Several authors address these problems by using the debiasing (or desparsifying) technique, in which a bias-correction term is added to a lasso-type estimator to produce an estimator with a tractable limiting distribution \citep{VanDeGeer2014, Zhang2014, Javanmard2014, Ning2017, Cai2017, Cai2021}.
Others use a partial penalization approach, penalizing the high dimensional nuisance parameters $\bvec_{\Mset^c}$ but leaving $\bvec_{\Mset}$ unpenalized in order to avoid introducing bias or imposing sparsity in the estimation of $\bvec_{\Mset}^*$ \citep{Wang2014, Wang2017, Shi2019, Jacobson2023}. 

Among testing procedures for high dimensional generalized linear models (GLMs), the partial penalized Wald, score, and likelihood ratio tests proposed by \cite{Shi2019} stand out for their exceptional flexibility.
They can be used to test all linear hypotheses, which take the general form $\text{H}_0: \Cmat \bvec^*_{\Mset} = \bm{t}$
and cover several important special cases, including simultaneous significance tests for multiple predictors and tests of contrasts.
In addition, \cite{Shi2019} demonstrate the asymptotic validity of their partial penalized tests in a broad high dimensional regime, in which the number of predictors $p$, the number of parameters being tested, the number of constraints under the null, and the number of nonzero parameters in the true model can all diverge with the sample size $n$, with $p$ diverging at an exponential rate.

While \citeauthor{Shi2019}'s (\citeyear{Shi2019}) partial penalized tests have been shown to possess desirable asymptotic properties, computation of the partial penalized estimators poses a theoretical problem. 
To reduce estimation bias and achieve selection consistency, \cite{Shi2019} use a folded-concave penalty in their partial penalized objective function.
As a consequence, this objective is likely to be non-convex and have multiple local minima.
In practice, \cite{Shi2019} use alternating direction method of multipliers (ADMM) algorithms to compute their partial penalized estimators.
In their theoretical study, they prove that there exist local solutions to their partial penalized estimation problems such that the partial penalized test statistics evaluated at these ``nice'' local solutions are approximately chi-square for large $n$.
There is no guarantee, however, that their ADMM algorithms will arrive at these unknown ``nice'' local solutions rather than some other local solutions.
As such, it remains an open question whether the computed local solutions possess the same asymptotic properties.

Other studies of partial penalized tests present the same disconnect between theory and computation.
\cite{Wang2014}, \cite{Wang2017}, and \cite{Jacobson2023} similarly use folded-concave penalties to obtain selection-consistent partial penalized estimators for their tests, giving them objectives which may have multiple local minima. 
Like \cite{Shi2019}, these authors establish the existence of unknown local solutions to their partial penalized estimation problems with nice theoretical properties rather than providing guarantees for specific computed solutions.

In this paper, we close the gap between theory and computation for partial penalized tests of high dimensional GLMs by establishing theoretical guarantees for specific computed solutions to the partial penalized estimation problems with folded-concave penalties.
In particular, we propose local linear approximation (LLA) algorithms to compute the folded-concave penalized reduced and full model estimators for testing $\text{H}_0: \Cmat \bvec^*_{\Mset} = \bm{t}$ and develop theory specifically for the LLA solutions.
We prove that our LLA algorithms converge to oracle estimators for the reduced and full models in two steps with probability converging to $1$ as $n \to \infty$---that is, the two-step LLA solutions possess the \textit{strong oracle property}. 
As a consequence, asymptotic results for the oracle estimators can be extended to the two-step LLA solutions.
We leverage this corollary of the strong oracle property to show that the partial penalized Wald, score, and likelihood ratio test statistics evaluated at the two-step LLA solutions are approximately chi-square for large $n$, thereby establishing the asymptotic validity of the partial penalized tests using the two-step LLA solutions.
Because the two-step LLA solutions can be directly computed in practice, this result closes the previous gap between theory and practice for partial penalized tests of GLMs.

Our theoretical study consists of two stages: (1) proving that the two-step LLA solutions possess the strong oracle property and (2) using the asymptotic properties of the oracle estimators to establish guarantees for the two-step LLA solutions. 
For the first stage, we adopt an approach similar to that of \cite{Fan2014}, identifying conditions under which the LLA algorithms converge to the oracle estimators in two steps and then showing that those conditions hold with probability converging to $1$ as $n \to \infty$.
Like \cite{Fan2014}, we are not concerned with whether the computed local solutions are the global minimizers; we only care that they possess the desired properties.

In the second stage of our theoretical study, we examine the asymptotic properties of the oracle estimators and use the strong oracle property to extend them to the two-step LLA solutions.
We examine the same broad ultra-high dimensional regime as \cite{Shi2019}. 
In this setting, the number of predictors in the oracle models diverges with $n$, meaning that classical asymptotic results for GLMs do not apply.
As such, we develop new theory for the Wald, score, and likelihood ratio tests in this regime.

While the oracle properties of folded-concave penalized estimators have been studied extensively \citep{Fan2001, Fan2004, Zhang2010, Fan2011, Fan2014},
prior research has focused primarily on the implications for estimation and variable selection.
We focus, instead, on inference, extending oracle estimation theory to establish guarantees for high dimensional hypothesis testing procedures.
The traditional oracle estimator is not a suitable target for testing $\text{H}_0: \Cmat \bvec^*_{\Mset} = \bm t$, as $\Mset$ may not be contained in the support set of $\bvec^*$.
As such, we introduce new oracle estimators, which include the elements of $\Mset$ among the set of important parameters regardless of whether they are in the support set, to serve as our targets for testing $\text{H}_0: \Cmat \bvec^*_{\Mset} = \bm t$.
Though partial penalized tests of GLMs are our main objects of study, we also demonstrate how such oracle estimators can be obtained via LLA algorithms in a broader class of partial penalized estimation problems. 
In doing so, we lay the groundwork for studying other high dimensional inference procedures using our strong oracle approach.

The rest of the paper is organized as follows:
In Section \ref{sec:partial penalized tests}, we review partial penalized tests for high dimensional GLMs and discuss the limitations of existing theory for these tests.
In Section \ref{sec:lla algorithms}, we introduce oracle estimators for the full and reduced models under $\text{H}_0:\Cmat \bvec^*_{\Mset} = \bm t$ and develop LLA algorithms for computing the partial penalized estimators for those models.  
In Section \ref{sec:lla theory}, we identify conditions under which the LLA algorithms converge to their corresponding oracle estimators in two steps and show that these conditions hold with high probability in high dimensional GLMs.
In Section \ref{sec:estimator asymptotics}, we study the asymptotic behavior of the oracle estimators and, using the strong oracle property, extend those results to the two-step LLA solutions.
In particular, we leverage these asymptotic properties to derive approximate distributions for the partial penalized test statistics evaluated at the two-step LLA solutions.
In Section \ref{sec:sims}, we conduct simulation studies comparing the partial penalized tests evaluated at the LLA solutions with tests using the oracle estimators. 
We close with some concluding remarks in Section \ref{sec:conclusion}.
The appendices contain theoretical proofs, a discussion of computational details of the LLA algorithms, additional simulations, and an application of our tests to assess the impact of nucleoside reverse transcriptase inhibitors (NRTIs) on the likelihood of repeat virological failure in people with human immunodeficiency virus (HIV).

\subsection*{Notation}
We adopt the following notation throughout our study.
Let $\mathbf A = [a_{i,j}] \in \reals^{n\times m}$ be a matrix.
We let $\mathbf{A}' = [a_{j,i}]$ denote the transpose of $\mathbf{A}$.
We use $\lambda_{\max}\{ \mathbf A \}$ and $\lambda_{\min}\{ \mathbf A \}$ to denote the largest and smallest eigenvalues, respectively, of $\mathbf A$.
If $\mathbf A$ is square, we use $\tr\{\mathbf A\}$ to denote its trace.
We use the following matrix norms: the $\ell_2$-norm $\norm{\mathbf A}_2 = \lambda_{\max}^{1/2}\{\mathbf A ' \mathbf A\}$, the $\ell_1$-norm $\norm{\mathbf A}_1 = \max_j \sum_i |a_{i,j}|$, the $\ell_{\infty}$-norm $\norm{\mathbf A}_{\infty} = \max_i \sum_j |a_{i,j}|$, the entrywise maximum $\norm{\mathbf A}_{\max} = \max_{i,j} |a_{i,j}|$, and the entrywise minimum $\norm{\mathbf A}_{\min} = \min_{i,j} |a_{i,j}|$.
For an index $j \in \{1, \ldots, m\}$, we use $\mathbf{A}_j$ to denote the $j$th column of $\mathbf A$.
Likewise, given a set $\mathcal T \subseteq \{1, \ldots, m\}$, we use $\mathbf A_{\mathcal T}$ to denote the submatrix of columns with indices in $\mathcal T$.
Given another set $\mathcal S \subseteq \{1, \ldots, n\}$, we let $\mathbf A_{\mathcal S, \mathcal T} = [a_{i,j}]_{i \in \mathcal S, j \in \mathcal T}$.
We use $|\mathcal T|$ to denote the cardinality of $\mathcal T$. 
For a vector $\bm v \in \reals^m$, we use $\bm v_{\mathcal T}$ to denote the subvector with indices in $\mathcal T$.
We use $\diag\{\bm v\}$ to denote the diagonal matrix with $\bm v$ along its diagonal.
Given a second vector $\bm w \in \reals^m$, we let $\bm v \circ \bm w$ denote the entrywise product.
For a function $f: \reals^m \to \reals$, we use $\nabla_{\mathcal T} f(\bm v)$ and $\nabla^2_{\mathcal T} f(\bm v)$ to denote the gradient and Hessian of $f$ with respect to $\bm v_{\mathcal T}$.

If $\mathbf S \in \reals^{m\times m}$ is a symmetric matrix, then by the spectral theorem it can expressed as a matrix product $\mathbf S = \mathbf Q  \mathbf D \mathbf Q'$, where $\mathbf Q$ is orthogonal and $\mathbf D = \diag\{d_1, \ldots, d_m\}$. If  $\mathbf S$ is positive semi-definite, then $d_1, \ldots, d_m \geq 0$. In this case, we define $\mathbf D^{1/2} = \diag\{d_1^{1/2}, \ldots, d_m^{1/2}\}$ and $\mathbf S^{1/2} = \mathbf Q  \mathbf D^{1/2} \mathbf Q'$ as well as $\mathbf D^{-1/2} = \diag\{d_1^{-1/2}, \ldots, d_m^{-1/2}\}$ and $\mathbf S^{-1/2} = \mathbf Q  \mathbf D^{-1/2} \mathbf Q'$.

\section{Partial Penalized Tests} \label{sec:partial penalized tests}
We begin by reviewing the partial penalized Wald, score, and likelihood ratio tests for high dimensional GLMs.
In a GLM, the conditional distribution of our response $y$ given the predictors $\xvec \in \reals^p$ belongs to an exponential family, with canonical probability density function of the form:
\begin{equation}
	p(y|\xvec, \bvec^*, \phi^*) = \exp \left(\frac{y \xvec'\bvec^* - b(\xvec'\bvec^*)}{\phi^*}\right)c(y) \text{,} \label{assumption:glm}
\end{equation}
where $\bvec^* \in \reals^p$ denotes the true coefficient vector and $\phi^*$ denotes the true value of the dispersion parameter (if one exists).

Suppose we observe a sample $\{(\xvec_i, y_i)\}_{i=1}^n$ of independent and identically distributed replicates of $(\xvec, y)$ and that we want to test the linear hypothesis
\begin{equation}
	\text{H}_0: \Cmat \bvec^*_{\Mset} = \bm{t}, \label{eqn:H0}
\end{equation}
where $\Mset \subset \{0,1,\ldots,p\}$, $\Cmat \in \reals^{r \times |\Mset|}$, and $\bm{t} \in \reals^r$.
Define $m = |\Mset|$.
We reorder our parameters so that $\Mset = \{1, \ldots, m\}$ and $\Mset^c = \{m+1, \ldots, p\}$.
We assume that $m \ll p$. 
In addition, we assume that the constraint matrix $\Cmat$ has full row rank so that none of the constraints under the null are redundant,
which implies $r \leq m$.

The partial penalized likelihood function for testing \eqref{eqn:H0} is given by
\begin{equation}
	Q_n(\bvec) = -\frac{1}{n} \sum_{i=1}^{n} \left\{ y_i \xvec_i'\bvec - b(\xvec_i'\bvec) \right\} + \sum_{j\notin \Mset} p_{\lambda}(|\beta_j|). \label{eqn:partial penalized likelihood}
\end{equation}
The parameters being tested under \eqref{eqn:H0} are left unpenalized in $Q_n(\bvec)$. 
This allows us to avoid imposing minimum signal strength assumptions on the elements of $\bvec_{\Mset}^*$ and ensures that the partial penalized tests have power at local alternatives.
\cite{Shi2019} introduced the following partial penalized estimators for testing \eqref{eqn:H0}:
\begin{align}
	\hat{\bvec}_0 &= \argmin_{\bvec} Q_n(\bvec) \text{\;\;\; subject to } \Cmat \bvec_{\Mset} = \bm{t} \label{eqn:Shi reduced}\\
	\hat{\bvec}_a & = \argmin_{\bvec} Q_n(\bvec). \label{eqn:Shi full}
\end{align}
We refer to $\hat{\bvec}_0$, which satisfies the constraints of the null hypothesis, as the \textit{reduced model} estimator and $\hat{\bvec}_a$, the unconstrained minimizer, as the \textit{full model} estimator.

We use $\hat{\bvec}_0$ and $\hat{\bvec}_a$ to compute the partial penalized test statistics.
For any estimator $\hat{\bvec}$, we define $\hat{\Sset}(\hat{\bvec}) = \{j \in \Mset^c : \hat{\beta}_{j} \neq 0 \}$. 
Using this notation, the estimated support set for the full model is $\hat{\Sset}(\hat{\bvec}_a)$ and the estimated support set for the reduced model is $\hat{\Sset}(\hat{\bvec}_0)$.
Define $\ell_n(\bvec) = -\frac{1}{n} \sum_{i=1}^{n} \left\{ y_i \xvec_i'\bvec - b(\xvec_i'\bvec) \right\}$ and 
$\widehat{\bm K}_n(\hat \bvec) = \nabla_{\Mset \cup \hat{\Sset}(\hat{\bvec})}^2 \ell_n(\hat{\bvec})$ for an estimator $\hat{\bvec}$.
Let $\hat{\phi}$ be a consistent estimator of $\phi^*$. 
The partial penalized Wald, score, and likelihood ratio tests for testing \eqref{eqn:H0} are given by
\begin{equation} 
	T_W(\hat{\bvec}_a) = (\Cmat \hat{\bvec}_{a, \Mset} - \bm{t})'
	\left( \Cmat  \left( n \widehat{\bm K}_n(\hat{\bvec}_a) \right)^{-1}_{\Mset, \Mset}  \Cmat' \right)^{-1}
	(\Cmat \hat{\bvec}_{a, \Mset} -\bm{t})/\hat{\phi}, \label{eqn:Wald stat}
\end{equation}
\begin{equation}
	T_S(\hat{\bvec}_0) = n \nabla_{\Mset \cup \hat{\Sset}(\hat{\bvec}_0)} \ell_n(\hat{\bvec}_0)'
	\left(n \widehat{\bm K}_n(\hat \bvec_0) \right)^{-1}
	n \nabla_{\Mset \cup \hat{\Sset}(\hat{\bvec}_0)} \ell_n(\hat{\bvec}_0)/\hat{\phi}, \label{eqn:score stat}
\end{equation}
and
\begin{equation}
	T_L(\hat{\bvec}_a, \hat{\bvec}_0) = -2n(\ell_n(\hat{\bvec}_a) - \ell_n(\hat{\bvec}_0) )/\hat{\phi}, \label{eqn:lrt stat}
\end{equation}
respectively.
We express the partial penalized test statistics as functions of the estimators $\hat{\bvec}_0$ and $\hat{\bvec}_a$ to underscore that the properties of the test statistics depend on the estimators used to compute them.

Suppose we want to test \eqref{eqn:H0} at significance level $\alpha \in (0,1)$. 
For each of the partial penalized test statistics, $T \in \{T_W(\hat{\bvec}_a), T_S(\hat{\bvec}_0), T_L(\hat{\bvec}_a, \hat{\bvec}_0)\}$, we reject $\text{H}_0$ when $T > \chi^2_\alpha(r)$, where $\chi^2_\alpha(r)$ denotes the upper-$\alpha$ quantile of a central chi-square distribution with $r$ degrees of freedom.

\cite{Shi2019} assume that the penalty function $p_{\lambda}(t)$ in \eqref{eqn:partial penalized likelihood} is folded-concave to reduce estimation bias and obtain partial penalized estimators which are selection consistent.
As a consequence, $Q_n(\bvec)$ is likely to be non-convex and have multiple local minima.
\cite{Shi2019} prove that there exist (unknown) local solutions, $\hat{\bvec}^*_0$ and $\hat{\bvec}^*_a$, to \eqref{eqn:Shi reduced} and \eqref{eqn:Shi full}, respectively, which are selection consistent, converge to $\bvec^*$ quickly, and have nice limiting expressions.
They leverage these properties to prove that $T_W(\hat{\bvec}_a^*)$, $T_S(\hat{\bvec}_0^*)$, and $T_L(\hat{\bvec}_a^*, \hat{\bvec}_0^*)$ approximately follow a central chi-square distribution with $r$ degrees of freedom under the null hypothesis and a non-central chi-square distribution with $r$ degrees of freedom under local alternatives when $n$ is large.
There is no guarantee, however, that $\hat{\bvec}^*_0$ and $\hat{\bvec}^*_a$ are the local solutions we will obtain in practice.
As such, it remains an open question whether the computed solutions to \eqref{eqn:Shi reduced} and \eqref{eqn:Shi full} possess the same desirable properties.

\section{Partial Penalized LLA Algorithms} \label{sec:lla algorithms}
For the next few sections, we broaden the scope of our discussion beyond GLMs to a broader class of partial penalized estimation problems.
Let $\ell_n(\bvec)$ be a convex loss function.
As before, we want to test \eqref{eqn:H0} where $\bvec^*$ denotes the true parameter vector.
We define the \textit{partial penalized loss} by
\begin{equation}
	Q_n(\bvec) = \ell_n(\bvec) + \sum_{j\notin \Mset} p_{\lambda}(|\beta_j|) \label{eqn:partial penalized loss}
\end{equation}
and the reduced and full model partial penalized estimators for testing \eqref{eqn:H0} by
\begin{align}
	\hat{\bvec}_0 &= \argmin_{\bvec} Q_n(\bvec) \text{\;\;\; subject to } \Cmat \bvec_{\Mset} = \bm{t} \label{eqn:reduced estimator} \\
	\hat{\bvec}_a & = \argmin_{\bvec} Q_n(\bvec), \label{eqn:full estimator}
\end{align}
respectively.
This general formulation covers partial penalized GLMs as well as other important statistical models.
If the model parameter is matrix (as in, for example, a Gaussian graphical model), we let $\bvec^*$ denote its vectorization.

Let $\Aset = \{j : \beta_j^* \neq 0\}$ denote the support set for the true model. 
We assume that the true model is sparse, meaning that $|\Aset| \ll p$.
The oracle estimator knows $\Aset$ in advance and leverages that information in estimating $\bvec^*$. 
For our generic convex loss function $\ell_n(\bvec)$, the oracle estimator is given by
$ \hat{\bvec}^{\oracle} = \argmin_{\bvec : \bvec_{\Aset^c} = \bm{0} } \ell_n(\bvec).$
While $\hat \bvec^{\oracle}$ cannot be directly computed in practice, it is a useful theoretical benchmark for high dimensional estimation.

In fully penalized estimation problems, it has been shown that folded-concave penalized estimators have the same limiting distribution as $\hat \bvec^{\oracle}$---a result known as the \textit{oracle property} \citep{Fan2001, Fan2004, Fan2011}.
As we have noted, folded-concave penalized objectives are likely to be non-convex and have multiple local minima.
\cite{Zou2008} proposed an LLA algorithm to turn (potentially non-convex) folded-concave penalized estimation problems into sequences of (convex) weighted lasso problems.
\cite{Fan2014} solved the multiple minima issue for a broad class of fully penalized estimation problems by establishing guarantees specifically for the LLA solutions, proving that the LLA algorithm converges to the oracle estimator in two steps with overwhelming probability---a result they termed the \textit{strong oracle property}.

While $\hat \bvec^{\oracle}$ is a sensible target when one's primary goals are to estimate $\bvec^*$ and recover $\Aset$, it is not suitable for making inference about $\bvec_{\Mset}^*$, as $\Mset$ may not be a subset of $\Aset$.
To suit our goal of testing \eqref{eqn:H0}, we need new oracle estimators to serve as our targets.
To that end, we propose the \textit{reduced oracle estimator} $\hat{\bvec}_0^{\oracle}$ and the \textit{full oracle estimator} $\hat{\bvec}_a^{\oracle}$:
\begin{align}
	\hat{\bvec}_0^{\oracle} & = \argmin_{\bvec : \bvec_{(\Mset \cup \Aset)^c} = \bm{0} } \ell_n(\bvec) \text{\;\;\; subject to } \Cmat \bvec_{\Mset} = \bm{t} \label{eqn:reduced oracle}\\
	\hat{\bvec}_a^{\oracle} & = \argmin_{\bvec : \bvec_{(\Mset \cup \Aset)^c} = \bm{0} } \ell_n(\bvec) . \label{eqn:full oracle}
\end{align}
Intuitively, $\hat{\bvec}_0^{\oracle}$ and $\hat{\bvec}_a^{\oracle}$ correspond to the reduced and full model estimators for testing \eqref{eqn:H0} if $\Aset$ was known.
Unlike $\hat \bvec^{\oracle}$, they include the elements of $\bvec_{\Mset}$ among the set of important parameters regardless of whether they are nonzero in the true model.
This ensures that $\hat{\bvec}_{0,\Mset}^{\oracle}$ and $\hat{\bvec}_{a,\Mset}^{\oracle}$ have tractable limiting distributions, making it possible to quantify their asymptotic uncertainty.

Just as we want an estimator which behaves like $\hat{\bvec}^{\oracle}$ asymptotically when our goal is to estimate $\bvec^*$, we want estimators which behave like $\hat{\bvec}_0^{\oracle}$ and $\hat{\bvec}_a^{\oracle}$ asymptotically when our goal is to test \eqref{eqn:H0}.
To achieve this, we assume that the penalty function $p_{\lambda}(\cdot)$ in \eqref{eqn:partial penalized loss} is folded-concave, satisfying
\begin{enumerate}[label = (\roman*), leftmargin=2\parindent]
	\item $p_{\lambda}(t)$ is increasing and concave on $t \in [0,\infty)$ with $p_{\lambda}(0) = 0$ \label{assumption: penalty i}
	\item $p_{\lambda}(t)$ is differentiable on $t \in (0, \infty)$ with $p_{\lambda}'(0) \defeq p_{\lambda}'(0^+) \geq a_1 \lambda$ \label{assumption: penalty ii}
	\item $p_{\lambda}'(t) \geq a_1 \lambda$ for $t \in (0, a_2 \lambda]$ \label{assumption: penalty iii}
	\item $p_{\lambda}'(t) = 0$ for $t \in [a\lambda, \infty)$ where $a > a_2$ \label{assumption: penalty iv}
\end{enumerate} 
where $a$, $a_1$, and $a_2$ are positive constants.
It is straightforward to show that \ref{assumption: penalty i} - \ref{assumption: penalty iv} are satisfied by the SCAD penalty \citep{Fan2001}, which has derivative
$
p_{\lambda}'(t) = \lambda \ind_{t \leq \lambda} + \frac{(a\lambda - t)_+}{a - 1}\ind_{t > \lambda}
$
where $a > 2$, and the MCP \citep{Zhang2010}, which has derivative
$
p_{\lambda}'(t) = \left(\lambda - \frac{t}{a}\right)_+
$
where $a > 1$, with $a_1 = a_2 = 1$ for the SCAD penalty and $a_1 = 1 - a^{-1}$, $a_2 = 1$ for the MCP.

When $p_{\lambda}(\cdot)$ is folded-concave, the partial penalized objective \eqref{eqn:partial penalized loss} is likely to be non-convex and to have multiple local minima.
Motivated by \citeauthor{Fan2014}'s (\citeyear{Fan2014}) strong oracle results for fully penalized estimators, we propose LLA algorithms to solve the partial penalized estimation problems \eqref{eqn:reduced estimator} and \eqref{eqn:full estimator} with folded-concave penalties, given as Algorithms \ref{algo:LLA for reduced estimator} and \ref{algo:LLA for full estimator} below.

\spacingset{1.25}
\begin{algorithm}[h]
	\caption{LLA algorithm for the reduced model estimator \eqref{eqn:reduced estimator}} \label{algo:LLA for reduced estimator}
	\begin{algorithmic}
		\State Initialize $\hat{\bvec}_0^{(0)} = \hat{\bvec}_0^{\initial}$ and compute the adaptive weights
		$$\hat{\bm{w}}^{(0)} = (\hat{w}_{m+1}^{(0)}, \ldots, \hat{w}_{p}^{(0)})' = (p_{\lambda}'(|\hat \bvec_{0, m+1}^{(0)}|), \ldots, p_{\lambda}'(|\hat \bvec_{0, p}^{(0)}|) )'$$
		\For{$b = 1, \ldots, B$}
			\State Compute $\hat \bvec_0^{(b)}$ by solving the weighted lasso problem
			\begin{equation}
				\hat \bvec_0^{(b)} 
				= \argmin_{\bvec} \ell_n(\bvec) + \sum_{j = m+1}^p \hat w_j^{(b-1)}|\beta_j| \text{ \, subject to \, } \Cmat \bvec_{\Mset} = \bm{t} \label{algo:reduced model lasso update}
			\end{equation}
			\State Update the adaptive weights: $\hat{w}_{j}^{(b)} = p_{\lambda}'(|\hat \bvec_{0, j}^{(b)}|)$ for $j = m+1, \ldots, p$
		\EndFor
	\end{algorithmic}
\end{algorithm}

\begin{algorithm}[h]
	\caption{LLA algorithm for the full model estimator \eqref{eqn:full estimator}} \label{algo:LLA for full estimator}
	\begin{algorithmic}
		\State Initialize $\hat{\bvec}_a^{(0)} = \hat{\bvec}_a^{\initial}$ and compute the adaptive weights
		$$\hat{\bm{w}}^{(0)} = (\hat{w}_{m+1}^{(0)}, \ldots, \hat{w}_{p}^{(0)})' = (p_{\lambda}'(|\hat \bvec_{a, m+1}^{(0)}|), \ldots, p_{\lambda}'(|\hat \bvec_{a, p}^{(0)}|) )'$$
		\For{$b = 1, \ldots, B$}
			\State Compute $\hat \bvec_a^{(b)}$ by solving the weighted lasso problem
			\begin{equation}
				\hat \bvec_a^{(b)} = \argmin_{\bvec} \ell_n(\bvec) + \sum_{j = m+1}^p \hat w_j^{(b-1)}|\beta_j| \label{algo:full model lasso update}
			\end{equation} 
			\State Update the adaptive weights: $\hat{w}_{j}^{(b)} = p_{\lambda}'(|\hat \bvec_{a, j}^{(b)}|)$ for $j = m+1, \ldots, p$
		\EndFor
	\end{algorithmic}
\end{algorithm}
\spacingset{1.45}

Our LLA algorithms differ from the LLA algorithms proposed by \cite{Zou2008} and \cite{Fan2014} in two key respects: 
First, the elements of $\bvec_{\Mset}$ are not penalized in our LLA updates \eqref{algo:reduced model lasso update} and \eqref{algo:full model lasso update}, as our LLAs approximate the partial penalized loss \eqref{eqn:partial penalized loss} rather than a fully penalized loss.
Second, the LLA updates \eqref{algo:reduced model lasso update} in Algorithm \ref{algo:LLA for reduced estimator} are subject to the linear constraints from \eqref{eqn:reduced estimator}.
As a consequence, \eqref{algo:reduced model lasso update} cannot be solved with standard implementations of the lasso.

\section{Strong Oracle Results for the LLA Solutions} \label{sec:lla theory}
Having developed LLA algorithms to solve \eqref{eqn:reduced estimator} and \eqref{eqn:full estimator}, we aim to establish strong oracle results for the two-step LLA solutions.
We start by examining the general case where $\ell_n(\bvec)$ is a convex loss function, then narrow our focus to partial penalized GLMs.

\subsection{Conditions for LLA convergence}
We begin by identifying conditions under which the LLA algorithms for the reduced and full models converge to the reduced and full oracle estimators, respectively, in two steps, in the general setting where $\ell_n(\bvec)$ is a convex loss function.
For ease of presentation, we define $\Sset = \Mset^c \cap \Aset$ such that $\Sset$ and $\Mset$ are disjoint and $\Mset \cup \Sset = \Mset \cup \Aset$. 
Let $s = |\Sset|$.
We assume that $\hat{\bvec}_0^{\oracle}$ and $\hat{\bvec}_a^{\oracle}$ are the unique solutions to \eqref{eqn:reduced oracle} and \eqref{eqn:full oracle}, meaning that they satisfy
\begin{enumerate}[label = (A\arabic*) , start = 1 , leftmargin=2\parindent]
	\item $ 
	\exists_{\nvec \in \reals^r} \ni \nabla_{\Mset \cup \Sset} \ell_n(\hat{\bvec}_0^{\oracle}) 
	= \begin{bmatrix}
		\Cmat'\nvec\\
		\bm{0}
	\end{bmatrix}
	$, $ \Cmat \hat{\bvec}_{0, \Mset}^{\oracle} = \bm{t} $ \label{assumption:reduced oracle unique} \text{ and }
	\item $\nabla_{\Mset \cup \Sset} \ell_n(\hat{\bvec}_{a}^{\oracle}) = \bm{0}$. \label{assumption:full oracle unique}
\end{enumerate}
We further assume that $\bvec^*$ satisfies the following minimum signal strength condition
\begin{enumerate}[label = (A\arabic*) , start = 3 , leftmargin=2\parindent]
	\item $\norm{\bvec^*_{\Sset}}_{\min} > (a+1)\lambda$. \label{assumption: min signal for betaS}
\end{enumerate}

Our first theorem establishes conditions under which the LLA algorithms for \eqref{eqn:reduced estimator} and \eqref{eqn:full estimator} find $\hat{\bvec}_0^{\oracle}$ and $\hat{\bvec}_a^{\oracle}$, respectively, in one step:
\begin{thm} \label{thm:one step LLA}
	Suppose that we are trying to solve the partial penalized minimization problems \eqref{eqn:reduced estimator} and \eqref{eqn:full estimator} with a folded-concave penalty $p_{\lambda}(\cdot)$ satisfying \emph{\ref{assumption: penalty i}} - \emph{\ref{assumption: penalty iv}} and that conditions \emph{\ref{assumption:reduced oracle unique}} - \emph{\ref{assumption: min signal for betaS}} hold.
	Let $a_0 = \min\{1, a_2\}$.
	Under the event
	\begin{equation}
		\Eset_{01} = \left\{ \norm{\hat{\bvec}_{0, \Mset^c}^{\initial} - \bvec_{\Mset^c}^*}_{\max} \leq a_0 \lambda \right\} \cap \left\{ \norm{ \nabla_{(\Mset \cup \Sset)^c} \ell_n( \hat{\bvec}_0^{\oracle} )  }_{\max} < a_1 \lambda \right\}, \label{eqn:one step lla reduced estimator event}
	\end{equation}
	the LLA algorithm for the reduced model estimator \eqref{eqn:reduced estimator} initialized by $\hat{\bvec}_0^{\initial}$ finds $\hat{\bvec}_0^{\oracle}$ in one step. 
	Likewise, under the event
	\begin{equation}
		\Eset_{a1} = \left\{ \norm{\hat{\bvec}_{a, \Mset^c}^{\initial} - \bvec_{\Mset^c}^*}_{\max} \leq a_0 \lambda \right\} \cap \left\{ \norm{ \nabla_{(\Mset \cup \Sset)^c} \ell_n( \hat{\bvec}_a^{\oracle} )  }_{\max} < a_1 \lambda \right\}, \label{eqn:one step lla full estimator event}
	\end{equation}
	the LLA algorithm for the full model estimator \eqref{eqn:full estimator} initialized by $\hat{\bvec}_a^{\initial}$ finds $\hat{\bvec}_a^{\oracle}$ in one step.
\end{thm}
\noindent As an immediate corollary to Theorem \ref{thm:one step LLA}, we derive the following finite sample bounds for the probabilities that the LLA algorithms find their respective oracle estimators in one step:
\begin{cor}\label{cor:one step LLA prob}
	Suppose that we are trying to solve the partial penalized minimization problems \eqref{eqn:reduced estimator} and \eqref{eqn:full estimator} with a folded-concave penalty $p_{\lambda}(\cdot)$ satisfying \emph{\ref{assumption: penalty i}} - \emph{\ref{assumption: penalty iv}} and that conditions \emph{\ref{assumption:reduced oracle unique}} - \emph{\ref{assumption: min signal for betaS}} hold.
	Then with probability at least $1 - \delta_{00} - \delta_{01}$ the LLA algorithm for the reduced model estimator \eqref{eqn:reduced estimator} initialized by $\hat{\bvec}_0^{\initial}$ finds $\hat{\bvec}_0^{\oracle}$ in one step, where
	$$ \delta_{00} = \Pr \left( \norm{\hat{\bvec}_{0, \Mset^c}^{\initial} - \bvec_{\Mset^c}^*}_{\max} > a_0 \lambda \right) $$
	and
	$$ \delta_{01} = \Pr \left( \norm{ \nabla_{(\Mset \cup \Sset)^c} \ell_n( \hat{\bvec}_0^{\oracle} )  }_{\max} \geq a_1 \lambda \right).$$
	Likewise, with probability at least $1 - \delta_{a0} - \delta_{a1}$ the LLA algorithm for the full model estimator \eqref{eqn:full estimator} initialized by $\hat{\bvec}_a^{\initial}$ finds $\hat{\bvec}_a^{\oracle}$ in one step, where
	$$ \delta_{a0} = \Pr \left( \norm{\hat{\bvec}_{a, \Mset^c}^{\initial} - \bvec_{\Mset^c}^*}_{\max} > a_0 \lambda \right) $$
	and
	$$ \delta_{a1} = \Pr \left( \norm{ \nabla_{(\Mset \cup \Sset)^c} \ell_n( \hat{\bvec}_a^{\oracle} )  }_{\max} \geq a_1 \lambda \right).$$
\end{cor}

Put concisely, Corollary \ref{cor:one step LLA prob} states that $\Pr(\hat{\bvec}^{(1)}_0 = \hat{\bvec}^{\oracle}_0) \geq 1 - \delta_{00} - \delta_{01}$ and $\Pr(\hat{\bvec}^{(1)}_a = \hat{\bvec}^{\oracle}_a) \geq 1 - \delta_{a0} - \delta_{a1}$.
We see that $\delta_{00}$ and $\delta_{a0}$ measure how close the initial estimators are to the true parameter vector. 
Because $\nabla_{(\Mset \cup \Sset)^c} \ell_n( \bvec^* )$ is concentrated around $\bm 0$, we can view $\delta_{01}$ and $\delta_{a1}$ as measuring how close $\hat{\bvec}^{\oracle}_0$ and $\hat{\bvec}^{\oracle}_a$ are to $\bvec^*$ in terms of the score function.

Our next theorem identifies conditions under which the LLA algorithms, having found $\hat{\bvec}_0^{\oracle}$ and $\hat{\bvec}_a^{\oracle}$, return to them in their next steps.
\begin{thm}\label{thm:two step lla}
	Suppose that we are trying to solve the partial penalized minimization problems \eqref{eqn:reduced estimator} and \eqref{eqn:full estimator} with a folded-concave penalty $p_{\lambda}(\cdot)$ satisfying \emph{\ref{assumption: penalty i}} - \emph{\ref{assumption: penalty iv}} and that conditions \emph{\ref{assumption:reduced oracle unique}} and \emph{\ref{assumption:full oracle unique}} hold.
	Under the event
	\begin{equation}
		\Eset_{02} = \left\{ \norm{ \nabla_{(\Mset \cup \Sset)^c} \ell_n( \hat{\bvec}_0^{\oracle} )  }_{\max} < a_1 \lambda \right\} \cap \left\{  \norm{\hat{\bvec}_{0, \Sset}^{\oracle}}_{\min} > a\lambda \right\}, \label{eqn:two step lla reduced estimator event}
	\end{equation}
	if the LLA algorithm for the reduced model estimator \eqref{eqn:reduced estimator} has found $\hat{\bvec}_0^{\oracle}$, then it will find $\hat{\bvec}_0^{\oracle}$ again in the next step.
	Likewise, under the event
	\begin{equation}
		\Eset_{a2} =  \left\{ \norm{ \nabla_{(\Mset \cup \Sset)^c} \ell_n( \hat{\bvec}_a^{\oracle} )  }_{\max} < a_1 \lambda \right\} \cap \left\{  \norm{\hat{\bvec}_{a, \Sset}^{\oracle}}_{\min} > a\lambda \right\}, \label{eqn:two step lla full estimator event}
	\end{equation}
	if the LLA algorithm for the full model estimator \eqref{eqn:full estimator} has found $\hat{\bvec}_a^{\oracle}$, then it will find $\hat{\bvec}_a^{\oracle}$ again in the next step.
\end{thm}
\noindent Theorem \ref{thm:two step lla} tells us that if $\Eset_{02}$ and $\Eset_{a2}$ are satisfied, then once the LLA algorithms have found the oracle estimators they will return to them again in every subsequent step.
As such, if $\mathcal{E}_{01} \cap \mathcal{E}_{02}$ and $\mathcal{E}_{a1} \cap \mathcal{E}_{a2}$ hold, then the two-step LLA solutions are equal to the fully-converged LLA solutions.

Together Theorems \ref{thm:one step LLA} and \ref{thm:two step lla} establish that if $\mathcal{E}_{01} \cap \mathcal{E}_{02}$ and $\mathcal{E}_{a1} \cap \mathcal{E}_{a2}$ hold, then Algorithms \ref{algo:LLA for reduced estimator} and \ref{algo:LLA for full estimator} will find $\hat{\bvec}^{\oracle}_0$ and $\hat{\bvec}^{\oracle}_a$, respectively, in one step and converge to them in two steps.
By applying the union bound, we derive the following finite sample bound for the probability that the LLA algorithms converge to their respective oracle estimators in two steps:
\begin{cor}\label{cor:two step LLA prob}
	Suppose that we are trying to solve the partial penalized minimization problems \eqref{eqn:reduced estimator} and \eqref{eqn:full estimator} with a folded-concave penalty $p_{\lambda}(\cdot)$ satisfying \emph{\ref{assumption: penalty i}} - \emph{\ref{assumption: penalty iv}} and that conditions \emph{\ref{assumption:reduced oracle unique}} - \emph{\ref{assumption: min signal for betaS}} hold. Then with probability at least  $1 - \delta_{00} - \delta_{01} - \delta_{02}$ the LLA algorithm for the reduced model estimator \eqref{eqn:reduced estimator} initialized by $\hat{\bvec}_0^{\initial}$ converges to $\hat{\bvec}_0^{\oracle}$ in two steps, where
	$$ \delta_{02} = \Pr\left(\norm{\hat{\bvec}_{0, \Sset}^{\oracle}}_{\min} \leq a\lambda\right). $$
	Likewise, with probability at least  $1 - \delta_{a0} - \delta_{a1} - \delta_{a2}$ the LLA algorithm for the full model estimator \eqref{eqn:full estimator} initialized by $\hat{\bvec}_a^{\initial}$ converges to $\hat{\bvec}_a^{\oracle}$ in two steps, where
	$$ \delta_{a2} = \Pr\left(\norm{\hat{\bvec}_{a, \Sset}^{\oracle}}_{\min} \leq a\lambda\right). $$
\end{cor}

Corollary \ref{cor:two step LLA prob} provides that $\Pr(\hat{\bvec}^{(2)}_0 = \hat{\bvec}^{\oracle}_0) \geq 1 - \delta_{00} - \delta_{01} - \delta_{02}$ and $\Pr(\hat{\bvec}^{(2)}_a = \hat{\bvec}^{\oracle}_a) \geq 1 - \delta_{a0} - \delta_{a1} - \delta_{a2}$.
We see that $\delta_{02}$ and $\delta_{a2}$ effectively measure how close the oracle estimators are to $\bvec^*$, since $\delta_{i2} \leq \Pr\left(\norm{ \hat{\bvec}_{i, \Sset}^{\oracle} - \bvec^*_{\Sset} }_2 >\lambda\right)$ for $i = 0,a$ under condition \ref{assumption: min signal for betaS}.
If we can show that $\delta_{ij} \to 0$ for all $i = 0,a$ and $j = 0,1,2$, then we have $\hat{\bvec}^{(2)}_0 = \hat{\bvec}^{\oracle}_0$ and $\hat{\bvec}^{(2)}_a = \hat{\bvec}^{\oracle}_a$ with probability converging to $1$ as $n \to \infty$, meaning that the two-step LLA solutions, $\hat{\bvec}^{(2)}_0$ and $\hat{\bvec}^{(2)}_a$, possess the strong oracle property.

\subsection{LLA results for GLMs}
For the rest of the study, we narrow our focus to partial penalized GLMs.
Our aim is to apply Corollary \ref{cor:two step LLA prob} to show that the two-step LLA solutions to the partial penalized estimation problems for GLMs, \eqref{eqn:Shi reduced} and \eqref{eqn:Shi full}, are equal to the reduced and full oracle estimators with probability converging to $1$ as $n \to \infty$.

\subsubsection{Assumptions for GLMs}
Let $\yvec = (y_1, \ldots, y_n)' \in \reals^n$ denote the vector of response values and $\xmat = (\xvec_1, \ldots, \xvec_n)' \in \reals^{n \times p}$ denote the design matrix.
We treat $\xmat$ as fixed.
For the remainder of this study, our loss function is
\begin{equation}
	\ell_n(\bvec) = - \frac{1}{n} \left\{ \yvec'\xmat \bvec - \bm{1}_n' \bm{b}(\xmat \bvec)  \right\}, \label{eqn:glm loss}
\end{equation}
where for $\bm{v} \in \reals^n$, $\bm b(\bm v) = (b(v_1), \ldots, b(v_n))'$.
Define $\bm \mu (\bm v) = (b'(v_1), \ldots, b'(v_n))'$ and
$\bm \Sigma (\bm v) = \diag \{ (b''(v_1), \ldots, b''(v_n))'\}$.
Using this notation, we can express the gradient of $\ell_n(\bvec)$ as $\nabla \ell_n(\bvec) = - \frac{1}{n} \xmat'(\yvec - \bm \mu(\xmat \bvec))$
and the Hessian as $\nabla^2 \ell_n (\bvec) = \frac{1}{n} \xmat'\bm \Sigma (\xmat \bvec) \xmat$.

We examine the ultra-high dimensional setting, where $\log p = O(n^{\eta})$ for some $\eta \in (0, 1)$.
We assume the following about the true coefficient vector $\bvec^*$:
\begin{enumerate}[label = (A\arabic*) , start = 4 , leftmargin=2\parindent]
	\item $\bvec^*$ is sparse and satisfies $\Cmat \bvec_{\Mset}^* = \bm{t} + \bm{h}_n$, where $\bm{h}_n \to \bm{0}$; \label{assumption:sparsity} 
	$\lambda_{\max}\{ (\Cmat \Cmat')^{-1} \} = O(1)$; and \label{assumption:Cmat eigenvalue} 
	$\norm{\bm{h}_n}_2 = O(\sqrt{\min\{s+m-r,r\}/n})$. \label{assumption:hn order}
\end{enumerate}
In effect, \ref{assumption:sparsity} states that either the null hypothesis \eqref{eqn:H0} or a sequence of local alternatives holds for the true model.
We allow the penalty parameter $\lambda$ to vary with $n$ and assume the following about its rate
\begin{enumerate}[label=(A\arabic*), start=5, leftmargin=2\parindent]
	\item $\max\{\sqrt{s+m}, \sqrt{\log p}\}/\sqrt{n} = o(\lambda)$. \label{assumption:rate of lambda}
\end{enumerate}
Define $\mathcal{N}_0 = \{\bvec \in \reals^p: \norm{\bvec_{\Mset \cup \Sset} - \bvec^*_{\Mset \cup \Sset} }_2 \leq \sqrt{(s+m) \log n / n}, \bvec_{(\Mset \cup \Sset)^c} = \bm 0 \}$.
We assume
\begin{enumerate}[label=(A\arabic*), start=6, leftmargin=2\parindent]
	\item 
	$\inf_{\bvec \in \mathcal{N}_0} \lambda_{\min} \left\{ \xmat_{\Mset \cup \Sset}' \bm \Sigma(\xmat \bvec) \xmat_{\Mset \cup \Sset} \right\} \geq c n$ for some $c > 0$; \label{assumption:glm:min hessian}\\
	$\sup_{\bvec \in \mathcal{N}_0} \max_{1 \leq j \leq p} \lambda_{\max} \left\{ \xmat_{\Mset \cup \Sset}' \diag\{ |\xmat_j|\circ |\bm b'''(\xmat \bvec)| \}\xmat_{\Mset \cup \Sset} \right\} = O(n)$; \label{assumption:glm:third derivative bound}\\
	$\lambda_{\max} \left\{ \xmat_{\Mset \cup \Sset}' \bm \Sigma(\xmat \bvec^*) \xmat_{\Mset \cup \Sset} \right\} = O(n)$; and \label{assumption:glm:max hessian}
	$ \norm{\xmat_{(\Mset \cup \Sset)^c}' \bm \Sigma(\xmat \bvec^*) \xmat_{\Mset \cup \Sset}}_{\infty} = O(n)$
	\item 
	$\exists_{B_1, B_2}$ such that $\forall_n$, $\max_{1 \leq j \leq p} \norm{\xmat_j}_{2}^2 \leq B_1 n $ and
	$\max_{1 \leq j \leq p} \norm{\xmat_j}_{\infty} =  B_2 \sqrt{\frac{n}{\log p}}$  \label{assumption:glm:design column rates}
	\item
	$\exists_{M, v_0>0}$ such that $\forall_n$, \label{assumption:glm:expectation bound for chernoff}
	$$ \max_{1 \leq i \leq n} \E \left[ \exp\left(\frac{|y_i - b'(\xvec_i'\bvec^*)|}{M} \right) - 1 - \frac{|y_i - b'(\xvec_i'\bvec^*)|}{M}\right]M^2 \leq \frac{v_0}{2} .$$
\end{enumerate}

\cite{Shi2019} also assume conditions \ref{assumption:hn order} - \ref{assumption:glm:expectation bound for chernoff} in their theoretical study of partial penalized estimators. 
(Moreover, in their Appendix S4, \cite{Shi2019} demonstrate that, when $\xmat$ is random, these conditions hold with high probability in linear, logistic, and Poisson regression under some mild assumptions.)
In addition, \cite{Shi2019} also place a minimum signal strength assumption---similar to our condition \ref{assumption: min signal for betaS}---on $\bvec^*_{\Sset}$, assuming that $\lambda = o(\norm{ \bvec^*_{\Sset} }_{\min})$.

\subsubsection{Convergence of exception probabilities}
We aim to prove that the exception probabilities, $\delta_{ij}$ for $i=0,a, j = 0,1,2$, from Corollaries \ref{cor:one step LLA prob} and \ref{cor:two step LLA prob} converge to $0$ as $n \to \infty$ when $\ell_n(\cdot)$ is the GLM loss \eqref{eqn:glm loss}.
Define
$$\delta_{i1}^{\text{GLM}} = \Pr \left( \norm{ \nabla_{(\Mset \cup \Sset)^c} \ell_n( \hat{\bvec}_i^{\oracle} )  }_{\max} \geq a_1 \lambda \right)$$ and
$$\delta_{i2}^{\text{GLM}} = \Pr\left(\norm{\hat{\bvec}_{i, \Sset}^{\oracle}}_{\min} \leq a\lambda\right)$$ where the index $i = 0,a$ denotes either the reduced or full model.
These probabilities relate to regularity conditions for the oracle estimators and do not depend on the initial estimators used.
Our first result for GLMs identifies conditions under which $\delta_{i1}^{\text{GLM}} \to 0$ and $\delta_{i2}^{\text{GLM}} \to 0$ as $n \to \infty$ for $i = 0,a$:
\begin{thm}\label{thm:glm lla covergence prob}
	If \emph{\ref{assumption:reduced oracle unique}}--\emph{\ref{assumption:glm:expectation bound for chernoff}} hold
	and $(s+m) = o(n^{1/2})$, 
	then $\delta_{01}^{\text{GLM}} \to 0$, $\delta_{02}^{\text{GLM}} \to 0$,
	$\delta_{a1}^{\text{GLM}} \to 0$, and $\delta_{a2}^{\text{GLM}} \to 0$ as $n \to \infty$.
\end{thm}

All that remains is to select good initial estimators so that $\delta_{00}$ and $\delta_{a0}$ also converge to $0$ as $n \to \infty$. 
Theorems \ref{thm:one step LLA} and \ref{thm:two step lla} do not require the initial estimators for the LLA algorithms to be partially penalized or to satisfy the linear constraints in \eqref{eqn:reduced estimator}.
As such, we use the $\ell_1$-penalized estimator, given by
\begin{equation}
	\hat{\bvec}^{\lasso} = \argmin_{\bvec} \ell_n(\bvec) + \lambda_{\lasso} \norm{\bvec}_1 ,\label{eqn:glm:lasso initial estimator}
\end{equation}
as the initial estimator for both LLA algorithms.

In order to bound $\norm{\hat{\bvec}^{\lasso} - \bvec^*}_2$, we assume that the following restricted eigenvalue (RE) condition holds: 
\begin{enumerate}[label=(A\arabic*), start=9, leftmargin=2\parindent]
	\item  $\kappa = \min_{\bm u \in \mathcal{C}} \frac{\bm{u}' \nabla^2 \ell_n(\bvec^*) \bm{u}}{\bm u'\bm u} \in (0, \infty)$, where $\mathcal{C} = \left\{\bm u \neq 0 : \norm{\bm u_{\mathcal{A}^c}}_1 \leq 3 \norm{\bm u_{\mathcal{A}}}_1 \right\}$.	\label{assumption:glm:RE condition} 
\end{enumerate}
\cite{Bickel2009} and \cite{Fan2014} impose equivalent RE conditions to derive estimation bounds for $\ell_1$-penalized least squares and logistic regression, respectively.
Condition \ref{assumption:glm:RE condition} effectively generalizes their RE conditions for other GLMs.
In addition, we assume that $b(\cdot)$ in the density function \eqref{assumption:glm} satisfies
\begin{enumerate}[label=(A\arabic*), start=10, leftmargin=2\parindent]
	\item $\exists_{K>0}$ such that $|b'''(t)| \leq K b''(t)$ for all $t$. \label{assumption:glm:self concordant}
\end{enumerate}
It is straightforward to show that \ref{assumption:glm:self concordant} holds for linear, logistic, and Poisson regression (see Appendix \ref{sec:technical conditions}).
In fact, \cite{Bach2010} and \cite{Fan2014} use this property to establish estimation bounds for $\ell_1$-penalized logistic regression.

The following theorem provides an estimation bound for the $\ell_1$-penalized estimator \eqref{eqn:glm:lasso initial estimator}:
\begin{thm}\label{thm:glm initial estimator}
	Define $c = \norm{\xmat}_{\max}$ and $\tilde s = |\mathcal{A}|$.
	Suppose that \emph{\ref{assumption:glm:design column rates}}--\emph{\ref{assumption:glm:self concordant}} hold and that
	$\lambda_{\lasso} \leq \frac{\kappa}{20 K c \tilde s}$. 
	Then 
	\begin{equation}
		\norm{\hat{\bvec}^{\lasso} - \bvec^*}_2 \leq \frac{5 \tilde s^{1/2} \lambda_{\lasso} }{\kappa} \label{eqn:glm:lasso rate}
	\end{equation}
	with probability at least 
	$1 - 2p\exp\left\{-\frac{1}{8} \frac{n \lambda_{\lasso}^2}{B_1 v_0 + \frac{1}{2}B_2M\sqrt{\frac{n}{\log p}} \lambda_{\lasso}}\right\}$. 
	Therefore, if $\lambda_{\lasso} = A \sqrt{\frac{\log p}{n}}$, where $A$ satisfies $A^2 > 16(B_1  v_0 + \frac{1}{2} A B_2 M)$, then \eqref{eqn:glm:lasso rate} holds with probability at least $1 - \frac{2}{p}$. 
\end{thm}

\begin{remark}
	\cite{Bickel2009}, \cite{Negahban2012}, and \cite{Fan2014} establish similar estimation bounds for $\ell_1$-penalized linear regression and logistic regression. 
	Theorem \ref{thm:glm initial estimator} extends their results to cover all GLMs satisfying \ref{assumption:glm:expectation bound for chernoff}-\ref{assumption:glm:self concordant}.
\end{remark}

Using Theorems \ref{thm:glm lla covergence prob} and \ref{thm:glm initial estimator}, we can show that the exception probabilities from Corollaries \ref{cor:one step LLA prob} and \ref{cor:two step LLA prob} converge to $0$ as $n \to \infty$ for the LLA algorithms initialized by the $\ell_1$-penalized estimator \eqref{eqn:glm:lasso initial estimator}.
Theorem \ref{thm:glm lla covergence prob} guarantees that $\delta_{i1}^{\text{GLM}} \to 0$, and $\delta_{i2}^{\text{GLM}} \to 0$ as $n \to \infty$ for $i = 0,a$.
Define $$\delta_0^{\lasso} = \Pr \left( \norm{\hat{\bvec}_{\Mset^c}^{\lasso} - \bvec_{\Mset^c}^*}_{\max} > a_0 \lambda \right).$$
As an immediate consequence of Theorem \ref{thm:glm initial estimator}, $\delta_0^{\lasso} \to 0$ as $n \to \infty$ if $\lambda \geq \frac{5 \tilde{s}^{1/2} \lambda_{\lasso}}{a_0 \kappa}$. 
As such, if we use $\hat{\bvec}^{\lasso}$ as our initial estimator, $\lambda$ is sufficiently large, and the conditions of Theorems \ref{thm:glm lla covergence prob} and \ref{thm:glm initial estimator} are satisfied, then the LLA algorithms for the reduced and full model estimators will converge to $\hat{\bvec}_0^{\oracle}$ and $\hat{\bvec}_a^{\oracle}$ in two steps with probability converging to $1$ as $n\to \infty$.
We summarize this result for the LLA algorithms initialized by $\hat{\bvec}^{\lasso}$ with the following corollary:
\begin{cor}\label{cor:glm lla convergence prob with lasso}
	Suppose that \emph{\ref{assumption:reduced oracle unique}}--\emph{\ref{assumption:glm:self concordant}} hold
	and that $(s+m) = o(n^{1/2})$. If $\lambda \geq \frac{5 \tilde{s}^{1/2} \lambda_{\lasso}}{a_0 \kappa}$ and $\lambda_{\lasso}$ satisfies the conditions in Theorem \ref{thm:glm initial estimator}, 
	then the LLA algorithm for the reduced model estimator \eqref{eqn:Shi reduced} initialized by $\hat{\bvec}^{\lasso}$ converges to $\hat{\bvec}_0^{\oracle}$ in two steps with probability at least $1 - \delta_{0}^{\lasso} - \delta_{01}^{\text{GLM}} - \delta_{02}^{\text{GLM}}$, where $\delta_0^{\lasso} \to 0$, $\delta_{01}^{\text{GLM}} \to 0$, and $\delta_{02}^{\text{GLM}} \to 0$ as $n \to \infty$.
	
	Likewise, the LLA algorithm for the full model estimator \eqref{eqn:Shi full} initialized by $\hat{\bvec}^{\lasso}$ converges to $\hat{\bvec}_a^{\oracle}$ in two steps with probability at least $1 - \delta_{0}^{\lasso} - \delta_{a1}^{\text{GLM}} - \delta_{a2}^{\text{GLM}}$, where $\delta_0^{\lasso} \to 0$,  $\delta_{a1}^{\text{GLM}} \to 0$, and $\delta_{a2}^{\text{GLM}} \to 0$ as $n \to \infty$.
\end{cor}
Corollary \ref{cor:glm lla convergence prob with lasso} provides that if we use $\hat{\bvec}^{\lasso}$ as our initial estimator, then $\Pr( \hat{\bvec}^{(2)}_0 = \hat{\bvec}^{\oracle}_0 ) \to 1$ and $\Pr( \hat{\bvec}^{(2)}_a = \hat{\bvec}^{\oracle}_a ) \to 1$ as $n \to \infty$, meaning that $\hat{\bvec}^{(2)}_0$ and $\hat{\bvec}^{(2)}_a$---the two-step LLA solutions to the partial penalized estimation problems for GLMs---possess the strong oracle property with respect to $\hat{\bvec}^{\oracle}_0$ and $\hat{\bvec}^{\oracle}_a$.
This key result lays the groundwork for the next stage of our theoretical study.

\section{Asymptotic Results for Oracle and LLA Solutions} \label{sec:estimator asymptotics}
In this section, we examine the asymptotic behavior of the two-step LLA solutions for partial penalized GLMs. 
In particular, we aim to derive approximate distributions for the partial penalized Wald, score, and likelihood ratio tests evaluated at the two-step LLA solutions.

As consequence of the strong oracle property (Corollary \ref{cor:glm lla convergence prob with lasso}), one can show that the asymptotic properties of $\hat{\bvec}^{\oracle}_0$ and $\hat{\bvec}^{\oracle}_a$ also hold for $\hat{\bvec}_0^{(2)}$ and $\hat{\bvec}_a^{(2)}$. 
Our theoretical strategy, therefore, is to prove limiting results for $\hat{\bvec}^{\oracle}_0$ and $\hat{\bvec}^{\oracle}_a$, then extend them to $\hat{\bvec}_0^{(2)}$ and $\hat{\bvec}_a^{(2)}$ using the strong oracle property.

\subsection{Asymptotic results for the oracle estimators}
If $s$, $m$, and $r$ were fixed, then we could simply apply classical results for GLMs to establish the limiting behavior of $\hat{\bvec}^{\oracle}_0$ and $\hat{\bvec}^{\oracle}_a$.
However, we allow $s$, $m$, and $r$ to diverge with $n$ and, therefore, must prove new results for this regime.

We begin by deriving rates of convergence and limiting expressions for $\hat{\bvec}^{\oracle}_0$ and $\hat{\bvec}^{\oracle}_a$.
Define $\bm{K}_n = \nabla_{\Mset \cup \Sset}^2 \ell_n(\bvec^*)$ and $\bm{\Psi} =  \begin{bsmallmatrix} \Cmat & \bm 0 \end{bsmallmatrix} \bm K_n^{-1} \begin{bsmallmatrix} \Cmat' \\ \bm 0'\end{bsmallmatrix}$.
\begin{lemma}\label{lemma:oracle asymptotics}
	Suppose that \emph{\ref{assumption:reduced oracle unique}}, \emph{\ref{assumption:full oracle unique}}, \emph{\ref{assumption:hn order}} and \emph{\ref{assumption:glm:min hessian}} - \emph{\ref{assumption:glm:expectation bound for chernoff}} hold.
	If $s+m = o(n^{1/2})$, then 
	\begin{equation}
		\norm{ \hat{\bvec}^{\oracle}_{0, \Mset \cup \Sset}  - \bvec^*_{\Mset \cup \Sset} }_2 = O_p\left( \sqrt{ \frac{s+m-r}{n} }\right) \label{eqn:reduced oracle consistency}
	\end{equation}
	and
	\begin{equation}
		\norm{ \hat{\bvec}^{\oracle}_{a, \Mset \cup \Sset}  - \bvec^*_{\Mset \cup \Sset} }_2 = O_p\left( \sqrt{ \frac{s+m}{n} }\right). \label{eqn:full oracle consistency}
	\end{equation}
	Furthermore, if $s+m = o(n^{1/3})$, then
	\begin{align}
		\sqrt{n}(\hat{\bvec}^{\oracle}_{0, \Mset \cup \Sset} - \bvec_{\Mset \cup \Sset}^* )
		= & \sqrt{n} \bm{K}_n^{-1/2} (\bm P_n - \bm I) \bm {K}_n^{-1/2} \nabla_{\Mset \cup \Sset} \ell_n(\bvec^*) \notag \\
		& -  \sqrt{n} \bm K_n^{-1/2} \bm P_n \bm K_n^{1/2} \begin{bmatrix} \Cmat'(\Cmat \Cmat')^{-1} \bm h_n \\ \bm 0 \end{bmatrix}
		+ o_p(1), \label{eqn:reduced oracle limiting expression}
	\end{align}
	where $\bm P_n = \bm K_n^{-1/2} \begin{bsmallmatrix} \Cmat' \\ \bm 0' \end{bsmallmatrix} \bm \Psi^{-1} \begin{bsmallmatrix} \Cmat & \bm 0 \end{bsmallmatrix} \bm K_n^{-1/2}$,
	and
	\begin{equation}
		\sqrt{n}(\hat{\bvec}^{\oracle}_{a, \Mset \cup \Sset} - \bvec_{\Mset \cup \Sset}^* ) = -\sqrt{n} \bm K_n^{-1} \nabla_{\Mset \cup \Sset} \ell_n(\bvec^*) + o_p(1). \label{eqn:full oracle limiting expression}
	\end{equation}
\end{lemma}
Using the properties established in Lemma \ref{lemma:oracle asymptotics}, we derive approximate distributions for the Wald, score, and likelihood ratio test statistics evaluated at $\hat{\bvec}^{\oracle}_0$ and $\hat{\bvec}^{\oracle}_a$:
\begin{lemma} \label{lemma:oracle test statistic distribution}
	Suppose that \emph{\ref{assumption:reduced oracle unique}}, \emph{\ref{assumption:full oracle unique}}, \emph{\ref{assumption:hn order}} and \emph{\ref{assumption:glm:min hessian}} - \emph{\ref{assumption:glm:expectation bound for chernoff}} hold.
	In addition, assume
	\begin{enumerate}[label=(A\arabic*), start=11, leftmargin=2\parindent]
		\item $\frac{r^{1/4}}{n^{3/2}} \sum_{i=1}^n \left( \xmat_{i, \Mset \cup \Sset} \bm K_n^{-1} \xmat_{i, \Mset \cup \Sset}' \right)^{3/2} \to 0$ \label{assumption:lyapunov condition}.
	\end{enumerate}
	Under these conditions, if $s+m = o(n^{1/3})$ and $|\hat{\phi} - \phi^*| = o_p(1)$,	
	then for $T = T_W(\hat{\bvec}_a^{\oracle})$, $T_S(\hat{\bvec}_0^{\oracle})$, and $T_L(\hat{\bvec}_a^{\oracle},\hat{\bvec}_0^{\oracle})$,  
	\begin{equation}
		\sup_x \left| \Pr(T \leq x) - \Pr(\chi^2(r, \nu_n) \leq x) \right| \to 0 \label{eqn:oracle test statistic convergence}
	\end{equation}
	as $n \to \infty$, where $\chi^2(r, \nu_n)$ is a noncentral chi-square random variable with $r$ degrees of freedom and noncentrality parameter
	$ \nu_n = n \bm{h}_n ' {\bm{\Psi}}^{-1}\bm{h}_n/\phi^*$.
\end{lemma}
\begin{remark}
	Since we allow $r$ to diverge with $n$, we cannot say that the Wald, score, and likelihood ratio test statistics ``converge in distribution'' to $\chi^2(r, \nu_n)$ random variables in the usual sense--- 
	because the target distribution is not fixed, the notion of convergence in distribution is not well-defined.
	Instead, \eqref{eqn:oracle test statistic convergence} establishes that for $T = T_W(\hat{\bvec}_a^{\oracle})$, $T_S(\hat{\bvec}_0^{\oracle})$, and $T_L(\hat{\bvec}_a^{\oracle},\hat{\bvec}_0^{\oracle})$, the difference between $\Pr(T \leq x)$ and $\Pr(\chi^2(r, \nu_n) \leq x)$ vanishes as $n \to \infty$.
	As such, the distribution of $\chi^2(r, \nu_n)$ gives us a good approximation of the distribution of $T$ for large $n$.
	To reflect this distinction between \eqref{eqn:oracle test statistic convergence} and the usual convergence in distribution, we say that the distribution of $T$ is ``approximately'' $\chi^2(r, \nu_n)$ for large $n$.
\end{remark}

\begin{remark}
	In our proof, we apply Theorem 1 of \cite{Bentkus2005} to derive a Lyapunov-type upper bound for $\sup_x \left| \Pr(T \leq x) - \Pr(\chi^2(r, \nu_n) \leq x) \right|$ which depends on $r$.
	Condition \ref{assumption:lyapunov condition} controls the growth rate of $r$ and guarantees that the upper bound converges to $0$ as $n \to \infty$.
	\cite{Shi2019} also assume \ref{assumption:lyapunov condition} to derive approximate distributions for the partial penalized test statistics.
\end{remark}

\subsection{Asymptotic results for the LLA solutions}
By leveraging the strong oracle property, we can use our asymptotic results for the oracle estimators to establish equivalent guarantees for the two-step LLA solutions.
By combining the strong oracle property and Lemma \ref{lemma:oracle asymptotics}, we derive convergence rates and limiting expressions for $\hat{\bvec}^{(2)}_0$ and $\hat{\bvec}^{(2)}_a$, and show that they are selection consistent:
\begin{thm}\label{thm:two step LLA asymptotics}
	Suppose that \emph{\ref{assumption:reduced oracle unique}}--\emph{\ref{assumption:glm:self concordant}} hold, 
	that $\hat{\bvec}^{\lasso}$ is used as the initial estimator for the LLA algorithms for \eqref{eqn:Shi reduced} and \eqref{eqn:Shi full},
	and that $\lambda_{\lasso}$ and $\lambda$ satisfy the conditions outlined in Corollary \ref{cor:glm lla convergence prob with lasso}.
	Under these conditions, if $s+m = o(n^{1/2})$, then 
	\begin{equation}
		\norm{ \hat{\bvec}^{(2)}_{0, \Mset \cup \Sset}  - \bvec^*_{\Mset \cup \Sset} }_2 = O_p\left( \sqrt{ \frac{s+m-r}{n} }\right) \label{eqn:reduced lla consistency}
	\end{equation}
	and
	\begin{equation}
		\norm{ \hat{\bvec}^{(2)}_{a, \Mset \cup \Sset}  - \bvec^*_{\Mset \cup \Sset} }_2 = O_p\left( \sqrt{ \frac{s+m}{n} }\right) \label{eqn:full lla consistency}
	\end{equation}
	and $\hat{\bvec}^{(2)}_{0, (\Mset \cup \Sset)^c } = \bm 0$ and  $\hat{\bvec}^{(2)}_{a, (\Mset \cup \Sset)^c } = \bm 0$ with probability converging to $1$ as $n \to \infty$.
	Furthermore, if $s+m = o(n^{1/3})$, then
	\begin{align}
		\sqrt{n}(\hat{\bvec}^{(2)}_{0, \Mset \cup \Sset} - \bvec_{\Mset \cup \Sset}^* )
		= & \sqrt{n} \bm{K}_n^{-1/2} (\bm P_n - \bm I) \bm {K}_n^{-1/2} \nabla_{\Mset \cup \Sset} \ell_n(\bvec^*) \notag \\
		& -  \sqrt{n} \bm K_n^{-1/2} \bm P_n \bm K_n^{1/2} \begin{bmatrix} \Cmat'(\Cmat \Cmat')^{-1} \bm h_n \\ \bm 0 \end{bmatrix}
		+ o_p(1), \label{eqn:reduced lla limiting expression}
	\end{align}
	where $\bm P_n = \bm K_n^{-1/2} \begin{bsmallmatrix} \Cmat' \\ \bm 0' \end{bsmallmatrix} \bm \Psi^{-1} \begin{bsmallmatrix} \Cmat & \bm 0 \end{bsmallmatrix} \bm K_n^{-1/2}$,
	and
	\begin{equation}
		\sqrt{n}(\hat{\bvec}^{(2)}_{a, \Mset \cup \Sset} - \bvec_{\Mset \cup \Sset}^* ) = -\sqrt{n} \bm K_n^{-1} \nabla_{\Mset \cup \Sset} \ell_n(\bvec^*) + o_p(1). \label{eqn:full lla limiting expression}
	\end{equation}
\end{thm}
 
Likewise, by combining the strong oracle property and Lemma \ref{lemma:oracle test statistic distribution}, we derive approximate large sample distributions for the partial penalized test statistics evaluated at $\hat{\bvec}^{(2)}_0$ and $\hat{\bvec}^{(2)}_a$---our main result for this paper:
\begin{thm} \label{thm:lla test statistic distribution}
	Suppose that \emph{\ref{assumption:reduced oracle unique}}--\emph{\ref{assumption:lyapunov condition}} hold, 
	that $\hat{\bvec}^{\lasso}$ is used as the initial estimator for the LLA algorithms for \eqref{eqn:Shi reduced} and \eqref{eqn:Shi full},
	and that $\lambda_{\lasso}$ and $\lambda$ satisfy the conditions outlined in Corollary \ref{cor:glm lla convergence prob with lasso}.
	Under these conditions, if $s+m = o(n^{1/3})$ and $|\hat{\phi} - \phi^*| = o_p(1)$,	
	then for $T = T_W(\hat{\bvec}_a^{(2)})$, $T_S(\hat{\bvec}_0^{(2)})$, and $T_L(\hat{\bvec}_a^{(2)},\hat{\bvec}_0^{(2)})$,  
	\begin{equation}
		\sup_x \left| \Pr(T \leq x) - \Pr(\chi^2(r, \nu_n) \leq x) \right| \to 0 \label{eqn:thm test statistic convergence}
	\end{equation}
	as $n \to \infty$, where $\chi^2(r, \nu_n)$ is a noncentral chi-square random variable with $r$ degrees of freedom and noncentrality parameter
	$ \nu_n = n \bm{h}_n ' {\bm{\Psi}}^{-1}\bm{h}_n/\phi^*$.
\end{thm}
Theorem \ref{thm:lla test statistic distribution} provides guarantees for the partial penalized test statistics evaluated at specific computed solutions to \eqref{eqn:Shi reduced} and \eqref{eqn:Shi full}---namely, the two-step LLA solutions.
Critically, these solutions can be directly obtained using Algorithms \ref{algo:LLA for reduced estimator} and \ref{algo:LLA for full estimator}.
As such, Theorem \ref{thm:lla test statistic distribution} closes the previous gap between theory and computation for partial penalized tests of GLMs.

Using the $\chi^2(r, \nu_n)$ approximation \eqref{eqn:thm test statistic convergence}, we can characterize the asymptotic behavior of the partial penalized tests conducted using the two-step LLA solutions.
We see that under the null \eqref{eqn:H0}, $\bm h_n = \bm 0$ and, by extension, $\nu_n = 0$.
As such, \eqref{eqn:thm test statistic convergence} implies that the test statistics are approximately $\chi^2(r)$ random variables under the null and therefore achieve their nominal size $\alpha$ asymptotically.
In addition, we see that when $\bm h_n \neq \bm 0$, the partial penalized tests have approximate power $\Pr(\chi^2(r, \nu_n) > \chi^2_{\alpha}(r))$.
Lastly, \eqref{eqn:thm test statistic convergence} implies 
$$\sup_x \left| \Pr(T_1 \leq x) - \Pr(T_2 \leq x) \right| \to 0$$
for $T_1, T_2 \in \{T_W(\hat{\bvec}_a^{(2)}), T_S(\hat{\bvec}_0^{(2)}), T_L(\hat{\bvec}_a^{(2)},\hat{\bvec}_0^{(2)})\}$, meaning the partial penalized test statistics are approximately equivalent for large $n$.
We summarize these findings below:
\begin{cor}\label{cor:test statistic asymptotics}
	Suppose that \emph{\ref{assumption:reduced oracle unique}}--\emph{\ref{assumption:lyapunov condition}} hold, 
	that $\hat{\bvec}^{\lasso}$ is used as the initial estimator for the LLA algorithms for \eqref{eqn:Shi reduced} and \eqref{eqn:Shi full},
	that $\lambda_{\lasso}$ and $\lambda$ satisfy the conditions outlined in Corollary \ref{cor:glm lla convergence prob with lasso},
	that $s+m = o(n^{1/3})$, and 
	that $|\hat{\phi} - \phi^*| = o_p(1)$.
	Then the partial penalized Wald, score, and likelihood ratio test statistics evaluated at the two-step LLA solutions to \eqref{eqn:Shi reduced} and \eqref{eqn:Shi full} satisfy the following:
	\begin{enumerate}[label = (\alph*), leftmargin=2\parindent]
		\item Under $\text{H}_0: \Cmat \bvec_{\Mset}^* = \bm t$, for any significance level $0 < \alpha < 1$, 
		$$ \lim_{n \to \infty} \Pr(T > \chi^2_{\alpha}(r))  =\alpha $$
		for $T = T_W(\hat{\bvec}_a^{(2)})$, $T_S(\hat{\bvec}_0^{(2)})$, and $T_L(\hat{\bvec}_a^{(2)},\hat{\bvec}_0^{(2)})$, where $\chi^2_{\alpha}(r)$ denotes the upper-$\alpha$ quantile of a central chi-square distribution with $r$ degrees of freedom.
		\item Under the alternative $\Cmat \bvec_{\Mset}^* = \bm{t} + \bm{h}_n$, where $\norm{\bm{h}_n}_2 = O(\sqrt{\min\{s+m-r,r\}/n})$, for any significance level $0 < \alpha < 1$,
		$$ \lim_{n \to \infty} \left| \Pr(T > \chi^2_{\alpha}(r)) - \Pr(\chi^2(r, \nu_n) > \chi^2_{\alpha}(r)) \right| = 0 $$
		for $T = T_W(\hat{\bvec}_a^{(2)})$, $T_S(\hat{\bvec}_0^{(2)})$, and $T_L(\hat{\bvec}_a^{(2)},\hat{\bvec}_0^{(2)})$, where $ \nu_n =n \bm{h}_n ' {\bm{\Psi}}^{-1}\bm{h}_n/\phi^*$.
		\item For $T_1, T_2 \in \{T_W(\hat{\bvec}_a^{(2)}), T_S(\hat{\bvec}_0^{(2)}), T_L(\hat{\bvec}_a^{(2)},\hat{\bvec}_0^{(2)})\}$, 
		$$\sup_x \left| \Pr(T_1 \leq x) - \Pr(T_2 \leq x) \right| \to 0 \text{ \; as \; } n \to \infty .$$
	\end{enumerate} 
\end{cor}

\section{Numerical Studies} \label{sec:sims}

Corollary \ref{cor:glm lla convergence prob with lasso} establishes that the two-step LLA solutions are equal to the oracle estimators with probability converging to $1$ as $n \to \infty$.
In these simulations, we examine whether the tests using the two-step LLA solutions and those using the oracle estimators provide comparable results when $n$ is relatively small.
In addition, we investigate whether the finite-sample performance of the partial penalized tests evaluated at the two-step LLA solutions aligns with the asymptotic guarantees given in Theorem \ref{thm:lla test statistic distribution} and Corollary \ref{cor:test statistic asymptotics}.
We use the SCAD penalty throughout these simulations and set $a = 3.7$, as suggested by \cite{Fan2001}.

\subsection{Linear model}
For our first case, we generate the response from a linear model $y_i = \xvec_i' \bvec^* + \varepsilon_i$ for $i = 1, \ldots, n$, where $\varepsilon_i \sim N(0, 1)$ and $\xvec_i \sim N(\bm{0}, \bm{\Sigma})$ is a $p$-dimensional vector of predictors with $\Sigma_{jk} = 0.5^{|j-k|}$ for $j,k = 1,\ldots, p$.
We set $\bvec^* = (2, -2 -h_1, 0, \ldots)'$ and vary $h_1$ and $p$ to create different simulation settings.
In particular, we run simulations with every combination of $p \in \{50, 200\}$ and $h_1 \in \{0, 0.1, 0.2, 0.4\}$.
We generate $n = 100$ observations per replication in every simulation setting.

We test the following hypotheses using the partial penalized Wald, score, and likelihood ratio tests evaluated at the two-step LLA solutions
as well as the Wald, score, and likelihood ratio tests evaluated at the oracle estimators:
$\text{H}_0^{(1)}: \beta_1^* + \beta_2^* = 0$, $\text{H}_0^{(2)}: \beta_2^* = -2$, and $\text{H}_0^{(3)}: \beta_1^* + \beta_2^* + \beta_3^* + \beta_4^* = 0$.
When $h_1 = 0$, all three null hypotheses are true and we expect the tests to achieve rejection rates near their nominal size.
As $h_1$ increases, the true alternative moves farther from the null and we expect the rejection rates of the tests to increase.

\spacingset{1}
\begin{table}[ht!]
	\centering
	\caption{Linear Model Testing Rejection Rates (\%)} \label{tbl:linear test sims ar1}
	\resizebox{\columnwidth}{!}
	{
	\begin{tabular}{rrllllll}
		\toprule
		\multicolumn{1}{c}{ } & \multicolumn{1}{c}{ } & \multicolumn{3}{c}{LLA} & \multicolumn{3}{c}{Oracle} \\
		\cmidrule(l{3pt}r{3pt}){3-5} \cmidrule(l{3pt}r{3pt}){6-8}
		$p$ & $h_1$ & LRT & Wald & Score & LRT & Wald & Score\\
		\midrule
		\addlinespace[0.3em]
		\multirow{15}{*}{$50$} & & \multicolumn{6}{c}{ $\text{H}_0^{(1)}$ }\\
		\cline{3-8} & 0.0 & 4.9 (0.68) & 4.9 (0.68) & 4.9 (0.68) & 4.9 (0.68) & 4.9 (0.68) & 4.9 (0.68)\\
		 & 0.1 & 15.7 (1.15) & 15.2 (1.14) & 14.4 (1.11) & 15.7 (1.15) & 15.2 (1.14) & 14.4 (1.11)\\
		 & 0.2 & 40.5 (1.55) & 40 (1.55) & 38.5 (1.54) & 40.5 (1.55) & 40 (1.55) & 38.3 (1.54)\\
		 & 0.4 & 92.4 (0.84) & 91.9 (0.86) & 91.3 (0.89) & 92.4 (0.84) & 91.9 (0.86) & 91.3 (0.89)\\
		\addlinespace[0.3em]
		 & & \multicolumn{6}{c}{ $\text{H}_0^{(2)}$ }\\
		\cline{3-8} & 0.0 & 4.4 (0.65) & 4.4 (0.65) & 3.6 (0.59) & 4.4 (0.65) & 4.4 (0.65) & 3.5 (0.58)\\
		 & 0.1 & 13.8 (1.09) & 13.6 (1.08) & 13 (1.06) & 13.8 (1.09) & 13.6 (1.08) & 12.9 (1.06)\\
		 & 0.2 & 38 (1.53) & 37.5 (1.53) & 36.6 (1.52) & 38.1 (1.54) & 37.5 (1.53) & 36.6 (1.52)\\
		 & 0.4 & 91 (0.9) & 91.2 (0.9) & 90.4 (0.93) & 91.4 (0.89) & 91.2 (0.9) & 90.4 (0.93)\\
		\addlinespace[0.3em]
		 & & \multicolumn{6}{c}{ $\text{H}_0^{(3)}$ }\\
		\cline{3-8} & 0.0 & 5.8 (0.74) & 5.6 (0.73) & 5.3 (0.71) & 5.8 (0.74) & 5.6 (0.73) & 5.2 (0.7)\\
		 & 0.1 & 12 (1.03) & 10.9 (0.99) & 10.4 (0.97) & 11.6 (1.01) & 10.9 (0.99) & 10.2 (0.96)\\
		 & 0.2 & 29.4 (1.44) & 28.6 (1.43) & 27.3 (1.41) & 29.5 (1.44) & 28.6 (1.43) & 27.3 (1.41)\\
		 & 0.4 & 80.7 (1.25) & 79.8 (1.27) & 78.8 (1.29) & 80.7 (1.25) & 79.8 (1.27) & 78.8 (1.29)\\
		\midrule
		\addlinespace[0.3em]
		\multirow{15}{*}{$200$} & & \multicolumn{6}{c}{ $\text{H}_0^{(1)}$ }\\
		\cline{3-8} & 0.0 & 4.2 (0.63) & 3.9 (0.61) & 3.8 (0.6) & 4.2 (0.63) & 3.9 (0.61) & 3.8 (0.6)\\
		 & 0.1 & 15.4 (1.14) & 14.7 (1.12) & 14 (1.1) & 15.3 (1.14) & 14.7 (1.12) & 14.1 (1.1)\\
		 & 0.2 & 42.9 (1.57) & 42.7 (1.56) & 41 (1.56) & 42.9 (1.57) & 42.7 (1.56) & 41 (1.56)\\
		 & 0.4 & 92.6 (0.83) & 92.4 (0.84) & 91.3 (0.89) & 92.6 (0.83) & 92.4 (0.84) & 91.3 (0.89)\\
		\addlinespace[0.3em]
		 & & \multicolumn{6}{c}{ $\text{H}_0^{(2)}$ }\\
		\cline{3-8} & 0.0 & 5.9 (0.75) & 5.5 (0.72) & 5.1 (0.7) & 5.8 (0.74) & 5.5 (0.72) & 5.1 (0.7)\\
		 & 0.1 & 16.3 (1.17) & 16.1 (1.16) & 15.2 (1.14) & 16.4 (1.17) & 16.1 (1.16) & 15.2 (1.14)\\
		 & 0.2 & 41.2 (1.56) & 41.1 (1.56) & 39.5 (1.55) & 41.2 (1.56) & 41.1 (1.56) & 39.5 (1.55)\\
		 & 0.4 & 92 (0.86) & 91.9 (0.86) & 91.5 (0.88) & 92.1 (0.85) & 91.9 (0.86) & 91.5 (0.88)\\
		\addlinespace[0.3em]
		 & & \multicolumn{6}{c}{ $\text{H}_0^{(3)}$ }\\
		\cline{3-8} & 0.0 & 5.5 (0.72) & 5.2 (0.7) & 4.8 (0.68) & 5.5 (0.72) & 5.2 (0.7) & 4.8 (0.68)\\
		 & 0.1 & 11.6 (1.01) & 11.2 (1) & 10.6 (0.97) & 11.5 (1.01) & 11.2 (1) & 10.6 (0.97)\\
		 & 0.2 & 28.3 (1.42) & 27.6 (1.41) & 26.8 (1.4) & 28.4 (1.43) & 27.6 (1.41) & 26.7 (1.4)\\
		 & 0.4 & 78.7 (1.29) & 78.1 (1.31) & 77.2 (1.33) & 78.7 (1.29) & 78.1 (1.31) & 77.3 (1.32)\\
		\bottomrule
	\end{tabular}
}
\end{table}
\spacingset{1.45}

The linear model simulation results are given in Table \ref{tbl:linear test sims ar1}, which shows the rejection rates for the tests with significance level $\alpha = 0.05$ over $1000$ simulation replications.
We see that the rejection rates for the partial penalized tests using the two-step LLA solutions are close to the rejection rates for the tests using the oracle estimators in every setting.
In addition, the partial penalized tests achieve rejection rates near their nominal size of $\alpha = 0.05$ when $\text{H}_0$ is true and their rejection rates increase as $h_1$ increases, suggesting that the chi-square approximation for the test statistics from \eqref{eqn:thm test statistic convergence} is appropriate.
We also see that the partial penalized Wald, score, and likelihood ratio tests attain similar rejection rates to each other within each simulation setting, echoing our theoretical finding that the three tests will agree asymptotically.

\subsection{Logistic regression model}
For our second case, we generate the response from a Bernoulli distribution,
with success probability $\Pr(y_i = 1|\xvec_i) = \frac{\exp(\xvec_i' \bvec^*)}{1 + \exp(\xvec_i' \bvec^*)}$ for $i=1, \ldots, n$,
where $\xvec_i \sim N(\bm{0}, \bm{\Sigma})$ with $\Sigma_{jk} = 0.5^{|j-k|}$ for $j,k = 1,\ldots, p$. 
We set $\bvec^* = (2, -2 -h_1, 0, \ldots)'$ and vary $h_1 \in \{0, 0.2, 0.4, 0.8\}$ and $p \in \{50, 200\}$ to create different simulation settings.
We generate $n = 300$ observations per replication in every setting.
As in the linear model simulations, we test $\text{H}_0^{(1)}$, $\text{H}_0^{(2)}$, and $\text{H}_0^{(3)}$ and compare tests using the two-step LLA solutions and the oracle estimators.

\spacingset{1}
\begin{table}[ht!]
	\centering
	\caption{Logistic Regression Testing Rejection Rates (\%)} \label{tbl:logistic test sims ar1}
	\resizebox{\columnwidth}{!}
	{
	\begin{tabular}{rrllllll}
		\toprule
		\multicolumn{1}{c}{ } & \multicolumn{1}{c}{ } & \multicolumn{3}{c}{LLA} & \multicolumn{3}{c}{Oracle} \\
		\cmidrule(l{3pt}r{3pt}){3-5} \cmidrule(l{3pt}r{3pt}){6-8}
		$p$ & $h_1$ & LRT & Wald & Score & LRT & Wald & Score\\
		\midrule
		\addlinespace[0.3em]
		\multirow{15}{*}{$50$} & & \multicolumn{6}{c}{ $\text{H}_0^{(1)}$ }\\
		\cline{3-8} & 0.0 & 6.67 (1.02) & 6 (0.97) & 6.5 (1.01) & 6.67 (1.02) & 5.83 (0.96) & 6.5 (1.01)\\
		 & 0.2 & 19.33 (1.61) & 18.5 (1.59) & 18.83 (1.6) & 19.17 (1.61) & 18.33 (1.58) & 18.83 (1.6)\\
		 & 0.4 & 60.33 (2) & 60.33 (2) & 60.17 (2) & 61 (1.99) & 60.33 (2) & 60.67 (1.99)\\
		 & 0.8 & 98.5 (0.5) & 98.83 (0.44) & 98.67 (0.47) & 99.17 (0.37) & 98.83 (0.44) & 99 (0.41)\\
		\addlinespace[0.3em]
		 & & \multicolumn{6}{c}{ $\text{H}_0^{(2)}$ }\\
		\cline{3-8} & 0.0 & 4.83 (0.88) & 3.67 (0.77) & 3.67 (0.77) & 4.67 (0.86) & 3.67 (0.77) & 3.67 (0.77)\\
		 & 0.2 & 12.5 (1.35) & 9.83 (1.22) & 10 (1.22) & 12.17 (1.33) & 9.83 (1.22) & 9.83 (1.22)\\
		 & 0.4 & 31.83 (1.9) & 25.5 (1.78) & 26.5 (1.8) & 31.33 (1.89) & 25.33 (1.78) & 26.33 (1.8)\\
		 & 0.8 & 79.83 (1.64) & 76.83 (1.72) & 76.83 (1.72) & 79.83 (1.64) & 76.83 (1.72) & 76.83 (1.72)\\
		\addlinespace[0.3em]
		 & & \multicolumn{6}{c}{ $\text{H}_0^{(3)}$ }\\
		\cline{3-8} & 0.0 & 3.67 (0.77) & 3.67 (0.77) & 3.67 (0.77) & 3.67 (0.77) & 3.67 (0.77) & 3.67 (0.77)\\
		 & 0.2 & 16.33 (1.51) & 16 (1.5) & 16.17 (1.5) & 16.33 (1.51) & 16 (1.5) & 16.17 (1.5)\\
		 & 0.4 & 43.83 (2.03) & 42.67 (2.02) & 43.33 (2.02) & 43.83 (2.03) & 42.67 (2.02) & 43.33 (2.02)\\
		 & 0.8 & 92.5 (1.08) & 93.5 (1.01) & 93.67 (0.99) & 93.67 (0.99) & 93.5 (1.01) & 93.67 (0.99)\\
		\midrule
		\addlinespace[0.3em]
		\multirow{15}{*}{$200$} & & \multicolumn{6}{c}{ $\text{H}_0^{(1)}$ }\\
		\cline{3-8} & 0.0 & 5.17 (0.9) & 5.17 (0.9) & 5.17 (0.9) & 5.17 (0.9) & 5.17 (0.9) & 5.17 (0.9)\\
		 & 0.2 & 21.33 (1.67) & 21.17 (1.67) & 21.33 (1.67) & 21.33 (1.67) & 21.17 (1.67) & 21.33 (1.67)\\
		 & 0.4 & 65.17 (1.95) & 64 (1.96) & 64.67 (1.95) & 65.17 (1.95) & 64.17 (1.96) & 64.83 (1.95)\\
		 & 0.8 & 97.83 (0.59) & 98.33 (0.52) & 98.33 (0.52) & 98.33 (0.52) & 98.33 (0.52) & 98.33 (0.52)\\
		\addlinespace[0.3em]
		 & & \multicolumn{6}{c}{ $\text{H}_0^{(2)}$ }\\
		\cline{3-8} & 0.0 & 4.83 (0.88) & 5 (0.89) & 5 (0.89) & 4.33 (0.83) & 4.83 (0.88) & 5 (0.89)\\
		 & 0.2 & 13.83 (1.41) & 10.67 (1.26) & 10.83 (1.27) & 13.83 (1.41) & 10.67 (1.26) & 10.83 (1.27)\\
		 & 0.4 & 36.83 (1.97) & 31.83 (1.9) & 32.17 (1.91) & 36.67 (1.97) & 31.67 (1.9) & 32.17 (1.91)\\
		 & 0.8 & 83 (1.53) & 79.17 (1.66) & 79 (1.66) & 82.67 (1.55) & 78.67 (1.67) & 78.83 (1.67)\\
		\addlinespace[0.3em]
		 & & \multicolumn{6}{c}{ $\text{H}_0^{(3)}$ }\\
		\cline{3-8} & 0.0 & 6.33 (0.99) & 6.5 (1.01) & 6.5 (1.01) & 6.5 (1.01) & 6.5 (1.01) & 6.5 (1.01)\\
		 & 0.2 & 16.5 (1.52) & 16.17 (1.5) & 16.5 (1.52) & 16.5 (1.52) & 16.17 (1.5) & 16.5 (1.52)\\
		 & 0.4 & 46.5 (2.04) & 46 (2.03) & 46.33 (2.04) & 46.67 (2.04) & 46.17 (2.04) & 46.5 (2.04)\\
		 & 0.8 & 93.67 (0.99) & 93.67 (0.99) & 93.83 (0.98) & 94 (0.97) & 93.67 (0.99) & 93.83 (0.98)\\
		\bottomrule
	\end{tabular}
}
\end{table}
\spacingset{1.45}

The logistic regression simulation results are given Table \ref{tbl:logistic test sims ar1}, which shows the rejection rates over 600 simulation replications for the tests with significance level $\alpha = 0.05$.
These results echo those from the linear model simulations. 
In particular, the rejection rates are similar for the tests using the two-step LLA solutions and the tests using the oracle estimators.
As in the linear regression simulations, the partial penalized tests all achieve rejection probabilities close to their nominal size when the null is true; 
the estimated rejection probabilities for the tests increase as $h_1$ increases; 
and the partial penalized Wald, score, and likelihood ratio tests achieve similar rejection probabilities in every setting, though the partial penalized likelihood ratio tests appears to achieve slightly higher rejection probabilities than the Wald and score tests when testing $\text{H}_0^{(2)}$ (the same is true, however, for the tests using the oracle estimators).
As such, the chi-square approximation for the test statistics seems appropriate in this case as well.

\section{Concluding Remarks} \label{sec:conclusion}
In this paper, we close the previous gap between theory and computation for partial penalized tests of GLMs by establishing theoretical guarantees for specific computed solutions to the partial penalized estimation problems---the two-step LLA solutions.
To accomplish this, we employ a novel proof strategy: showing that the two-step LLA solutions possess the strong oracle property (Corollary \ref{cor:glm lla convergence prob with lasso}), then using the strong oracle property and the asymptotic properties of the oracle estimators to establish guarantees for the partial penalized tests evaluated at the two-step LLA solutions (Theorem \ref{thm:lla test statistic distribution} and Corollary \ref{cor:test statistic asymptotics}).
This is, to our knowledge, the first time the strong oracle property has been used to study the asymptotic properties of testing procedures in this way.
While partial penalized tests of GLMs are our primary objects of study, we also propose LLA algorithms and develop theory for a much broader class of partial penalized estimation problems, laying the groundwork for studying and developing other high dimensional inference procedures using this strong oracle approach.

\bibliographystyle{apalike}
\bibliography{lla_test}

\appendix

	\section{Computational Details} \label{sec:implementation details}
	In this appendix, we discuss computation of the partial penalized estimators in greater detail.
	We describe how we compute the LLA updates in Algorithms \ref{algo:LLA for reduced estimator} and \ref{algo:LLA for full estimator} and how we select the penalty parameters $\lambda_{\lasso}$ and $\lambda$ in practice.
	We focus on computation of the reduced model estimator \eqref{eqn:Shi reduced}, as a similar approach works for the full model estimator \eqref{eqn:Shi full}.
	
	\subsection{ADMM-within-LLA}
	As we noted in Section \ref{sec:lla algorithms}, because the weighted lasso problem \eqref{algo:reduced model lasso update} in Algorithm \ref{algo:LLA for reduced estimator} includes linear constraints, it cannot be solved with existing algorithms for the lasso.  
	We use an ADMM approach to solve \eqref{algo:reduced model lasso update}, as it enables us to enforce those constraints.
	We introduce $\evec \in \reals^{p-m}$ to rewrite \eqref{algo:reduced model lasso update} as 
	\begin{equation}
		(\hat \bvec_0^{(b)}, \hat \evec) = \argmin_{\bvec,\evec} \ell_n(\bvec) + \sum_{j = 1}^{p-m} \hat w_{j+m}^{(b-1)}|\eta_j| \text{ \; subject to \; } \Cmat \bvec_{\Mset} = \bm{t} \text{, } \bvec_{\Mset^c} = \evec. \label{eqn:dual objective}
	\end{equation}
	The augmented Lagrangian for \eqref{eqn:dual objective} can be expressed as
	\begin{align*}
		L_{\rho}(\bvec, \evec, \nvec) 
		& = \ell_n(\bvec) + \sum_{j = 1}^{p-m} \hat w_{j+m}^{(b-1)}|\eta_j|
		+ \frac{\rho}{2}  \norm{ \Cmat \bvec_\Mset - \bm{t} + \frac{\nvec_1}{\rho} }_2^2 \\
		& \hspace*{0.4cm} + \frac{\rho}{2} \norm{ \bvec_{\Mset^c} - \evec + \frac{\nvec_2}{\rho} }_2^2 
		- \frac{\rho}{2} \norm{\frac{\nvec}{\rho}}_2^2, \label{eqn:scaled lagrangian}
	\end{align*}
	with penalty parameter $\rho > 0$ and dual variables $\nvec = (\nvec_1', \nvec_2')' \in \reals^{r+p-m}$.
	Applying the dual ascent method, we arrive at Algorithm \ref{algo:ADMM} to solve \eqref{eqn:dual objective}.
	
	\spacingset{1.25}
	\begin{algorithm}
		\caption{ADMM algorithm for reduced model LLA updates \eqref{eqn:dual objective}}
		\label{algo:ADMM}
		\begin{algorithmic}
			\State Initialize $(\bvec^{[0]}, \evec^{[0]}, \nvec^{[0]})$, $k = 1$
			
			\Repeat
			\State $\bvec^{[k]} = \argmin_{\bvec} \ell_n(\bvec)
			+ \frac{\rho}{2}  \norm{ \Cmat \bvec_\Mset - \bm{t} + \frac{\nvec_1^{[k-1]}}{\rho} }_2^2
			+ \frac{\rho}{2} \norm{ \bvec_{\Mset^c} - \evec^{[k-1]} + \frac{\nvec_2^{[k-1]}}{\rho} }_2^2 $
			
			\State $\evec^{[k]} = \argmin_{\evec} \sum_{j = 1}^{p-m} \hat w_{j+m}^{(b-1)}|\eta_j|
			+ \frac{\rho}{2} \norm{ \bvec_{\Mset^c}^{[k]} - \evec + \frac{\nvec_2^{[k-1]}}{\rho} }_2^2 $
			
			\State $\nvec_1^{[k]} = \nvec_1^{[k-1]} + \rho ( \Cmat \bvec_\Mset^{[k]} - \bm{t})$
			
			\State $\nvec_2^{[k]} = \nvec_2^{[k-1]} + \rho ( \bvec_{\Mset^c}^{[k]}  - \evec^{[k]} )$
			
			\State $k = k+1$
			\Until{primal and dual residuals are sufficiently small}
			
			\State $\hat \bvec_0^{(b)} = \bvec^{[k]}$
			\State $\hat \evec = \evec^{[k]}$
		\end{algorithmic}
	\end{algorithm}
	\spacingset{1.45}
	
	\subsection{Choosing $\lambda_{\lasso}$ and $\lambda$}
	We initialize Algorithm \ref{algo:LLA for reduced estimator} with the $\ell_1$-penalized estimator \eqref{eqn:glm:lasso initial estimator}, as is supported by our theoretical results. 
	We need $\norm{\hat{\bvec}_{0, \Mset^c}^{\lasso} - \bvec_{\Mset^c}^*}_{\max} \leq a_0 \lambda$ for the LLA algorithm to converge to the oracle estimator in two steps.
	In other words, we want to minimize the estimation error of the initial lasso estimator.
	We approximate this in practice by computing $\hat \bvec^{\lasso}$ for a sequence of $\lambda_{\lasso}$ values and selecting the $\lambda_{\lasso}$ value which minimizes the estimated prediction error in 10-fold cross validation.
	
	Having selected $\lambda_{\lasso}$, we solve \eqref{eqn:Shi reduced} for a sequence of $\lambda$ values.
	Let $\hat{\bvec}_0^{\lambda}$ denote the LLA solution with penalty parameter $\lambda$.
	We pick $\lambda$ based on the following information criterion:
	\begin{equation}
		\hat{\lambda} =  \argmin_{\lambda} \, n \, \ell_n(\hat{\bvec}_0^{\lambda}) + c_n \norm{\hat{\bvec}_0^{\lambda}}_0 \label{eqn:GIC}
	\end{equation}
	where $c_n = \max\{ \log n, \log(\log n) \log p \}$. 
	\cite{Fan2013} and \cite{Schwarz1978} show that penalty coefficients of $\log(\log n)\log p$ and $\log n$ guarantee consistent model selection for GLMs in the ultra-high dimensional and fixed-$p$ settings, respectively. 
	Similar arguments to theirs could be used to extend these guarantees to partial penalized estimators.
	
	\subsection{Estimating $\phi^*$ in linear regression models}
	In a linear regression model, $\phi^* = {\sigma}^2$, the variance of the additive error term $\varepsilon$ in $y  = \xvec'\bvec^* + \varepsilon$.
	In practice, We estimate $\phi^*$ with $\hat{\phi}_0 = \frac{1}{n - |\hat S (\hat \bvec_0) | - |\Mset| - 1 }\sum_{i=1}^{n} (y_i - \xvec_i'\hat \bvec_0)^2 $ for the reduced model and
	$\hat{\phi}_a = \frac{1}{n - |\hat S (\hat \bvec_a) | - |\Mset|- 1 }\sum_{i=1}^{n} (y_i - \xvec_i'\hat \bvec_a)^2 $ for the full model.
	Theorem \ref{thm:lla test statistic distribution} requires that $|\hat \phi - \phi^*| = o_p(1)$.
	\cite{Shi2019} prove that $\hat{\phi}_a = \phi^* + o_p(1)$ in Section S2 of their supplementary material.
	A similar argument to theirs can be used to show that $\hat{\phi}_0 = \phi^* + o_p(1)$ under the conditions of Theorem \ref{thm:two step LLA asymptotics}.

	\section{Additional Simulations}
	As we have noted, the use of a folded-concave penalty in the partial penalized likelihood \eqref{eqn:partial penalized likelihood} makes \eqref{eqn:Shi reduced} and \eqref{eqn:Shi full} potentially non-convex problems with multiple local minima.
	In these simulations, we examine whether our LLA algorithms and the ADMM algorithms proposed by \cite{Shi2019} arrive at different local solutions in practice.
	
	We generate the response from a linear model $y_i = \xvec_i' \bvec^* + \varepsilon_i$ for $i = 1, \ldots, n$, where $\varepsilon_i \sim N(0, 1)$ and $\xvec_i \sim N(\bm{0}, \bm \Sigma)$ is a $p$-dimensional vector of predictors
	with $\Sigma_{jk} = 0.5^{|j-k|}$ for $j,k = 1, \ldots, p$. 
	We set the data-generating coefficient vector to be $\bvec^* = (3, 1.5, 0, 0, 2, \bm 0_{p-5})$. 
	We run simulations with both $p = 50$ and $p = 200$, generating $n = 100$ observations per dataset in both settings.
	
	For each simulated dataset, we compute the reduced and full model estimators for testing $\text{H}_0: \beta_3^* = 0$ using our two-step LLA solutions and \citeauthor{Shi2019}'s ADMM algorithms.
	For each method, we compute the partial penalized estimators for a sequence of $\lambda$ values and select a final $\hat \lambda$ using the information criterion \eqref{eqn:GIC}.
	To compare the two-step LLA and ADMM solutions, we calculate the $\ell_1$ loss $\lVert \hat{\bvec} - \bvec^* \rVert_1$ and $\ell_2$ loss $\lVert \hat{\bvec} - \bvec^* \rVert_2$ as well as the number of false positive (\#FP) and false negative (\#FN) variable selections.
	
	Table \ref{tbl:linear estimation sims} reports the mean of each metric over 500 simulation replications, with its standard error given in parentheses.
	We see that the $\ell_1$ and $\ell_2$ losses for the LLA and ADMM solutions differ from each other for both the full and reduced models.
	Likewise, the LLA and ADMM solutions deliver different numbers of false positive and false negative variable selections.
	These results illustrate that the LLA and ADMM algorithms can arrive at different solutions in practice, underscoring the importance of developing theory for specific local solutions to \eqref{eqn:Shi reduced} and \eqref{eqn:Shi full}.
	
	\spacingset{1}
	\begin{table}[t]
		\centering
		\caption{LLA and ADMM Estimator Comparison}\label{tbl:linear estimation sims}
		\begin{tabular}{lllll}
			\toprule
			Method & $\ell_1$ loss & $\ell_2$ loss & \#FP & \#FN\\
			\midrule
			\addlinespace[0.3em]
			& \multicolumn{4}{c}{$p = 50$}\\
			\cline{2-5}
			LLA (full) & 3.25 (0.03) & 3.07 (0.05) & 1 (0) & 0.4 (0.02)\\
			ADMM (full) & 3.28 (0.03) & 3.14 (0.05) & 1 (0) & 0.41 (0.02)\\
			LLA (reduced) & 2.08 (0.03) & 1.66 (0.04) & 0.48 (0.02) & 0.06 (0.01)\\
			ADMM (reduced) & 2.13 (0.03) & 1.75 (0.04) & 0.35 (0.02) & 0.07 (0.01)\\
			\addlinespace[0.3em]
			& \multicolumn{4}{c}{$p = 200$}\\
			\cline{2-5}
			LLA (full) & 3.25 (0.03) & 3.07 (0.05) & 1 (0) & 0.43 (0.02)\\
			ADMM (full) & 3.27 (0.03) & 3.13 (0.05) & 1 (0) & 0.44 (0.02)\\
			LLA (reduced) & 2.07 (0.03) & 1.67 (0.04) & 0.38 (0.02) & 0.08 (0.01)\\
			ADMM (reduced) & 2.11 (0.03) & 1.75 (0.04) & 0.48 (0.02) & 0.09 (0.01)\\
			\bottomrule
		\end{tabular}
	\end{table}
	\spacingset{1.45}

	\section{Testing Antiretroviral Treatments for HIV}\label{sec:hiv data}
	As an application, we analyze data from OPTIONS \citep{Gandhi2020}, a partially randomized trial of antiretroviral therapies for people with human immunodeficiency virus (HIV) (the data are available from the Stanford HIV Drug Resistance Database at \url{https://hivdb.stanford.edu/pages/clinicalStudyData/ACTG5241.html}).
	The trial examined whether adding or omitting nucleoside reverse transcriptase inhibitors (NRTIs) from the treatment regimens of patients who are failing protease inhibitor-based treatment impacts the likelihood of subsequent virological failure.
	The researchers developed an optimized treatment regimen for each participant based on their treatment history and drug resistance.
	Participants with moderate drug resistance were randomly assigned to either add or omit NRTIs from their treatment regimens while participants with high drug resistance were all given treatment regimens which included NRTIs.
	
	Participants were regularly evaluated over the course of 96 weeks, with researchers measuring viral load (copies/mL HIV RNA) to determine whether virological failure had occurred.
	A participant was said to have experienced virological failure if (1) their $log_{10}$ viral load decreased by less than 1 by week 12, (2) their viral load exceeded 200 copies/mL at any point after dropping below 200 copies/mL, (3) their viral load had never been suppressed to less than 200 copies/mL by week 24, or (4) their viral load exceeded 200 copies/mL at or after week 48.
	Virological failure was confirmed with repeated measurement.
	By week 96, 361 participants remained in the trial.
	Of the remaining participants, 128 had experienced virological failure based on the criteria above.
	
	The outcome of interest is whether virological failure occurred at any point within the 96 week study period.
	We model this binary outcome using logistic regression.
	The predictors in our model include the antiretroviral medications assigned to the participants at the start of the trial, the week in which the participants started these medications, protease and reverse transcriptase gene mutations in the participants' HIV infections, and HIV subtype.
	Because participants were assigned personalized regimens of several medications, a separate indicator variable is used to denote the inclusion/exclusion of each medication from a participant's regimen.
	The data contain $n = 361$ observations and $p = 1225$ predictors.
	
	We test whether the NRTIs included in some patients' initial treatment regimens are significant predictors of virological failure.
	The participants' treatment regimens could include the following five NRTIs: AZT, 3TC, FTC, ABC, and TDF.
	Let $\bvec_{\Mset}^*$ denote the subvector of coefficients for the inclusion/exclusion indicators for these five NRTIs. 
	We test H$_0: \bvec_{\Mset}^* = \bm 0$ at significance level $0.05$ using the partial penalized likelihood ratio, Wald, and score tests evaluated at the two-step LLA solutions.
	The p-values from these tests are $p_{\text{LRT}} = 0.847$, $p_{\text{Wald}} = 0.870$, and $p_{\text{score}} = 0.866$.
	As such, all three tests fail to reject the null hypothesis and we conclude that these NRTIs are not significant predictors of virological failure within 96 weeks of treatment assignment.
	
	\section{Theoretical Proofs}\label{sec:proofs}
	This appendix contains proofs of the results from the main paper.
	
	\subsection{Proofs of general LLA convergence results}
	\begin{proof}[Proof of Theorem \ref{thm:one step LLA}]
		We will focus on proving the result for Algorithm \ref{algo:LLA for reduced estimator} as the proof for Algorithm \ref{algo:LLA for full estimator} follows the same lines.
		
		Recall that $\hat{\bvec}_0^{(0)} = \hat{\bvec}_0^{\initial}$. Under \ref{assumption: min signal for betaS} and $\left\{ \norm{\hat{\bvec}_{0, \Mset^c}^{\initial} - \bvec_{\Mset^c}^*}_{\max} \leq a_0 \lambda \right\}$, we see that for $j \in \Sset$,
		$$ \left| \hat{\beta}_{0,j}^{(0)} \right| \geq \norm{\bvec_{\Sset}^*}_{\min} - \norm{\hat{\bvec}_{0, \Mset^c}^{\initial} - \bvec_{\Mset^c}^*}_{\max} > (a+1)\lambda - \lambda = a\lambda .$$
		Therefore by property \ref{assumption: penalty iv} of $p_{\lambda}(\cdot)$, $p_{\lambda}(|\hat{\beta}_{0,j}^{(0)}|) = 0$ for $j \in  \Sset$.
		As an immediate consequence, $\hat{\bvec}_0^{(1)}$ can be expressed as the solution to
		\begin{align}
			\hat{\bvec}_0^{(1)} 
			& = \argmin_{\bvec} \ell_n(\bvec) + \sum_{j \in \Mset^c} p_{\lambda}'(|\hat{\beta}_{0,j}^{(0)}|)|\beta_j| \text{ subject to } \Cmat \bvec_{\Mset} = \bm{t} \notag \\
			& = \argmin_{\bvec} \ell_n(\bvec) + \sum_{j \in (\Mset \cup \Sset)^c} p_{\lambda}'(|\hat{\beta}_{0,j}^{(0)}|)|\beta_j| \text{ subject to } \Cmat \bvec_{\Mset} = \bm{t} . \label{eqn:proof 1 beta01 reduced}
		\end{align}
		
		We will now show that $\hat{\bvec}_0^{\oracle}$ is the unique global solution to \eqref{eqn:proof 1 beta01 reduced} under the event $ \left\{ \norm{ \nabla_{(\Mset \cup \Sset)^c} \ell_n( \hat{\bvec}_0^{\oracle} )  }_{\max} < a_1 \lambda \right\} $. 
		Because $\ell_n(\bvec)$ is convex, we see that for all $\bvec \in \reals^p$
		\begin{align}
			\ell_n(\bvec) 
			& \geq \ell_n(\hat{\bvec}_0^{\oracle}) + \nabla \ell_n(\hat{\bvec}_0^{\oracle})'(\bvec - \hat{\bvec}_0^{\oracle}) \notag \\
			& \geq \ell_n(\hat{\bvec}_0^{\oracle}) + \nabla_{(\Mset \cup \Sset)^c} \ell_n(\hat{\bvec}_0^{\oracle})'(\bvec_{(\Mset \cup \Sset)^c} - \hat{\bvec}_{0,(\Mset \cup \Sset)^c}^{\oracle}) \notag\\
			& \hspace{0.4cm} + \nabla_{\Mset \cup \Sset} \ell_n(\hat{\bvec}_0^{\oracle})'(\bvec _{\Mset \cup \Sset}- \hat{\bvec}_{0,\Mset \cup \Sset}^{\oracle}) \notag\\
			& = \ell_n(\hat{\bvec}_0^{\oracle}) + \sum_{j \in (\Mset \cup \Sset)^c} \nabla_j \ell_n(\hat{\bvec}_0^{\oracle})(\beta_j - \hat{\beta}_{0,j}^{\oracle}) \notag\\
			& \hspace{0.4cm} + \nvec'\Cmat(\bvec_{\Mset} - \hat{\bvec}_{0,\Mset}^{\oracle}) \label{eqn:proof 1 elln lower bound}
		\end{align}
		for some $\nvec \in \reals^r$ by \ref{assumption:reduced oracle unique}.
		
		Let $\bvec$ satisfy $\Cmat\bvec_{\Mset} = \bm{t}$.
		Since $\Cmat \hat{\bvec}_{0,\Mset}^{\oracle} = \bm{t}$,
		we see that
		$
		\nvec'\Cmat(\bvec_{\Mset} - \hat{\bvec}_{0,\Mset}^{\oracle}) = \nvec'\bm{t} - \nvec'\bm{t} = \bm 0
		$.
		In addition, under \ref{assumption: min signal for betaS} and $\left\{ \norm{\hat{\bvec}_{0, \Mset^c}^{\initial} - \bvec_{\Mset^c}^*}_{\max} \leq a_0 \lambda \right\}$, we see that for $j \in (\Mset \cup \Sset)^c$
		$$ \left| \hat{\beta}_{0,j}^{(0)} \right| \leq \norm{\hat{\bvec}_{0, \Mset^c}^{\initial} - \bvec_{\Mset^c}^*}_{\max} \leq a_0 \lambda \leq a_2\lambda .$$
		As such, properties \ref{assumption: penalty ii} and \ref{assumption: penalty iii} of $p_{\lambda}(\cdot)$ provide that 
		$p_{\lambda}'(|\hat{\beta}_{0,j}^{(0)}|) \geq a_1 \lambda$ for $j \in (\Mset \cup \Sset)^c$.
		Combining these findings with \eqref{eqn:proof 1 elln lower bound} and the fact that $\hat{\bvec}_{0, (\Mset \cup \Sset)^c}^{\oracle} = \bm{0}$, we find	
		\begin{align*}
			\bigg\{ \ell_n(\bvec) + \sum_{j \in (\Mset \cup \Sset)^c} & p_{\lambda}'(|\hat{\beta}_{0,j}^{(0)}|)|\beta_j| \bigg\} - \bigg\{ \ell_n(\hat{\bvec}_0^{\oracle}) + \sum_{j \in (\Mset \cup \Sset)^c} p_{\lambda}'(|\hat{\beta}_{0,j}^{(0)}|)|\hat{\beta}_{0,j}^{\oracle}| \bigg\} \\
			& =  \ell_n(\bvec) - \ell_n(\hat{\bvec}_0^{\oracle}) +  \sum_{j \in (\Mset \cup \Sset)^c} p_{\lambda}'(|\hat{\beta}_{0,j}^{(0)}|)|\beta_j|\\
			& \geq \sum_{j \in (\Mset \cup \Sset)^c} \nabla_j \ell_n(\hat{\bvec}_0^{\oracle})(\beta_j - \hat{\beta}_{0,j}^{\oracle}) +  \sum_{j \in (\Mset \cup \Sset)^c} p_{\lambda}'(|\hat{\beta}_{0,j}^{(0)}|)|\beta_j| \\
			& = \sum_{j \in (\Mset \cup \Sset)^c}\bigg\{p_{\lambda}'(|\hat{\beta}_{0,j}^{(0)}|) + \nabla_j \ell_n(\hat{\bvec}_0^{\oracle})  \sgn(\beta_j) \bigg\}|\beta_j| \\
			& \geq \sum_{j \in (\Mset \cup \Sset)^c}\bigg\{a_1\lambda - \left|\nabla_j \ell_n(\hat{\bvec}_0^{\oracle}) \right|  \bigg\}|\beta_j|\\
			& \geq 0
		\end{align*}  
		where the final inequality follows from  $ \left\{ \norm{ \nabla_{(\Mset \cup \Sset)^c} \ell_n( \hat{\bvec}_0^{\oracle})  }_{\max} < a_1 \lambda \right\} $. The final inequality in the above expression is strict unless $\beta_j = 0$ for all $j \in (\Mset \cup \Sset)^c$. This combined with the fact that $\hat{\bvec}_0^{\oracle}$ is the unique solution to \eqref{eqn:reduced oracle} implies that $\hat{\bvec}_0^{\oracle}$ is the unique solution to \eqref{eqn:proof 1 beta01 reduced}. Therefore $\hat{\bvec}_0^{(1)} = \hat{\bvec}_0^{\oracle}$.
	\end{proof}
	
	\begin{proof}[Proof of Theorem \ref{thm:two step lla}]
		We will focus on proving the result for Algorithm \ref{algo:LLA for reduced estimator} as the proof for Algorithm \ref{algo:LLA for full estimator} follows similar lines.
		
		Suppose that the LLA algorithm for the reduced model estimator has found $\hat{\bvec}_0^{\oracle}$ in the $(b-1)$th step---that is, $\hat{\bvec}_0^{(b-1)} = \hat{\bvec}_0^{\oracle}$. We see that if $\norm{\hat{\bvec}_{0, \Sset}^{\oracle}}_{\min} > a\lambda$, then $p_{\lambda}(|\hat{\beta}_{0,j}^{\oracle}|) = 0$ for $j \in \Sset$ by property \ref{assumption: penalty iv} of $p_{\lambda}(\cdot)$. As such, $\hat{\bvec}_0^{(b)}$ can be expressed as the solution to
		\begin{align}
			\hat{\bvec}_0^{(b)} 
			& = \argmin_{\bvec} \ell_n(\bvec) + \sum_{j \in \Mset^c} p_{\lambda}'(|\hat{\beta}_{0,j}^{(b-1)}|)|\beta_j| \text{ subject to } \Cmat \bvec_{\Mset} = \bm{t} \notag \\
			& = \argmin_{\bvec} \ell_n(\bvec) + \sum_{j \in (\Mset \cup \Sset)^c} p_{\lambda}'(|\hat{\beta}_{0,j}^{(b-1)}|)|\beta_j| \text{ subject to } \Cmat \bvec_{\Mset} = \bm{t} . \label{eqn:proof 2 beta01 reduced}
		\end{align}
		Note that \eqref{eqn:proof 2 beta01 reduced} is almost identical to \eqref{eqn:proof 1 beta01 reduced} in the proof of Theorem \ref{thm:one step LLA}.
		Since $\hat{\bvec}_{0,(\Mset \cup \Sset)^c}^{\oracle} = \bm{0}$, we see that $p_{\lambda}'(|\hat{\beta}_{0,j}^{(b-1)}|) = p_{\lambda}'(0) \geq a_1 \lambda$ for $j \in (\Mset \cup \Sset)^c$ by condition \ref{assumption: penalty ii}.
		As such, by the same argument as in the proof of Theorem \ref{thm:one step LLA} if $\left\{ \norm{ \nabla_{(\Mset \cup \Sset)^c} \ell_n( \hat{\bvec}_0^{\oracle} )  }_{\max} < a_1 \lambda \right\}$, then $\hat{\bvec}_0^{\oracle}$ is the unique solution to \eqref{eqn:proof 2 beta01 reduced} and therefore $\hat{\bvec}_0^{(b)} = \hat{\bvec}_0^{\oracle}$.
	\end{proof}
	
	\subsection{Proofs of LLA convergence results for GLMs}
	\begin{proof}[Proof of Theorem \ref{thm:glm lla covergence prob}]
		We focus on proving $\delta_{01}^{\text{GLM}} \to 0$ and  $\delta_{02}^{\text{GLM}} \to 0$ as the proof for the full model case follows the same lines.\\
		
		\noindent \textbf{Part 1.}
		We start by proving that $\delta_{02}^{\text{GLM}} \to 0$.
		We define $\tilde{\bvec}^*$ as follows:
		\begin{equation*}
			\begin{cases}
				\tilde{\bvec}^*_{\Mset} = \bvec^*_{\Mset} - \Cmat'(\Cmat \Cmat')^{-1} \bm{h}_n\\
				\tilde{\bvec}^*_{\Mset^c} = \bvec^*_{\Mset^c}.
			\end{cases}
		\end{equation*}
		We see that under \ref{assumption:hn order},
		$\Cmat \tilde{\bvec}^*_{\Mset} - \bm t = \Cmat\bvec^*_{\Mset} - \Cmat \Cmat'(\Cmat \Cmat')^{-1} \bm{h}_n - \bm t = \Cmat \bvec^*_{\Mset} - \bm{h}_n - \bm t = \bm 0$.
		We further see that under \ref{assumption:hn order},
		$ \norm{ \tilde{\bvec}^*_{\Mset} - \bvec^*_{\Mset} }_2^2
		= \norm{ \Cmat'(\Cmat \Cmat')^{-1} \bm{h}_n }_2^2
		= \bm{h}_n'  (\Cmat \Cmat')^{-1} \Cmat \Cmat'(\Cmat \Cmat')^{-1} \bm{h}_n 
		= \bm{h}_n'  (\Cmat \Cmat')^{-1} \bm{h}_n
		\leq \norm{\bm{h}_n}_2^2 \lambda_{\max} \{(\Cmat \Cmat')^{-1}\}
		= O( \norm{\bm{h}_n}_2^2)
		= O\left( \frac{s+m-r}{n} \right) $.
		Since $\tilde{\bvec}^*_{\Mset^c} = \bvec^*_{\Mset^c}$, this implies that $\norm{ \tilde{\bvec}^*_{\Mset \cup \Sset} - \bvec^*_{\Mset \cup \Sset} }_2 = O\left( \sqrt{\frac{s+m-r}{n}} \right)$.
		
		If $\bvec$ satisfies $\Cmat \bvec_{\Mset} = \bm t$, then $\Cmat(\bvec_{\Mset} - \tilde{\bvec}^*_{\Mset}) = \bm t - \bm t = \bm 0$, meaning that $\bvec_{\Mset} - \tilde{\bvec}^*_{\Mset}$ belongs to the null space of $\Cmat$.
		Let $\mathbf{Z}$ be a basis matrix for the null space of $\Cmat$ with orthonormal columns, so that $\mathbf{Z}'\mathbf{Z} = \mathbf{I}_{m-r}$.
		Then for any $\bvec$ satisfying $\Cmat \bvec_{\Mset} = \bm t$ there exists a unique $\bm v \in \reals^{m-r}$ such that $\bvec_{\Mset} - \tilde{\bvec}^*_{\Mset} = \mathbf{Z} \bm v$.
		For any $\evec \in \reals^{s+m-r}$, we define the affine transformation $\bvec(\evec)$ by
		\begin{equation*}
			\begin{cases}
				\bvec(\evec)_{\Mset} = \tilde{\bvec}^*_{\Mset} + \mathbf{Z} \evec_{\widetilde{\Mset}}\\
				\bvec(\evec)_{\Sset} = \tilde{\bvec}^*_{\Sset} + \evec_{\widetilde{\Mset}^c}\\
				\bvec(\evec)_{(\Mset \cup \Sset)^c} = \tilde{\bvec}^*_{(\Mset \cup \Sset)^c}
			\end{cases}
		\end{equation*}
		where $\widetilde{\Mset} = \{1, \ldots, m-r\}$.
		We reorder our indices so that $\Mset = \{1, \ldots, m\}$ and $\Sset = \{m+1, \ldots, m+s\}$ and define 
		$\widetilde{\mathbf{Z}} 
		= \begin{bsmallmatrix}
			\mathbf{Z} & \bm 0 \\ 
			\bm 0 & \mathbf{I}_s
		\end{bsmallmatrix}$, 
		so that 
		$\bvec(\evec) 
		= \begin{bsmallmatrix}
			\widetilde{\mathbf{Z}} \evec \\ 
			\bm 0
		\end{bsmallmatrix}$.
		We further define $\bar{\ell}_n(\evec) = \ell_n(\bvec(\evec))$.
		We see that $\forall \bm v \in \reals^{m-r}$, $\norm{\mathbf{Z} \bm v}_2^2 = \bm v' \mathbf{Z}' \mathbf{Z} \bm v = \bm v' \bm v = \norm{\bm v}_2^2$.
		Therefore $\norm{ \bvec(\evec)_{\Mset \cup \Sset} - \tilde{\bvec}^*_{\Mset \cup \Sset} }_2^2 = \norm{ \widetilde{\mathbf{Z}}\evec }_2^2 = \norm{\mathbf{Z}\evec_{\widetilde{\Mset}}}_2^2 + \norm{\evec_{\widetilde{\Mset^c}}}_2^2 = \norm{\evec}_2^2$.
		
		For $\tau > 0$, define $\mathcal{N}_{\tau} = \left\{ \evec \in \reals^{s+m-r}: \norm{\evec}_2 \leq \tau \sqrt{\frac{s+m-r}{n}} \right\} $ and consider the event
		$$
		\mathcal{E}_n = \left\{ \bar{\ell}_n(\bm 0) < \min_{\eta \in \partial \mathcal{N}_{\tau}} \bar{\ell}_n(\evec) \right\}
		$$
		where $\partial \mathcal{N}_{\tau}$ denotes the boundary of $\mathcal{N}_{\tau}$. Suppose that $\mathcal{E}_n$ holds. Then by the extreme value theorem there exists a local minimizer $\hat{\evec}$ of $\bar{\ell}_n(\evec)$ in $\mathcal{N}_{\tau}$. 
		This, in turn, implies that $\hat{\bvec} = \bvec(\hat{\evec})$ is a local minimizer of \eqref{eqn:reduced oracle}
		with the property that
		$ \norm{ \hat{\bvec}_{\Mset \cup \Sset} - \tilde{\bvec}^*_{\Mset \cup \Sset} }_2 \leq \tau \sqrt{\frac{s+m-r}{n}}$.
		Because $\ell_n(\bvec)$ is convex and the constraints in \eqref{eqn:reduced oracle} are affine, $\hat{\bvec}$ is a global minimizer of \eqref{eqn:reduced oracle}. 
		Since $\hat{\bvec}^{\oracle}_0$ is the unique global minimizer of \eqref{eqn:reduced oracle} under \ref{assumption:reduced oracle unique}, this implies that $\hat{\bvec} = \hat{\bvec}^{\oracle}_0$ and, by extension,
		$ \norm{ \hat{\bvec}_{0,\Mset \cup \Sset}^{\oracle} - \tilde{\bvec}^*_{\Mset \cup \Sset} }_2 \leq \tau \sqrt{\frac{s+m-r}{n}}$.
		Our aim, therefore, is to show that $\Pr(\mathcal{E}_n) \to 1$ as $n \to \infty$ for large $\tau$.
		
		Let $\tau \leq \sqrt{\log n}/2$ and $\evec \in \partial \mathcal{N}_{\tau}$. By a second order Taylor expansion, we have
		\begin{equation}
			\bar{\ell}_n(\evec) = \bar{\ell}_n(\bm 0) + \evec'\nabla\bar{\ell}_n(\bm 0) + \frac{1}{2} \evec' \nabla^2\bar{\ell}_n(\tilde{\evec}) \evec \label{eqn:glm:taylor expansion of bar ell}
		\end{equation}
		where $\tilde{\evec}$ lies on the line segment between $\bm 0$ and $\evec$. Clearly $\tilde{\evec} \in \mathcal{N}_{\tau}$.
		As such, we see that if $n$ is sufficiently large, then
		\begin{align*}
			\norm{ \bvec(\tilde \evec)_{\Mset \cup \Sset} - \bvec^*_{\Mset \cup \Sset}  }_2
			& \leq \norm{ \bvec(\tilde \evec)_{\Mset \cup \Sset} - \tilde{\bvec}^*_{\Mset \cup \Sset}  }_2 + \norm{ \tilde{\bvec}^*_{\Mset \cup \Sset} - \bvec^*_{\Mset \cup \Sset}  }_2\\
			& \leq \tau \sqrt{\frac{s+m-r}{n}} + O\left( \sqrt{\frac{s+m-r}{n}} \right)\\
			& \leq \tau \sqrt{\frac{s+m-r}{n}} + \frac{1}{2} \sqrt{\frac{(s+m)\log n}{n}}\\
			& \leq \sqrt{\frac{(s+m)\log n}{n}},
		\end{align*}
		meaning that $\bvec(\tilde \evec) \in \mathcal{N}_0$.
		Applying the chain rule, we see that for any $\evec$, $\nabla\bar{\ell}_n(\evec) = \widetilde{\mathbf{Z}} ' \nabla_{\Mset \cup \Sset} \ell_n(\bvec(\evec))$ and $\nabla^2\bar{\ell}_n(\evec) = \widetilde{\mathbf{Z}} ' \nabla_{\Mset \cup \Sset}^2 \ell_n(\bvec(\evec)) \widetilde{\mathbf{Z}}$.
		We have already established that $\norm{ \widetilde{\mathbf{Z}}\evec }_2 = \norm{\evec}_2$ for any $\evec$.
		Suppose that $n$ is large enough that $\bvec(\tilde{\evec}) \in \mathcal{N}_0$.
		Then we have
		\begin{align*}
			\evec' \nabla^2\bar{\ell}_n(\tilde{\evec}) \evec
			& = \evec' \widetilde{\mathbf{Z}}'\nabla_{\Mset \cup \Sset}^2 \ell_n(\bvec(\tilde \evec)) \widetilde{\mathbf{Z}} \evec \\
			& \geq \norm{ \widetilde{\mathbf{Z}} \evec }_2^2 \lambda_{\min} \left\{ \nabla_{\Mset \cup \Sset}^2 \ell_n(\bvec(\tilde \evec)) \right\} \\
			& = \norm{ \evec }_2^2 \lambda_{\min} \left\{ \frac{1}{n} \xmat_{\Mset \cup \Sset} \bm \Sigma (\xmat \bvec(\tilde{\evec}) ) \xmat_{\Mset \cup \Sset} \right\}\\
			& \geq \norm{ \evec }_2^2 c
		\end{align*}
		where the first inequality follows from the min-max theorem and the final inequality follows from \ref{assumption:glm:min hessian}.
		
		Returning to \eqref{eqn:glm:taylor expansion of bar ell}, we see that by the Cauchy-Schwarz inequality
		\begin{align*}
			\bar{\ell}_n(\evec) 
			& = \bar{\ell}_n(\bm 0) + \evec'\nabla\bar{\ell}_n(\bm 0) + \frac{1}{2} \evec' \nabla^2\bar{\ell}_n(\tilde{\evec}) \evec \\
			& \geq \bar{\ell}_n(\bm 0) - \norm{\evec}_2 \norm{ \nabla\bar{\ell}_n(\bm 0) }_2 + \frac{1}{2} \norm{\evec}_2^2 c\\
			& = \bar{\ell}_n(\bm 0) + \tau \sqrt{\frac{s+m-r}{n}}\left(- \norm{ \nabla\bar{\ell}_n(\bm 0) }_2 + \frac{c \tau}{2} \sqrt{\frac{s+m-r}{n}} \right) .
		\end{align*}
		Therefore $\bar{\ell}_n(\evec) > \bar{\ell}_n(\bm 0)$ for all $\evec \in \partial \mathcal{N}_{\tau}$---and therefore $\mathcal{E}_n$ holds---if $n$ is sufficiently large and the following event holds:
		\begin{equation}
			\norm{ \nabla\bar{\ell}_n(\bm 0) }_2 < \frac{c \tau}{2} \sqrt{\frac{s+m-r}{n}}.\label{eqn:glm:gradient of ell bar bound}
		\end{equation}
		Suppose that
		\begin{equation}
			\E\left[ \norm{ \nabla\bar{\ell}_n(\bm 0) }_2^2 \right] = O\left(\frac{s+m-r}{n}\right) \label{eqn:glm:expectation of gradient of ell bar bound}
		\end{equation}
		Then Markov's inequality yields
		\begin{align}
			\Pr\left( \norm{ \nabla\bar{\ell}_n(\bm 0) }_2 < \frac{c \tau}{2} \sqrt{\frac{s+m-r}{n}} \right)
			& > 1 - \frac{\E\left[ \norm{ \nabla\bar{\ell}_n(\bm 0) }_2^2 \right] 4n}{c^2 \tau^2 (s+m-r)} \notag \\
			& = 1 - O(\tau^{-2}) .\label{eqn:glm:markov for gradient}
		\end{align}
		
		Let $\tau_n$ be a diverging sequence such that $\tau_n \leq \sqrt{\log n}/2$ and $\tau_n \sqrt{\frac{s+m}{n}} = o(\lambda)$. 
		As we have shown, $\mathcal{E}_n$ implies that $ \norm{ \hat{\bvec}_{0,\Mset \cup \Sset}^{\oracle} - \tilde{\bvec}^*_{\Mset \cup \Sset} }_2 \leq \tau_n \sqrt{\frac{s+m-r}{n}}$.
		Consequently, \eqref{eqn:glm:markov for gradient} provides that with probability at least $1 - O(\tau_n^{-2})$, 
		$
		\norm{ \hat{\bvec}_{0,\Mset \cup \Sset}^{\oracle} - \bvec^*_{\Mset \cup \Sset} }_2 
		\leq \norm{ \hat{\bvec}_{0,\Mset \cup \Sset}^{\oracle} - \tilde{\bvec}^*_{\Mset \cup \Sset} }_2 + \norm{\tilde{\bvec}^*_{\Mset \cup \Sset} - \bvec^*_{\Mset \cup \Sset} }_2 
		\leq \tau_n \sqrt{\frac{s+m-r}{n}}  + O\left(\sqrt{\frac{s+m-r}{n}}\right)
		= o(\lambda)
		< \lambda 
		$
		for sufficiently large $n$, under \ref{assumption:rate of lambda}.
		Moreover, if $\norm{ \hat{\bvec}_{0,\Mset \cup \Sset}^{\oracle} - \bvec^*_{\Mset \cup \Sset} }_2 < \lambda$, then by \ref{assumption: min signal for betaS},
		$
		\norm{\hat{\bvec}^{\oracle}_{0,\Sset}}_{\min}
		\geq \norm{\bvec^*_{\Sset}}_{\min} - \norm{ \hat{\bvec}_{0,\Mset \cup \Sset}^{\oracle} - \bvec^*_{\Mset \cup \Sset} }_{\max} > (a+1)\lambda - \lambda = a \lambda .
		$
		Therefore we see that $$\delta_{02}^{\text{GLM}} = \Pr\left(\norm{\hat{\bvec}_{0, \Sset}^{\oracle}}_{\min} \leq a\lambda \right) \leq O(\tau_n^{-2}) \to 0$$ as $n \to \infty$ and this part of the proof is complete.
		
		All that remains is to prove \eqref{eqn:glm:expectation of gradient of ell bar bound}.
		Note that $\nabla\bar{\ell}_n(\bm 0) = \widetilde{\mathbf{Z}}' \nabla\ell_n( \tilde{\bvec}^*  )$. 
		Leveraging the properties of the Fisher Information matrix for exponential families, we find
		\begin{align}
			\E\left[ \norm{ \nabla\bar{\ell}_n(\bm 0) }_2^2 \right]
			& = \E\left[ \norm{ \widetilde{\mathbf{Z}}' \nabla\ell_n(\tilde{\bvec}^* ) }_2^2 \right] \notag \\
			& = \E\left[ (\nabla\ell_n(\tilde{\bvec}^*))' \widetilde{\mathbf{Z}} \widetilde{\mathbf{Z}}' \nabla\ell_n(\tilde{\bvec}^*) \right] \notag \\
			& = \tr \left\{ \widetilde{\mathbf{Z}}' \E \left[ \nabla\ell_n(\tilde{\bvec}^*) (\nabla\ell_n(\tilde{\bvec}^*))' \right] \widetilde{\mathbf{Z}} \right\} \notag \\
			& = \frac{\phi^*}{n^2} \tr \left\{ \widetilde{\mathbf{Z}}' \xmat_{\Mset \cup \Sset}'\bm{\Sigma}(\xmat \tilde{\bvec}^*)\xmat_{\Mset \cup \Sset} \widetilde{\mathbf{Z}} \right\} \notag \\
			& = \frac{\phi^*}{n^2} \tr \left\{ \widetilde{\mathbf{Z}}' \xmat_{\Mset \cup \Sset}'\bm{\Sigma}(\xmat \bvec^*)\xmat_{\Mset \cup \Sset} \widetilde{\mathbf{Z}} \right\} \notag \\ 
			& \hspace{0.5cm} + \frac{\phi^*}{n^2} \tr \left\{ \widetilde{\mathbf{Z}}' \xmat_{\Mset \cup \Sset}'\left[ \bm{\Sigma}(\xmat \tilde{\bvec}^*) - \bm{\Sigma}(\xmat \bvec^*) \right]\xmat_{\Mset \cup \Sset} \widetilde{\mathbf{Z}} \right\} \notag \\
			& = I_1 + I_2. \label{eqn:glm:I1 and I2}
		\end{align}
		Recall that  $\norm{ \widetilde{\mathbf{Z}}\bm v }_2 = \norm{\bm v}_2$ for any $\bm v \in \reals^{s+m-r}$.
		As such, we see that
		\begin{align*}
			I_1 
			& \leq \frac{(s+m-r)\phi^*}{n^2} \lambda_{\max} \left\{ \widetilde{\mathbf{Z}}' \xmat_{\Mset \cup \Sset}'\bm{\Sigma}(\xmat \bvec^*)\xmat_{\Mset \cup \Sset} \widetilde{\mathbf{Z}} \right\} \\
			& = \frac{(s+m-r)\phi^*}{n^2} \max_{\norm{\bm v}_2 = 1} \bm v' \widetilde{\mathbf{Z}}' \xmat_{\Mset \cup \Sset}'\bm{\Sigma}(\xmat \bvec^*)\xmat_{\Mset \cup \Sset} \widetilde{\mathbf{Z}} \bm v\\
			& \leq \frac{(s+m-r)\phi^*}{n^2} \max_{\norm{\bm u}_2 = 1} \bm u' \xmat_{\Mset \cup \Sset}'\bm{\Sigma}(\xmat \bvec^*)\xmat_{\Mset \cup \Sset} \bm u\\
			& = \frac{(s+m-r)\phi^*}{n^2} \lambda_{\max} \left\{ \xmat_{\Mset \cup \Sset}'\bm{\Sigma}(\xmat \bvec^*)\xmat_{\Mset \cup \Sset} \right\}\\
			& = O\left(\frac{s+m-r}{n}\right)
		\end{align*}
		where the final equality follows from \ref{assumption:glm:max hessian}.
		By the same argument, we find that
		\begin{equation}
			I_2 \leq \frac{(s+m-r)\phi^*}{n^2} \max_{\norm{\bm u}_2 = 1} \bm u'\xmat_{\Mset \cup \Sset}'\left[ \bm{\Sigma}(\xmat \tilde{\bvec}^*) - \bm{\Sigma}(\xmat \bvec^*) \right]\xmat_{\Mset \cup \Sset} \bm u. \label{eqn:glm:I2 bound}
		\end{equation}
		
		Let $\norm{\bm u}_2 = 1$. By the mean value theorem, we know that there exists $\bar{\bvec}$ on the line segment between $\bvec^*$ and $\tilde{\bvec}^*$ such that
		\begin{align*}
			\bigg|\bm u'\xmat_{\Mset \cup \Sset}'\bigg[ \bm{\Sigma}&(\xmat \tilde{\bvec}^*)  - \bm{\Sigma}(\xmat \bvec^*) \bigg]\xmat_{\Mset \cup \Sset} \bm u \bigg|\\
			& \leq \sum_{j \in \Mset \cup \Sset} (\bm u'\xmat_{\Mset \cup \Sset}'\diag\left\{ |\xmat_j| \circ |\bm b'''(\xmat \bar \bvec)| \right\}\xmat_{\Mset \cup \Sset} \bm u) |\tilde{\beta}_j^* - {\beta}_j^*|\\
			& \leq \max_{j \in \Mset \cup \Sset}(\bm u'\xmat_{\Mset \cup \Sset}'\diag\left\{ |\xmat_j| \circ |\bm b'''(\xmat \bar \bvec)| \right\}\xmat_{\Mset \cup \Sset} \bm u) \norm{\tilde{\bvec}^*_{\Mset \cup \Sset} -  {\bvec}^*_{\Mset \cup \Sset}}_1\\
			& \leq  \max_{j \in \Mset \cup \Sset} \lambda_{\max} \left\{\xmat_{\Mset \cup \Sset}'\diag\left\{ |\xmat_j| \circ |\bm b'''(\xmat \bar \bvec)| \right\}\xmat_{\Mset \cup \Sset} \right\} \norm{\tilde{\bvec}^*_{\Mset \cup \Sset} -  {\bvec}^*_{\Mset \cup \Sset}}_1.
		\end{align*}
		We have already shown that $\norm{\tilde{\bvec}^*_{\Mset \cup \Sset} -  {\bvec}^*_{\Mset \cup \Sset}}_2 = O\left( \sqrt{\frac{s+m-r}{n}} \right)$. Since $\bar \bvec$ is on the line segment between $\bvec^*$ and $\tilde{\bvec}^*$, clearly $\norm{\bar{\bvec}_{\Mset \cup \Sset} -  {\bvec}^*_{\Mset \cup \Sset}}_2 = O\left( \sqrt{\frac{s+m-r}{n}} \right) \leq \sqrt{\frac{(s+m)\log n}{n}}$---that is, $\bar \bvec \in \mathcal{N}_0$--- if $n$ is sufficiently large.
		As such,
		\begin{align*}
			\max_{j \in \Mset \cup \Sset} \lambda_{\max} \bigg\{\xmat_{\Mset \cup \Sset}' & \diag\left\{ |\xmat_j| \circ |\bm b'''(\xmat \bar \bvec)| \right\}\xmat_{\Mset \cup \Sset} \bigg\} \\
			&\leq \sup_{\bvec \in \mathcal{N}_0} \max_{j \in \Mset \cup \Sset} \lambda_{\max} \left\{\xmat_{\Mset \cup \Sset}'\diag\left\{ |\xmat_j| \circ |\bm b'''(\xmat \bvec)| \right\}\xmat_{\Mset \cup \Sset} \right\}\\
			&= O(n) 
		\end{align*}
		by \ref{assumption:glm:third derivative bound}.
		Therefore we see that for any $\bm u$ satisfying $\norm{\bm u}_2 = 1$,
		\begin{align*}
			\bigg|\bm u'\xmat_{\Mset \cup \Sset}'\bigg[ \bm{\Sigma}(\xmat \tilde{\bvec}^*)  - \bm{\Sigma}(\xmat \bvec^*) \bigg]\xmat_{\Mset \cup \Sset} \bm u \bigg|
			& \leq O(n) \norm{\tilde{\bvec}^*_{\Mset \cup \Sset} -  {\bvec}^*_{\Mset \cup \Sset}}_1 \\
			& \leq O(n \sqrt{s+m}) \norm{\tilde{\bvec}^*_{\Mset \cup \Sset} -  {\bvec}^*_{\Mset \cup \Sset}}_2\\
			& = O (\sqrt{n} (s+m))\\
			& = O(n)
		\end{align*}
		where the last equality holds since $s+m = o(\sqrt{n})$. Returning to \eqref{eqn:glm:I2 bound}, we conclude
		that $I_2 = \frac{(s+m-r)\phi^*}{n^2} O(n) = O\left(\frac{s+m-r}{n}\right)$.
		Applying this and our finding that $I_1 = O(\frac{s+m-r}{n})$ to \eqref{eqn:glm:I1 and I2}, we see that \eqref{eqn:glm:expectation of gradient of ell bar bound} holds.\\
		
		\noindent \textbf{Part 2.}
		We now shift to proving that $\delta_{01}^{\text{GLM}} \to 0$.
		Let $j \in (\Mset \cup \Sset)^c$. Applying the mean value theorem and using the fact that $\hat{\bvec}_{0,(\Mset \cup \Sset)^c}^{\oracle} = \bvec_{(\Mset \cup \Sset)^c}^* = \bm 0$, we find
		\begin{align*}
			\nabla_j \ell_n(\hat{\bvec}^{\oracle}_0) 
			& = \nabla_j \ell_n(\bvec^*) + \left[\nabla^2 \ell_n(\bar \bvec^{(j)})(\hat{\bvec}^{\oracle}_0 - \bvec^*)\right]_{j} \\
			& = \nabla_j \ell_n(\bvec^*) + \left[\nabla^2 \ell_n(\bar \bvec^{(j)})\right]_{j, \Mset \cup \Sset}(\hat{\bvec}^{\oracle}_{0,\Mset \cup \Sset} - \bvec^*_{\Mset \cup \Sset})  \\
			& = \nabla_j \ell_n(\bvec^*) + \left[\nabla^2 \ell_n(\bvec^*)\right]_{j, \Mset \cup \Sset}(\hat{\bvec}^{\oracle}_{0,\Mset \cup \Sset} - \bvec^*_{\Mset \cup \Sset}) +R_j
		\end{align*}
		where $\bar{\bvec}^{(j)}$ lies on the line segment between $\bvec^*$ and $\hat{\bvec}_0^{\oracle}$
		and we define $R_j = [\nabla^2 \ell_n(\bar \bvec^{(j)}) - \nabla^2 \ell_n(\bvec^*) ]_{j, \Mset \cup \Sset}(\hat{\bvec}^{\oracle}_{0,\Mset \cup \Sset} - \bvec^*_{\Mset \cup \Sset})$. We define $\bm R = (R_{m+s+1}, \ldots, R_p)' $ and write
		\begin{equation}
			\nabla_{(\Mset \cup \Sset)^c} \ell_n(\hat{\bvec}^{\oracle}_0) 
			= \nabla_{(\Mset \cup \Sset)^c} \ell_n(\bvec^*) + \left[\nabla^2 \ell_n(\bvec^*)\right]_{(\Mset \cup \Sset)^c, \Mset \cup \Sset}(\hat{\bvec}^{\oracle}_{0,\Mset \cup \Sset} - \bvec^*_{\Mset \cup \Sset}) +\bm R. \label{eqn:glm:mvt for gradient at oracle}
		\end{equation}
		In order to show that $\norm{ \nabla_{(\Mset \cup \Sset)^c} \ell_n( \hat{\bvec}_0^{\oracle} )  }_{\max} < a_1 \lambda$ with probability converging to $1$, we will bound the terms on the right hand side of \eqref{eqn:glm:mvt for gradient at oracle} under high probability events.
		
		We begin by bounding $\norm{\bm R}_{\max}$.
		In part 1, we showed that the event 
		$$\mathcal{F}_n = \left\{ \norm{ \hat{\bvec}_{0,\Mset \cup \Sset}^{\oracle} - \tilde{\bvec}^*_{\Mset \cup \Sset} }_2 \leq \tau \sqrt{\frac{s+m-r}{n}} \right\}$$ 
		occurs with probability at least $1 - O(\tau^{-2})$.
		We see that $\mathcal{F}_n$ implies
		$\norm{ \hat{\bvec}_{0,\Mset \cup \Sset}^{\oracle} - \bvec^*_{\Mset \cup \Sset} }_2 
		\leq \norm{ \hat{\bvec}_{0,\Mset \cup \Sset}^{\oracle} - \tilde{\bvec}^*_{\Mset \cup \Sset} }_2 + \norm{\tilde{\bvec}^*_{\Mset \cup \Sset} - \bvec^*_{\Mset \cup \Sset} }_2 
		\leq \tau \sqrt{\frac{s+m-r}{n}}  + O\left(\sqrt{\frac{s+m-r}{n}}\right)$.
		Let $\tau \leq \sqrt{\log n}/2$.
		Suppose that $\mathcal{F}_n$ holds and that $n$ is large enough that
		$
		\norm{ \hat{\bvec}_{0,\Mset \cup \Sset}^{\oracle} - \bvec^*_{\Mset \cup \Sset} }_2 
		\leq \sqrt{\frac{(s+m)\log n}{n}} 
		$, 
		so that $\hat{\bvec}_{0}^{\oracle} \in \mathcal{N}_0$.
		Applying the mean value theorem again, we see that for all $j \in (\Mset \cup \Sset)^c$,
		\begin{align}
			|R_j|
			& \leq \frac{1}{n} (\hat{\bvec}^{\oracle}_{0, \Mset \cup \Sset} - \bvec^*_{\Mset \cup \Sset})' \xmat_{\Mset \cup \Sset}' \diag\{ |\xmat_j| \circ |\bm b'''(\xmat \bar{\bar{\bvec}}^{(j)} ) | \}\xmat_{\Mset \cup \Sset}(\hat{\bvec}^{\oracle}_{0, \Mset \cup \Sset} - \bvec^*_{\Mset \cup \Sset}) \nonumber \\
			& \leq \max_{j \in (\Mset \cup \Sset)^c} \lambda_{\max}\left\{ \frac{1}{n} \xmat_{\Mset \cup \Sset}' \diag\{ |\xmat_j| \circ |\bm b'''(\xmat \bar{\bar{\bvec}}^{(j)} ) | \}\xmat_{\Mset \cup \Sset} \right\} \norm{\hat{\bvec}^{\oracle}_{0, \Mset \cup \Sset} - \bvec^*_{\Mset \cup \Sset}}_2^2 \label{eqn:glm:Rj bound} 
		\end{align}
		where $\bar{\bar{\bvec}}^{(j)}$ lies on the line segment between $\bvec^*$ and $\bar{\bvec}^{(j)}$.
		Since $\bar{\bar{\bvec}}^{(j)}$ is between $\bvec^*$ and $\hat{\bvec}_0^{\oracle}$, clearly  $\bar{\bar{\bvec}}^{(j)} \in \mathcal{N}_0$ for all $j \in (\Mset \cup \Sset)^c$. By extension, \ref{assumption:glm:third derivative bound} implies that
		$
		|R_j| = O\left(\norm{\hat{\bvec}^{\oracle}_{0, \Mset \cup \Sset} - \bvec^*_{\Mset \cup \Sset}}_2^2\right)
		$ for all $j \in (\Mset \cup \Sset)^c$ and therefore
		$$
		\norm{\bm R}_{\max} 
		= O\left(\norm{\hat{\bvec}^{\oracle}_{0, \Mset \cup \Sset} - \bvec^*_{\Mset \cup \Sset}}_2^2\right) = O\left( \tau^2\frac{(s+m)}{n} \right) + O\left( \tau \frac{(s+m)}{n} \right) + O\left( \frac{s+m}{n} \right).
		$$
		
		Turning our focus to the second term in \eqref{eqn:glm:mvt for gradient at oracle}, we see that under  $\mathcal{F}_n$ and \ref{assumption:glm:max hessian}
		\begin{align*}
			\bigg\lVert\left[\nabla^2 \ell_n(\bvec^*)\right]_{(\Mset \cup \Sset)^c, \Mset \cup \Sset}&(\hat{\bvec}^{\oracle}_{0,\Mset \cup \Sset} - \bvec^*_{\Mset \cup \Sset})\bigg\rVert_{\max} \\
			&\leq \norm{[\nabla^2 \ell_n(\bvec^*)]_{(\Mset \cup \Sset)^c, \Mset \cup \Sset}}_{\infty}\norm{\hat{\bvec}^{\oracle}_{0,\Mset \cup \Sset} - \bvec^*_{\Mset \cup \Sset}}_{\max} \\
			&\leq \frac{1}{n} \norm{\xmat_{(\Mset \cup \Sset)^c}' \bm \Sigma(\xmat \bvec^*) \xmat_{\Mset \cup \Sset}}_{\infty} \norm{ \hat{\bvec}^{\oracle}_{0,\Mset \cup \Sset} - \bvec^*_{\Mset \cup \Sset} }_2\\
			&= O\left(  \norm{ \hat{\bvec}^{\oracle}_{0,\Mset \cup \Sset} - \bvec^*_{\Mset \cup \Sset} }_2 \right)\\
			&= O\left( \tau \sqrt{ \frac{s+m}{n} } \right) + O\left( \sqrt{ \frac{s+m}{n} } \right).
		\end{align*}
		
		Lastly, we will bound $\norm{\nabla_{(\Mset \cup \Sset)^c} \ell_n(\bvec^*)}_{\max}$. We know
		$\nabla_{(\Mset \cup \Sset)^c} \ell_n(\bvec^*) = -\frac{1}{n}\xmat_{(\Mset \cup \Sset)^c}'(\yvec - \bm \mu(\xmat \bvec^*))$.
		Applying Proposition 4 of \cite{Fan2011}, we see that under \ref{assumption:glm:design column rates} and \ref{assumption:glm:expectation bound for chernoff}, for any $\gamma > 0$ and any $j \in (\Mset \cup \Sset)^c$
		\begin{align*}
			\Pr\left( |\xmat_j'(\yvec - \bm \mu(\xmat \bvec^*))| > \gamma\sqrt{n \log p} \right)
			& = 2 \exp\left\{-\frac{1}{2} \frac{\gamma^2 n \log p}{\norm{\xmat_j}_2^2 v_0 + \norm{\xmat_j}_{\infty} M \gamma\sqrt{n \log p} } \right\}\\
			& \leq 2 \exp\left\{ - \frac{1}{2} \frac{\gamma^2 \log p}{ B_1 v_0 + \gamma B_2 M }\right\}.
		\end{align*}
		Applying the union bound, we find
		\begin{align*}
			\Pr\left( \norm{\nabla_{(\Mset \cup \Sset)^c} \ell_n(\bvec^*)}_{\max} > \gamma\sqrt{\frac{\log p}{n}} \right)
			& = \Pr\left( \norm{\xmat_{(\Mset \cup \Sset)^c}'(\yvec - \bm \mu(\xmat \bvec^*))}_{\max} > \gamma\sqrt{n \log p} \right)\\
			& \leq 2(p-m-s) \exp\left\{ - \frac{1}{2} \frac{\gamma^2 \log p}{ B_1 v_0 + \gamma B_2 M }\right\}\\
			& =  2 \exp\left\{ - \frac{1}{2} \frac{\gamma^2 \log p}{ B_1 v_0 + \gamma B_2 M } + \log(p-m-s)\right\}.
		\end{align*}
		Suppose that $\gamma$ is large enough that $\gamma^2 > 2(B_1 v_0 + \gamma B_2 M)$.
		Then the right hand side of the previous expression rapidly converges to $0$ as $n, p\to \infty$. Therefore $\norm{\nabla_{(\Mset \cup \Sset)^c} \ell_n(\bvec^*)}_{\max} = O_p\left( \sqrt{\frac{\log p}{n}} \right)$.
		
		Let $\tau_n$ be a diverging sequence such that $\tau \leq \sqrt{\log n}/2$ and $\tau_n \cdot \max \left\{\frac{\sqrt{s+m}}{\sqrt{n}}, \frac{\sqrt{\log p}}{\sqrt n} \right\} = o(\lambda)$. 
		Define 
		$\mathcal{G}_n = \left\{\norm{\nabla_{(\Mset \cup \Sset)^c} \ell_n(\bvec^*)}_{\max} \leq \gamma\sqrt{\frac{\log p}{n}}\right\}$ with $\gamma$ large enough that $\Pr(\mathcal{G}_n^c) = o(1)$.
		Pulling everything together, we see that if $\mathcal{F}_n$ and $\mathcal{G}_n$ hold, then under \ref{assumption:rate of lambda} and our assumption that $s+m = o(\sqrt{n})$,
		\begin{align*}
			\norm{\nabla_{(\Mset \cup \Sset)^c} \ell_n(\hat{\bvec}^{\oracle}_0)}_{\max}
			& \leq \norm{ \nabla_{(\Mset \cup \Sset)^c} \ell_n(\bvec^*) }_{\max}\\
			& \hspace*{0.5cm} + \norm{ \left[\nabla^2 \ell_n(\bvec^*)\right]_{(\Mset \cup \Sset)^c,\Mset \cup \Sset}(\hat{\bvec}^{\oracle}_{0,\Mset \cup \Sset} - \bvec^*_{\Mset \cup \Sset})}_{\max} + \norm{\bm R}_{\max}\\
			& = O\left( \sqrt{\frac{\log p}{n}} \right) + O\left( \tau_n \sqrt{ \frac{s+m}{n} } \right) + O\left( \sqrt{ \frac{s+m}{n} } \right)\\
			& \hspace*{0.5cm} +  O\left( \tau_n^2\frac{(s+m)}{n} \right) + O\left( \tau_n \frac{(s+m)}{n} \right) + O\left( \frac{s+m}{n} \right)\\
			& \leq o(\lambda) +  o(\lambda) +  o(\lambda) +  O\left( \tau_n \sqrt{\frac{s+m}{n}} \sqrt{\frac{(s+m) \log n }{n}}\right) \\
			& \hspace*{0.5cm}  + O\left( \tau_n \sqrt{\frac{s+m}{n}} \sqrt{\frac{s+m}{n}}\right) + O\left( \sqrt{\frac{s+m}{n}} \sqrt{\frac{s+m}{n}} \right)\\
			& = o(\lambda) +  o(\lambda)O(1) + o(\lambda)O(1) + o(\lambda)O(1)\\
			& = o(\lambda) ,
		\end{align*}
		and, consequently, $\norm{\nabla_{(\Mset \cup \Sset)^c} \ell_n(\hat{\bvec}^{\oracle}_0)}_{\max} < a_1 \lambda$ for large $n$.
		As such, we have
		\begin{align*}
			\Pr \left(\norm{\nabla_{(\Mset \cup \Sset)^c} \ell_n(\hat{\bvec}^{\oracle}_0)}_{\max} \geq a_1 \lambda \right) \leq \Pr(\mathcal{F}_n^c) + \Pr(\mathcal{G}_n^c) = O(\tau_n^{-2}) + o(1) \to 0
		\end{align*}
		as $n \to \infty$, completing the proof.
	\end{proof}
	
	\begin{proof}[Proof of Theorem \ref{thm:glm initial estimator}]
		Define $F(\evec) = \ell_n(\bvec^* + \evec) - \ell_n(\bvec^*) + \lambda_{\lasso}\left(\norm{\bvec^* + \evec}_1 - \norm{\bvec^*}_1\right)$ and $\hat{\evec} = \argmin_{\evec} F(\evec)$. By \eqref{eqn:glm:lasso initial estimator}, we know $\hat{\evec} = \hat{\bvec}^{\lasso} - \bvec^*$. In addition, we see that $F(\hat \evec) \leq F(\bm 0) = 0$. Because $\ell_n(\bvec)$ is convex, clearly $F(\evec)$ is convex as well.
		Define the set $\mathcal{K}(\epsilon) = \mathcal{C} \cap \{\norm{\evec}_2  = \epsilon\}$ where $\mathcal{C} = \left\{\bm u \neq 0 : \norm{\bm u_{\mathcal{A}^c}}_1 \leq 3 \norm{\bm u_{\mathcal{A}}}_1 \right\}$ as in \ref{assumption:glm:RE condition}.
		Let $\epsilon = 5 \tilde{s}^{1/2}\lambda_{\lasso}/\kappa$.
		Lemma 4 of \cite{Negahban2012} provides that if $\hat \evec \in \mathcal{C}$ and $F(\evec) > \bm 0$ for all $\evec \in \kappa(\epsilon)$, then $\norm{\hat \evec}_2 \leq \epsilon$.
		As such, we aim to prove that $\hat \evec \in \mathcal{C}$ and $F(\evec) > \bm 0$ for all $\evec \in \kappa(\epsilon)$ under the event $\mathcal{F}_n = \left\{\norm{\nabla \ell_n(\bvec^*)}_{\max} \leq \frac{1}{2}\lambda_{\lasso}  \right\}$, and show that $\mathcal{F}_n$ holds with high probability.
		
		We start by proving that $\hat \evec \in \mathcal{C}$ under $\mathcal{F}_n$.
		By \eqref{eqn:glm:lasso initial estimator}, we know that 
		$$
		\ell_n(\hat \bvec^{\lasso}) + \lambda_{\lasso} \shortnorm{\hat \bvec^{\lasso}}_1 \leq \ell_n(\bvec^*) + \lambda_{\lasso} \shortnorm{\bvec^*}_1.
		$$
		By the convexity of $\ell_n(\cdot)$, this implies
		$$
		(\nabla \ell_n(\bvec^*))'(\hat \bvec^{\lasso} - \bvec^*) + \lambda_{\lasso} \shortnorm{\hat \bvec^{\lasso}}_1 
		\leq \ell_n(\hat \bvec^{\lasso}) -  \ell_n(\bvec^*) + \lambda_{\lasso} \shortnorm{\hat \bvec^{\lasso}}_1 
		\leq \lambda_{\lasso}\shortnorm{\bvec^*}_1.
		$$
		Suppose that $\mathcal{F}_n$ holds. Then clearly
		$$
		|(\nabla \ell_n(\bvec^*))'(\hat \bvec^{\lasso} - \bvec^*)| 
		\leq \shortnorm{\nabla \ell_n(\bvec^*)}_{\max} \shortnorm{ \hat \bvec^{\lasso} - \bvec^* }_1
		\leq \frac{1}{2}\lambda_{\lasso}\shortnorm{ \hat \bvec^{\lasso} - \bvec^* }_1.
		$$
		Combining the previous two expressions, we have 
		$$
		-\frac{1}{2}\lambda_{\lasso}\shortnorm{ \hat \bvec^{\lasso} - \bvec^* }_1 + \lambda_{\lasso} \shortnorm{\hat \bvec^{\lasso}}_1
		\leq \lambda_{\lasso}\shortnorm{\bvec^*}_1,
		$$
		which we can rewrite as 
		$$
		\frac{1}{2}\shortnorm{ \hat \bvec^{\lasso} - \bvec^* }_1
		\leq \shortnorm{\bvec^*}_1 - \shortnorm{\hat \bvec^{\lasso}}_1 + \shortnorm{ \hat \bvec^{\lasso} - \bvec^* }_1.
		$$
		
		We know that for all $j \in \mathcal{A}^c$, $\beta_j^* = 0$ and, consequently, 
		$
		|\beta_j^*| - | \hat{\beta}_j^{\lasso} | + | \hat {\beta}_j^{\lasso} - \beta_j^* | = 0
		$.
		As such, the previous inequality implies
		$$
		\frac{1}{2}\shortnorm{ \hat \bvec^{\lasso} - \bvec^* }_1
		\leq \shortnorm{\bvec^*_{\mathcal{A}}}_1 - \shortnorm{\hat \bvec^{\lasso}_{\mathcal{A}}}_1 + \shortnorm{ \hat \bvec^{\lasso}_{\mathcal{A}} - \bvec^*_{\mathcal{A}} }_1
		\leq 2 \shortnorm{ \hat \bvec^{\lasso}_{\mathcal{A}} - \bvec^*_{\mathcal{A}} }_1.
		$$
		Therefore we see that
		$$
		\frac{1}{2}\shortnorm{ \hat \bvec^{\lasso}_{\mathcal{A}^c} - \bvec^*_{\mathcal{A}^c} }_1 + \frac{1}{2}\shortnorm{ \hat \bvec^{\lasso}_{\mathcal{A}} - \bvec^*_{\mathcal{A}} }_1
		\leq 2 \shortnorm{ \hat \bvec^{\lasso}_{\mathcal{A}} - \bvec^*_{\mathcal{A}} }_1
		$$
		or, equivalently,
		$$
		\shortnorm{ \hat \bvec^{\lasso}_{\mathcal{A}^c} - \bvec^*_{\mathcal{A}^c} }_1
		\leq 3 \shortnorm{ \hat \bvec^{\lasso}_{\mathcal{A}} - \bvec^*_{\mathcal{A}} }_1.
		$$
		As such, $\mathcal{F}_n$ implies $\hat \evec \in \mathcal{C}$.
		
		We now shift to proving that $F(\evec) > \bm 0$ for all $\evec \in \kappa(\epsilon)$ under $\mathcal{F}_n$. 
		We start by deriving a lower bound for $\norm{\bvec^* + \evec}_1 - \norm{\bvec^*}_1$, finding that for all $\evec$
		\begin{equation}
			\norm{\bvec^* + \evec}_1 - \norm{\bvec^*}_1 
			= 	\norm{\bvec^*_{\mathcal{A}} + \evec_{\mathcal{A}}}_1 + \norm{\evec_{\mathcal{A}^c}}_1 - \norm{\bvec^*_{\mathcal{A}}}_1 
			\geq \norm{\evec_{\mathcal{A}^c}}_1 - \norm{\evec_{\mathcal{A}}}_1.\label{eqn:glm:difference of l1 norms bound}
		\end{equation}
		Next we will derive a lower bound for $\ell_n(\bvec^* + \evec) - \ell_n(\bvec^*)$. Let $\evec \in \kappa(\epsilon)$.
		Define $g:\reals \to \reals$ by $g(t) = \ell_n(\bvec^* + t \evec)$. One can easily show that
		\begin{eqnarray*}
			g''(t) & = & \frac{1}{n} \sum_{i=1}^n b''(\xvec_i'(\bvec^* + t \evec))(\xvec_i'\evec)^2\\
			g'''(t) & = & \frac{1}{n} \sum_{i=1}^n b'''(\xvec_i'(\bvec^* + t \evec))(\xvec_i'\evec)^3.
		\end{eqnarray*}
		By \ref{assumption:glm:self concordant}, we know that $\exists_{K > 0}$ such that $0 \leq |b'''(t)| \leq K b''(t)$ for all $t$. Therefore we find
		\begin{align}
			|g'''(t)|
			& \leq  \frac{1}{n} \sum_{i=1}^n | b'''(\xvec_i'(\bvec^* + t \evec))|\cdot|(\xvec_i'\evec)^3|\nonumber\\
			& \leq  \frac{1}{n} \sum_{i=1}^n K b''(\xvec_i'(\bvec^* + t \evec))(\xvec_i'\evec)^2\cdot|\xvec_i'\evec| \nonumber\\
			& \leq K \cdot \max_i |\xvec_i'\evec| \cdot g''(t) \nonumber \\
			& \leq K \cdot c \norm{\evec}_1 \cdot g''(t) \label{eqn:glm:g(3) bound}.
		\end{align}
		Because $\evec \in \mathcal{C}$, we have
		$ \norm{\evec}_1 \leq 4 \norm{\evec_{\mathcal{A}}}_1 \leq 4 \tilde{s}^{1/2} \norm{\evec_{\mathcal{A}}}_2 \leq 4 \tilde{s}^{1/2} \norm{\evec}_2$.
		Combining this with \eqref{eqn:glm:g(3) bound}, we find
		$ |g'''(t)| \leq 4 K c \tilde{s}^{1/2}\norm{\evec}_2 \cdot g''(t)$.
		Let $R = 4 K c \tilde{s}^{1/2}$ for ease of notation. By Proposition 1 of \cite{Bach2010}, since $|g'''(t)| \leq R\norm{\evec}_2 \cdot g''(t)$, we have
		\begin{equation*}
			\ell_n(\bvec^* + \evec) - \ell_n(\bvec^*)
			\geq \evec'\nabla \ell_n(\bvec^*) + \frac{\evec'\nabla^2 \ell_n(\bvec^*)\evec}{ R^2 \norm{\evec}_2^2}\left( e^{-R\norm{\evec}_2} + R\norm{\evec}_2 -1\right).
		\end{equation*}
		Define $h(z) = z^{-2} \left(e^{-z} + z - 1\right)$. Using elementary calculus, one can show that $h(z)$ is decreasing for $z > 0$.
		Clearly $R \norm{\evec}_2 > 0$. Under our assumption that $\lambda_{\lasso} \leq \frac{\kappa}{20K c \tilde s}$, we see that $R \norm{\evec}_2 = 20 K c \tilde{s}\lambda_{\lasso}/\kappa \leq 1$.
		As such, $ h(R\norm{\evec}_2) \geq h(1) > \frac{1}{3} $.
		Therefore we have
		\begin{equation}
			\ell_n(\bvec^* + \evec) - \ell_n(\bvec^*)
			\geq \evec'\nabla \ell_n(\bvec^*) + \frac{1}{3}\evec'\nabla^2 \ell_n(\bvec^*)\evec . \label{eqn:glm:lower bound for elln evec diff}
		\end{equation}
		
		We see that if $\mathcal{F}_n$ holds, then
		\begin{equation}
			\evec'\nabla \ell_n(\bvec^*) 
			\geq \norm{\evec}_1 \norm{\nabla \ell_n(\bvec^*)}_{\max} 
			\geq - \frac{1}{2}\lambda_{\lasso} \norm{\evec}_1. \label{eqn:glm:evec prod with gradient bound}
		\end{equation}
		Combining \eqref{eqn:glm:difference of l1 norms bound}, \eqref{eqn:glm:lower bound for elln evec diff}, and \eqref{eqn:glm:evec prod with gradient bound}, we see that under $\mathcal{F}_n$ and \ref{assumption:glm:RE condition},
		\begin{align*}
			F(\evec) 
			& = \ell_n(\bvec^* + \evec) - \ell_n(\bvec^*) + \lambda_{\lasso}\left(\norm{\bvec^* + \evec}_1 - \norm{\bvec^*}_1\right) \\
			& \geq \evec'\nabla \ell_n(\bvec^*) + \frac{1}{3}\evec'\nabla^2 \ell_n(\bvec^*)\evec + \lambda_{\lasso} \left(\norm{\evec_{\mathcal{A}^c}}_1 - \norm{\evec_{\mathcal{A}}}_1 \right)\\
			& \geq -\frac{1}{2}\lambda_{\lasso} \norm{\evec}_1 + \frac{1}{3} \kappa \norm{\evec}_2^2 + \lambda_{\lasso} \left(\norm{\evec_{\mathcal{A}^c}}_1 - \norm{\evec_{\mathcal{A}}}_1 \right)\\
			& \geq \frac{1}{3} \kappa \norm{\evec}_2^2 - \frac{3}{2} \lambda_{\lasso} \norm{\evec_{\mathcal{A}}}_1\\
			& \geq \frac{1}{3} \kappa \norm{\evec}_2^2 - \frac{3}{2} \lambda_{\lasso} \tilde{s}^{1/2} \norm{\evec_{\mathcal{A}}}_2\\
			& \geq \frac{1}{3} \kappa \norm{\evec}_2^2 - \frac{3}{2} \lambda_{\lasso} \tilde{s}^{1/2} \norm{\evec}_2\\
			& = \frac{1}{3} \kappa \epsilon^2 - \frac{3}{2} \lambda_{\lasso} \tilde{s}^{1/2} \epsilon \\
			& = \frac{5}{6} \frac{\tilde s \lambda_{\lasso}^2}{\kappa}\\
			& > 0.
		\end{align*}
		
		We have shown that if $\mathcal{F}_n$ holds, then $\hat \evec \in \mathcal{C}$ and $F(\evec) > \bm 0$ for all $\evec \in \kappa(\epsilon)$ and, consequently, $\shortnorm{\hat \bvec^{\lasso} - \bvec^* }_2  = \shortnorm{\hat \evec}_2 \leq \epsilon = 5 \tilde{s}^{1/2}\lambda_{\lasso}/\kappa$ by Lemma 4 of \cite{Negahban2012}.
		To finish the proof, we need only derive a lower bound for the probability that $\mathcal{F}_n$ occurs.
		Under \ref{assumption:glm:design column rates} and \ref{assumption:glm:expectation bound for chernoff}, we can apply Proposition 4 of \cite{Fan2011} and the union bound (as in the proof of Theorem \ref{thm:glm lla covergence prob}) to show
		\begin{align*}
			\Pr(\mathcal{F}_n^c)
			& = \Pr\left( \norm{\xmat'(\yvec - \bm\mu(\xmat'\bvec^*))  }_{\max} > \frac{n}{2}\lambda_{\lasso} \right)\\
			& \leq 2p\exp\left\{  -\frac{1}{8} \frac{n^2 \lambda_{\lasso}^2}{\max_j\norm{\xmat_j}_2^2v_0 + \max_j \norm{\xmat_j}_{\infty} \frac{nM}{2}\lambda_{\lasso} } \right\}\\
			& \leq 2 p \exp \left\{ -\frac{1}{8}\frac{n \lambda_{\lasso}^2}{B_1 v_0 + \frac{1}{2}B_2 M\sqrt{\frac{n}{\log p}} \lambda_{\lasso}   } \right\}.
		\end{align*}
		
	\end{proof}

	\subsection{Proofs of asymptotic results for the oracle estimators}
	\begin{proof}[Proof of Lemma \ref{lemma:oracle asymptotics}]
		We focus on proving the results for $\hat{\bvec}^{\oracle}_0$ as the proofs for $\hat{\bvec}^{\oracle}_a$ follow the same lines.\\
		
		\noindent \textbf{Part 1.} 
		Define $\bar{\ell}_n(\evec)$, $\mathcal{N}_{\tau}$, and $\mathcal{E}_n$ as in the proof of Theorem \ref{thm:glm lla covergence prob}.
		In the proof of Theorem \ref{thm:glm lla covergence prob}, we showed that if $\tau \leq \sqrt{\log n}/2$ and $n$ is sufficiently large, then $\mathcal{E}_n$ implies
		$ \norm{\hat{\bvec}^{\oracle}_{0, \Mset \cup \Sset} - \bvec^*_{\Mset \cup \Sset}}_2 \leq (\tau + \bar c)\sqrt{\frac{s+m-r}{n}}$ for some $\bar c > 0$
		and $\Pr(\mathcal{E}_n) > 1 - O(\tau^{-2})$.
		As such, for every $\epsilon > 0$, we can find $N \in \mathbb{N}$, $\tau \leq \sqrt{\log N}/2$ such that
		$$
		\Pr\left( \norm{\hat{\bvec}^{\oracle}_{0, \Mset \cup \Sset} - \bvec^*_{\Mset \cup \Sset}}_2 \leq (\tau + \bar c)\sqrt{\frac{s+m-r}{n}} \right) \geq \Pr(\mathcal{E}_n) > 1 - \epsilon 
		$$
		for all $n > N$.
		Therefore
		$
		\norm{\hat{\bvec}^{\oracle}_{0, \Mset \cup \Sset} - \bvec^*_{\Mset \cup \Sset}}_2 = O_p\left(  \sqrt{\frac{s+m-r}{n}} \right)  
		$ 
		by definition.\\
		
		\noindent \textbf{Part 2.} 
		We have shown that $\norm{\hat{\bvec}^{\oracle}_{0, \Mset \cup \Sset} - \bvec^*_{\Mset \cup \Sset}}_2 = O_p\left( \sqrt{\frac{s+m-r}{n}} \right)$. It is straightforward to show that this implies $\Pr(\hat{\bvec}^{\oracle}_0 \in \mathcal{N}_0) \to 1$ as $n \to \infty$.
		We will use these results to establish \eqref{eqn:reduced oracle limiting expression}.
		
		Applying the mean value theorem componentwise, as in \eqref{eqn:glm:mvt for gradient at oracle}, we find
		\begin{align}
			\nabla_{\Mset \cup \Sset} \ell_n(\hat{\bvec}^{\oracle}_0) 
			& = \nabla_{\Mset \cup \Sset} \ell_n(\bvec^*) +  \bm K_n(\hat{\bvec}^{\oracle}_{0, \Mset \cup \Sset} - \bvec^*_{\Mset \cup \Sset}) + \bm{R} \label{eqn:glm:mvt for MuS gradient}
		\end{align}
		where $R_j = [\nabla^2 \ell_n(\tilde \bvec^{(j)}) - \nabla^2 \ell_n(\bvec^*)]_{j, \Mset \cup \Sset} (\hat{\bvec}^{\oracle}_{0, \Mset \cup \Sset} - \bvec^*_{\Mset \cup \Sset})$ and $\tilde \bvec^{(j)}$ is on the line segment between $\bvec^*$ and $\hat{\bvec}^{\oracle}_0$ for $j \in \Mset \cup \Sset$.
		Applying the mean value theorem again, as in \eqref{eqn:glm:Rj bound}, and using the fact that $\Pr(\hat{\bvec}^{\oracle}_0 \in \mathcal{N}_0) \to 1$ as $n \to \infty$, one can show that
		$$
		\norm{\bm R}_{\max} = O_p\left(\norm{\hat{\bvec}^{\oracle}_{0, \Mset \cup \Sset} - \bvec^*_{\Mset \cup \Sset} }_2^2 \right) = O_p\left( \frac{s+m}{n} \right).
		$$
		Therefore we find
		$$
		\norm{\bm R}_{2} \leq (s+m)	\norm{\bm R}_{\max} = O_p\left(\frac{(s+m)^{3/2}}{n}\right) = o_p(n^{-1/2})
		$$
		since $s+m = o(n^{1/3})$.
		
		Combining \eqref{eqn:glm:mvt for MuS gradient} with \ref{assumption:reduced oracle unique}, we find
		$$
		\begin{bmatrix}
			\Cmat ' \nvec\\
			\bm 0
		\end{bmatrix} 
		= \nabla_{\Mset \cup \Sset} \ell_n(\bvec^*) +  \bm K_n(\hat{\bvec}^{\oracle}_{0, \Mset \cup \Sset} - \bvec^*_{\Mset \cup \Sset}) + \bm{R}
		$$
		for some $\nvec \in \reals^r$. Rearranging terms and multiplying by $\sqrt{n} \bm K_n^{-1}$ yields
		\begin{equation*}
			\sqrt{n} (\hat{\bvec}^{\oracle}_{0, \Mset \cup \Sset} - \bvec^*_{\Mset \cup \Sset}) 
			= -\sqrt{n} \bm K_n^{-1}  \nabla_{\Mset \cup \Sset} \ell_n(\bvec^*) +  
			\sqrt{n} \bm K_n^{-1}
			\begin{bmatrix}
				\Cmat ' \\
				\bm 0
			\end{bmatrix} \nvec 
			- \sqrt{n} \bm K_n^{-1} \bm{R}.
		\end{equation*}
		Under \ref{assumption:glm:min hessian}, we see that
		$
		\norm{\bm K_n^{-1}}_2 
		= \lambda_{\max} \{ \bm K_n^{-1} \}   
		= \lambda_{\min}^{-1} \{ \bm K_n \}  
		\leq \frac{1}{c}   
		$
		for all $n$ and, by extension,
		$
		\norm{\bm K_n^{-1} \bm R }_2 \leq \norm{ \bm K_n^{-1} }_2 \norm{\bm R}_2 = o_p(n^{-1/2})
		$.
		As such, the previous expression simplifies to
		\begin{equation}
			\sqrt{n} (\hat{\bvec}^{\oracle}_{0, \Mset \cup \Sset} - \bvec^*_{\Mset \cup \Sset}) 
			= -\sqrt{n} \bm K_n^{-1}  \nabla_{\Mset \cup \Sset} \ell_n(\bvec^*) +  
			\sqrt{n} \bm K_n^{-1}
			\begin{bmatrix}
				\Cmat ' \\
				\bm 0
			\end{bmatrix} \nvec 
			+ o_p(1). \label{eqn:glm:root n oracle expression with nu}
		\end{equation}
		
		To simplify further, we solve for $\nvec$. We multiply both sides of \eqref{eqn:glm:root n oracle expression with nu} by $\frac{1}{\sqrt n} \begin{bsmallmatrix} \Cmat & \bm 0 \end{bsmallmatrix}$ to find
		\begin{align*}
			\begin{bmatrix} \Cmat & \bm 0 \end{bmatrix}(\hat{\bvec}^{\oracle}_{0, \Mset \cup \Sset} - \bvec^*_{\Mset \cup \Sset}) 
			& = - \begin{bmatrix} \Cmat & \bm 0 \end{bmatrix} \bm K_n^{-1}  \nabla_{\Mset \cup \Sset} \ell_n(\bvec^*) +  
			\bm \Psi \nvec 
			+ \begin{bmatrix} \Cmat & \bm 0 \end{bmatrix} \widetilde{\bm R},
		\end{align*}
		where $\widetilde{\bm R} = o_p(n^{-1/2})$.
		Since $\Cmat \hat{\bvec}_{0, \Mset}^{\oracle} = \bm t$ and $\Cmat \bvec^*_{\Mset} = \bm t + \bm h_n$ by \ref{assumption:hn order}, we see that
		$ \begin{bsmallmatrix} \Cmat & \bm 0 \end{bsmallmatrix}(\hat{\bvec}^{\oracle}_{0, \Mset \cup \Sset} - \bvec^*_{\Mset \cup \Sset}) = \Cmat \hat{\bvec}_{0, \Mset}^{\oracle} - \Cmat\bvec^*_{\Mset} = -\bm h_n$.
		Combining this with the previous expression, we find
		$$
		- \bm h_n = - \begin{bmatrix} \Cmat & \bm 0 \end{bmatrix} \bm K_n^{-1}  \nabla_{\Mset \cup \Sset} \ell_n(\bvec^*) +  \bm \Psi \nvec + \begin{bmatrix} \Cmat & \bm 0 \end{bmatrix}\widetilde{\bm R}.
		$$
		Rearranging terms and multiplying by $\bm \Psi^{-1}$ yields
		$$
		\nvec = - \bm \Psi^{-1} \bm h_n + \bm \Psi^{-1}\begin{bmatrix} \Cmat & \bm 0 \end{bmatrix} \bm K_n^{-1}  \nabla_{\Mset \cup \Sset} \ell_n(\bvec^*) - \bm \Psi^{-1} \begin{bmatrix} \Cmat & \bm 0 \end{bmatrix}\widetilde{\bm R}.
		$$
		Plugging this expression for $\nvec$ into \eqref{eqn:glm:root n oracle expression with nu}, we find
		\begin{align}
			\sqrt{n} (\hat{\bvec}^{\oracle}_{0, \Mset \cup \Sset} - \bvec^*_{\Mset \cup \Sset}) 
			= & -\sqrt{n} \bm K_n^{-1}  \nabla_{\Mset \cup \Sset} \ell_n(\bvec^*) 
			- \sqrt{n} \bm K_n^{-1} \begin{bmatrix} \Cmat'\\ \bm 0 \end{bmatrix} \bm \Psi^{-1} \bm h_n \nonumber \\
			& + \sqrt{n} \bm K_n^{-1} \begin{bmatrix} \Cmat'\\ \bm 0 \end{bmatrix} \bm \Psi^{-1}\begin{bmatrix} \Cmat & \bm 0 \end{bmatrix} \bm K_n^{-1}  \nabla_{\Mset \cup \Sset} \ell_n(\bvec^*)  \nonumber \\
			& - \sqrt{n} \bm K_n^{-1} \begin{bmatrix} \Cmat'\\ \bm 0 \end{bmatrix} \bm \Psi^{-1} \begin{bmatrix} \Cmat & \bm 0 \end{bmatrix}\widetilde{\bm R} 
			+ o_p(1)  \nonumber \\
			= & \sqrt{n}\bm K_n^{-1/2} (\bm P_n - \bm I) \bm K_n^{-1/2}  \nabla_{\Mset \cup \Sset} \ell_n(\bvec^*) 
			- \sqrt{n} \bm K_n^{-1} \begin{bmatrix} \Cmat'\\ \bm 0 \end{bmatrix} \bm \Psi^{-1} \bm h_n  \nonumber \\
			& - \sqrt{n} \bm K_n^{-1} \begin{bmatrix} \Cmat'\\ \bm 0 \end{bmatrix} \bm \Psi^{-1} \begin{bmatrix} \Cmat & \bm 0 \end{bmatrix}\widetilde{\bm R} 
			+ o_p(1). \label{eqn:glm:root n oracle expression without nu}
		\end{align}
		
		Since $\bm h_n = \Cmat \Cmat'(\Cmat \Cmat')^{-1}\bm h_n $, we find that 
		\begin{align}
			\bm K_n^{-1} \begin{bmatrix} \Cmat'\\ \bm 0 \end{bmatrix} \bm \Psi^{-1} \bm h_n
			& = \bm K_n^{-1} \begin{bmatrix} \Cmat'\\ \bm 0 \end{bmatrix} \bm \Psi^{-1}  \begin{bmatrix} \Cmat & \bm 0 \end{bmatrix} \begin{bmatrix} \Cmat'(\Cmat \Cmat')^{-1}\bm h_n \\ \bm 0 \end{bmatrix} \nonumber \\
			& = \bm K_n^{-1/2} \bm P_n \bm K_n^{1/2} \begin{bmatrix} \Cmat'(\Cmat \Cmat')^{-1}\bm h_n \\ \bm 0 \end{bmatrix}. \label{eqn:Kn term expansion for reduced estimator}
		\end{align}
		In addition, we see that under \ref{assumption:glm:max hessian},
		\begin{align*}
			\norm{  \bm K_n^{-1} \begin{bsmallmatrix} \Cmat'\\ \bm 0 \end{bsmallmatrix} \bm \Psi^{-1} \begin{bsmallmatrix} \Cmat & \bm 0 \end{bsmallmatrix}\widetilde{\bm R} }_2
			&\leq \norm{\bm K_n^{-1/2} }_2 \norm{  \bm K_n^{-1/2} \begin{bsmallmatrix} \Cmat'\\ \bm 0 \end{bsmallmatrix} \bm \Psi^{-1} \begin{bsmallmatrix} \Cmat & \bm 0 \end{bsmallmatrix}\bm K_n^{-1/2}  }_2 \norm{\bm K_n^{1/2}}_2 \shortnorm{\widetilde{\bm R}}_2\\
			& = \lambda_{\max}^{1/2} \{\bm K_n^{-1} \} \lambda_{\max}\{ \bm K_n^{-1/2} \begin{bsmallmatrix} \Cmat'\\ \bm 0 \end{bsmallmatrix} \bm \Psi^{-1} \begin{bsmallmatrix} \Cmat & \bm 0 \end{bsmallmatrix}\bm K_n^{-1/2}  \} \lambda_{\max}^{1/2} \{\bm K_n \} \shortnorm{\widetilde{\bm R}}_2\\
			& = \lambda_{\max}^{1/2} \{\bm K_n^{-1} \} \lambda_{\max}\{ \begin{bsmallmatrix} \Cmat & \bm 0 \end{bsmallmatrix}\bm K_n^{-1}  \begin{bsmallmatrix} \Cmat'\\ \bm 0 \end{bsmallmatrix} \bm \Psi^{-1} \} \lambda_{\max}^{1/2} \{\bm K_n \} \shortnorm{\widetilde{\bm R}}_2\\
			& = \lambda_{\max}^{1/2} \{\bm K_n^{-1} \} \lambda_{\max}\{ \bm \Psi \bm \Psi^{-1} \} \lambda_{\max}^{1/2} \{\bm K_n \} \shortnorm{\widetilde{\bm R}}_2\\
			& = O_p(\shortnorm{\widetilde{\bm R}}_2)\\
			& = o_p(n^{-1/2}).
		\end{align*}
		Plugging these expressions into \eqref{eqn:glm:root n oracle expression without nu} yields \eqref{eqn:reduced oracle limiting expression}.
	\end{proof}
	
	We rely on the following lemma in our proof of Lemma \ref{lemma:oracle test statistic distribution}.
	\begin{lemma} \label{lemma:intermediate bounds for oracle}
		Under the conditions of Lemma \ref{lemma:oracle test statistic distribution}, the following bounds hold
		\begin{eqnarray}
			\lambda_{\max}\{ \bm{K}_n\} & = & O(1) \label{eqn:intermediate lemma 1}\\
			\lambda_{\max}\{ \bm{K}_n^{1/2}\} & = & O(1) \label{eqn:intermediate lemma 2}\\
			\lambda_{\max}\{ \bm{K}_n^{-1/2}\} & = & O(1) \label{eqn:intermediate lemma 3}\\
			\lambda_{\max}\{ \bm{\Psi}^{-1} \} & = & O(1) \label{eqn:intermediate lemma 4}\\
			\norm{ \bm \Psi^{-1/2} \begin{bsmallmatrix} \Cmat & \bm 0 \end{bsmallmatrix} }_2 & = & O(1) \label{eqn:intermediate lemma 5}\\
			\norm{ \bm \Psi^{1/2} \left( \begin{bsmallmatrix} \Cmat & \bm 0 \end{bsmallmatrix} \Kna^{-1} \begin{bsmallmatrix} \Cmat' \\ \bm 0 \end{bsmallmatrix} \right)^{-1} \bm \Psi^{1/2} - \mathbf{I}_r }_2 & = & O_p\left( \frac{s+m}{\sqrt{n}} \right) \label{eqn:intermediate lemma 6}\\
			\norm{ \bm K_n^{1/2} \Kno^{-1} \bm K_n^{1/2} - \mathbf{I}_{m+s} }_2 & = & O_p\left( \frac{s+m}{\sqrt{n}} \right), \label{eqn:intermediate lemma 7}
		\end{eqnarray}
		where $\Kna = \nabla_{\Mset \cup \Sset}^2 \ell_n(\hat{\bvec}_a^{\oracle} )$ and $\Kno = \nabla_{\Mset \cup \Sset}^2 \ell_n(\hat{\bvec}_0^{\oracle} )$.
	\end{lemma}
	\noindent Note that \eqref{eqn:intermediate lemma 1} - \eqref{eqn:intermediate lemma 4} come from Lemma 5.1 of \cite{Shi2019}. We include them here for ease of reference.
	
	\begin{proof}[Proof of Lemma \ref{lemma:oracle test statistic distribution}]
		Define $\bm \omega_n = - \sqrt{n}  \begin{bsmallmatrix} \Cmat & \bm 0 \end{bsmallmatrix}\bm K_n^{-1} \nabla_{\Mset \cup \Sset} \ell_n(\bvec^*) $ and
		$T_0 = (\bm \omega_n + \sqrt{n} \bm h_n)' \bm \Psi^{-1}\\ (\bm \omega_n + \sqrt{n} \bm h_n)/\phi^*$.
		For ease of notation, we let $T_W = T_W(\hat \bvec_0^{\oracle})$, $T_S = T_S(\hat \bvec_a^{\oracle})$, and $T_L = T_L(\hat \bvec_a^{\oracle}, \hat \bvec_0^{\oracle})$ throughout the proof.
		The test statistics can be written as
		\begin{eqnarray*}
			T_W & = & n(\Cmat \hat{\bvec}_{a, \Mset}^{\oracle} - \bm{t})'
			\left( \begin{bsmallmatrix} \Cmat & \bm 0 \end{bsmallmatrix} \Kna^{-1} \begin{bsmallmatrix} \Cmat' \\ \bm 0 \end{bsmallmatrix} \right)^{-1}(\Cmat \hat{\bvec}_{a, \Mset}^{\oracle} -\bm{t})/\hat{\phi}\\
			T_S & = & n \nabla_{\Mset \cup \Sset} \ell_n(\hat{\bvec}_0^{\oracle})' \Kno^{-1} \nabla_{\Mset \cup \Sset} \ell_n(\hat{\bvec}_0^{\oracle})/\hat{\phi}\\
			T_L & = & -2n(\ell_n(\hat{\bvec}_a^{\oracle}) - \ell_n(\hat{\bvec}_0^{\oracle}) )/\hat{\phi}.
		\end{eqnarray*}
		
		This proof is divided into four parts:
		We will show (1) $T_W = T_0 + o_p(r)$; (2) $T_S = T_0 + o_p(r)$; (3) $T_L = T_0 + o_p(r)$; and (4) for any $T$ satisfying $T = T_0 + o_p(r)$, $\sup_x |\Pr(T \leq x) - \Pr( \chi^2(r,\nu_n) \leq x )| \to 0 $ as $n \to \infty$.\\
		
		\noindent \textbf{Part 1.}
		Lemma \ref{lemma:oracle asymptotics} establishes that
		$$
		\sqrt{n} (\hat \bvec_{a, \Mset \cup \Sset }^{\oracle}  - \bvec_{\Mset \cup \Sset}^*) = -\sqrt{n} \bm K_n^{-1} \nabla_{\Mset \cup \Sset} \ell_n(\bvec^*) + \bm R_a
		$$
		where $\norm{\bm R_a}_2 = o_p(1) $.
		Left-multiplying by $\begin{bsmallmatrix} \Cmat & \bm 0 \end{bsmallmatrix}$, we find
		$$
		\sqrt{n} \begin{bsmallmatrix} \Cmat & \bm 0 \end{bsmallmatrix} (\hat \bvec_{a, \Mset \cup \Sset }^{\oracle}  - \bvec_{\Mset \cup \Sset}^*)
		= \bm \omega_n + \begin{bsmallmatrix} \Cmat & \bm 0 \end{bsmallmatrix}  \bm R_a.
		$$
		We know that $\begin{bsmallmatrix} \Cmat & \bm 0 \end{bsmallmatrix}\bvec_{\Mset \cup \Sset}^* = \Cmat \bvec_{\Mset}^* = \bm t + \bm h_n$ by \ref{assumption:hn order} and see that 
		$\begin{bsmallmatrix} \Cmat & \bm 0 \end{bsmallmatrix} \hat \bvec_{a, \Mset \cup \Sset }^{\oracle} = \Cmat \hat \bvec_{a,\Mset}^{\oracle} $.
		Therefore we can simplify the previous expression to
		$$
		\sqrt{n}(\Cmat \hat \bvec_{a, \Mset \cup \Sset}^{\oracle} - \bm t) 
		= \bm \omega_n + \sqrt{n} \bm h_n + \begin{bsmallmatrix} \Cmat & \bm 0 \end{bsmallmatrix}  \bm R_a.
		$$
		Multiplying by $\bm \Psi^{-1/2}$ and applying \eqref{eqn:intermediate lemma 5} from Lemma \ref{lemma:intermediate bounds for oracle}, we find 
		\begin{equation}
			\sqrt{n}\bm \Psi^{-1/2}(\Cmat \hat \bvec_{a, \Mset \cup \Sset}^{\oracle} - \bm t) 
			= \bm \Psi^{-1/2} \bm \omega_n + \sqrt{n} \bm \Psi^{-1/2}\bm h_n + o_p(1). \label{eqn:TW proof 1}
		\end{equation}
		
		Because the distribution of $y$ belongs to an exponential family, we find that\\
		$
		\E\left[ \norm{ \bm \Psi^{-1/2} \bm \omega_n }_2^2  \right]
		= \tr \left\{ \bm \Psi^{-1/2} \E[ \bm \omega_n \bm \omega_n' ] \bm \Psi^{-1/2} \right\} 
		= \phi^* \tr \left\{  \bm \Psi^{-1/2} \bm \Psi \bm \Psi^{-1/2} \right\}
		= r \phi^*.
		$
		Therefore by Markov's inequality,
		\begin{equation}
			\norm{ \bm \Psi^{-1/2} \bm \omega_n }_2 = O_p( \sqrt{r} ) . \label{eqn:TW proof 2} 
		\end{equation}
		In addition, \ref{assumption:hn order} and \eqref{eqn:intermediate lemma 4} in Lemma \ref{lemma:intermediate bounds for oracle} imply that
		\begin{equation}
			\norm{ \sqrt{n} \bm \Psi^{-1/2} \bm h_n }_2 = O_p(\sqrt{r}). \label{eqn:TW proof 3}
		\end{equation}
		Combining \eqref{eqn:TW proof 1}, \eqref{eqn:TW proof 2}, and \eqref{eqn:TW proof 3}, we have
		\begin{equation}
			\norm{ \sqrt{n}\bm \Psi^{-1/2}(\Cmat \hat \bvec_{a, \Mset \cup \Sset}^{\oracle} - \bm t) }_2 = O_p(\sqrt{r}) \label{eqn:TW proof 4}.
		\end{equation}
		
		Define $T_{W,0} = n( \Cmat \hat \bvec_{0, \Mset}^{\oracle} - \bm t )'\bm \Psi^{-1} ( \Cmat \hat \bvec_{0, \Mset}^{\oracle} -  \bm t )/\hat \phi$.
		By \eqref{eqn:TW proof 4} and \eqref{eqn:intermediate lemma 5} in Lemma \ref{lemma:intermediate bounds for oracle}, we have
		\begin{align}
			\hat \phi | T_W - T_{W,0} |
			& = \norm{ \sqrt{n} \bm \Psi^{-1/2} ( \Cmat \hat \bvec_{0, \Mset}^{\oracle} - \bm t ) }_2^2 \norm{ \bm \Psi^{1/2} \left( \begin{bsmallmatrix} \Cmat & \bm 0 \end{bsmallmatrix} \Kna^{-1} \begin{bsmallmatrix} \Cmat' \\ \bm 0 \end{bsmallmatrix} \right)^{-1} \bm \Psi^{1/2} - \mathbf{I}_r }_2 \nonumber \\
			& =  O_p\left(\frac{r(s+m)}{\sqrt{n} } \right) \nonumber\\
			& = o_p(r) \label{eqn:TW proof 5}
		\end{align}
		since $s+m = o(n^{1/3})$.
		Because $\phi^* > 0$, our assumption that $\hat \phi = \phi^*  + o_p(1)$ implies $1/\hat \phi = O_p(1)$, a result we will use throughout this proof.
		Combining this with \eqref{eqn:TW proof 5}, we see that
		$ T_W = T_{W,0} + o_p(r)$.
		
		Define $T_{W,1} = \norm{ \bm \Psi^{-1/2} \bm \omega_n + \sqrt{n} \bm \Psi^{-1/2} \bm h_n }_2^2/\hat \phi$.
		Taking \eqref{eqn:TW proof 1} and applying the Cauchy-Schwarz inequality, we find
		\begin{align*}
			\hat \phi T_{W,0}
			& = \norm{ \bm \Psi^{-1/2} \bm \omega_n + \sqrt{n} \bm \Psi^{-1/2} \bm h_n + o_p(1) }_2^2 \\
			& = \norm{ \bm \Psi^{-1/2} \bm \omega_n + \sqrt{n} \bm \Psi^{-1/2} \bm h_n}_2^2 +  o_p(1) + o_p\left(\bm \Psi^{-1/2} \bm \omega_n + \sqrt{n} \bm \Psi^{-1/2} \bm h_n \right)\\
			& = \norm{ \bm \Psi^{-1/2} \bm \omega_n + \sqrt{n} \bm \Psi^{-1/2} \bm h_n}_2^2 +  o_p(1) + O_p(\sqrt{r} ) \\
			& = \hat \phi T_{W,1} + o_p(r),
		\end{align*}
		where the third equality follows from \eqref{eqn:TW proof 2} and \eqref{eqn:TW proof 3}.
		Since  $1/\hat \phi = O_p(1)$, this implies $ T_{W, 0} = T_{W,1} + o_p(r) $.
		
		We see that
		\begin{equation}
			|T_{W,1} - T_0 | = \frac{|\phi^* - \hat \phi|}{\hat \phi \phi^*} \norm{\bm \Psi^{-1/2} \bm \omega_n + \sqrt{n} \bm \Psi^{-1/2} \bm h_n}_2^2 = o_p(r) \label{eqn:TW proof 6}
		\end{equation}
		due to \eqref{eqn:TW proof 2}, \eqref{eqn:TW proof 3}, the fact that $1/\hat \phi = O_p(1)$, and our assumption that $|\hat \phi - \phi^*| = o_p(1)$.
		Therefore we have $T_W = T_{W,0} + o_p(r) = T_{W,1} + o_p(r) =T_{0} + o_p(r)$, completing this portion of the proof.\\
		
		\noindent \textbf{Part 2.}
		In \eqref{eqn:glm:mvt for MuS gradient} in the proof of Lemma \ref{lemma:oracle asymptotics}, we established that 
		$$
		\nabla_{\Mset \cup \Sset} \ell_n(\hat \bvec_0^{\oracle}) 
		= \nabla_{\Mset \cup \Sset} \ell_n(\bvec^*) + \bm K_n (\hat \bvec_{0,\Mset \cup \Sset}^{\oracle} - \bvec^*_{\Mset \cup \Sset} ) + \bm R_1
		$$
		where $\norm{\bm R_1}_2 = o_p(n^{-1/2})$.
		Multiplying both sides by $\sqrt{n}$, we find 
		\begin{equation}
			\sqrt{n} \nabla_{\Mset \cup \Sset} \ell_n(\hat \bvec_0^{\oracle}) 
			= \sqrt{n} \nabla_{\Mset \cup \Sset} \ell_n(\bvec^*) + \sqrt{n} \bm K_n (\hat \bvec_{0,\Mset \cup \Sset}^{\oracle} - \bvec^*_{\Mset \cup \Sset} ) + o_p(1). \label{eqn:TS proof 1}
		\end{equation}
		Lemma \ref{lemma:oracle asymptotics} and \eqref{eqn:Kn term expansion for reduced estimator} in the proof of Lemma \ref{lemma:oracle asymptotics} provide that
		\begin{align*}
			\sqrt{n}(\hat{\bvec}^{\oracle}_{0, \Mset \cup \Sset} - \bvec_{\Mset \cup \Sset}^* )
			= & \sqrt{n} \bm{K}_n^{-1/2} (\bm P_n - \bm I) \bm {K}_n^{-1/2} \nabla_{\Mset \cup \Sset} \ell_n(\bvec^*) \notag \\
			& -  \sqrt{n} \bm K_n^{-1/2} \bm P_n \bm K_n^{1/2} \begin{bmatrix} \Cmat'(\Cmat \Cmat')^{-1} \bm h_n \\ \bm 0 \end{bmatrix}
			+ \bm R_2\\
			= & \sqrt{n} \bm{K}_n^{-1/2} (\bm P_n - \bm I) \bm {K}_n^{-1/2} \nabla_{\Mset \cup \Sset} \ell_n(\bvec^*) \notag \\
			& -  \sqrt{n} \bm K_n^{-1} \begin{bmatrix} \Cmat'\\ \bm 0 \end{bmatrix} \bm \Psi^{-1} \bm h_n 
			+ \bm R_2
		\end{align*}
		where $\norm{\bm R_2}_2 = o_p(1) $.
		Due to \eqref{eqn:intermediate lemma 1}, we have
		$ \norm{ \bm K_n \bm R_2 }_2 \leq \norm{ \bm K_n }_2 \norm{ \bm R_2 }_2 = o_p(1)$.
		Therefore multiplying the previous expression by $\bm K_n$ yields
		\begin{align}
			\sqrt{n}\bm K_n(\hat{\bvec}^{\oracle}_{0, \Mset \cup \Sset} - \bvec_{\Mset \cup \Sset}^* )
			= & \sqrt{n} \bm{K}_n^{1/2} (\bm P_n - \bm I) \bm {K}_n^{-1/2} \nabla_{\Mset \cup \Sset} \ell_n(\bvec^*) \nonumber \\ 
			& - \sqrt{n} \begin{bmatrix} \Cmat'\\ \bm 0 \end{bmatrix} \bm \Psi^{-1} \bm h_n  
			+ o_p(1). \label{eqn:TS proof 2}
		\end{align}
		
		Plugging \eqref{eqn:TS proof 2} into \eqref{eqn:TS proof 1}, we find
		\begin{equation*}
			\sqrt{n} \nabla_{\Mset \cup \Sset} \ell_n(\hat \bvec_0^{\oracle})
			= \sqrt{n} \bm{K}_n^{1/2} \bm P_n \bm {K}_n^{-1/2} \nabla_{\Mset \cup \Sset} \ell_n(\bvec^*)
			- \sqrt{n} \begin{bmatrix} \Cmat'\\ \bm 0 \end{bmatrix} \bm \Psi^{-1} \bm h_n  
			+ o_p(1)
		\end{equation*}
		We see that $\norm{\bm K_n^{-1/2}}_2 = \lambda_{\max}\{ \bm K_n^{-1/2} \} = O_p(1)$ by \eqref{eqn:intermediate lemma 3}. 
		Therefore multiplying both sides of the previous expression by $\bm K_n^{-1/2}$ yields
		\begin{align}
			\sqrt{n} \bm K_n^{-1/2} \nabla_{\Mset \cup \Sset} \ell_n(\hat \bvec_0^{\oracle})
			= & \sqrt{n} \bm P_n \bm {K}_n^{-1/2} \nabla_{\Mset \cup \Sset} \ell_n(\bvec^*) \nonumber \\
			& - \sqrt{n} \bm K_n^{-1/2} \begin{bmatrix} \Cmat'\\ \bm 0 \end{bmatrix} \bm \Psi^{-1} \bm h_n  
			+ o_p(1). \label{eqn:TS proof 3}
		\end{align}
		Since the distribution of $y$ belongs to an exponential family and $\bm P_n$ is a projection matrix, we find
		$
		\E \left[ \norm{\sqrt{n} \bm P_n \bm {K}_n^{-1/2} \nabla_{\Mset \cup \Sset} \ell_n(\bvec^*)}_2^2 \right]
		= \tr \left\{ \bm P_n \bm {K}_n^{-1/2} \bm K_n  \bm {K}_n^{-1/2} \bm P_n  \right\}
		= \tr \{ \bm P_n \} = \tr \{ \bm \Psi \bm \Psi^{-1} \} = \tr \{ \mathbf{I}_r \} = r.
		$
		Therefore by Markov's inequality
		\begin{equation}
			\norm{\sqrt{n} \bm P_n \bm {K}_n^{-1/2} \nabla_{\Mset \cup \Sset} \ell_n(\bvec^*)}_2 = O_p(\sqrt{r}). \label{eqn:TS proof 4a}
		\end{equation}
		In addition, we see that
		\begin{align}
			\norm{\sqrt{n} \bm K_n^{-1/2} \begin{bmatrix} \Cmat'\\ \bm 0 \end{bmatrix} \bm \Psi^{-1} \bm h_n}_2
			& \leq \norm{\bm K_n^{-1/2}}_2 \norm{\begin{bmatrix} \Cmat'\\ \bm 0 \end{bmatrix} \bm \Psi^{-1/2}}_2\norm{ \sqrt{n}  \bm \Psi^{-1/2} \bm h_n}_2 \nonumber \\
			& = O_p(\sqrt{r}) \label{eqn:TS proof 4b}
		\end{align}
		by \eqref{eqn:intermediate lemma 3}, \eqref{eqn:intermediate lemma 5}, and \eqref{eqn:TW proof 3}.
		Applying these bounds to \eqref{eqn:TS proof 3}, we have
		\begin{equation}
			\norm{\sqrt{n} \bm K_n^{-1/2} \nabla_{\Mset \cup \Sset} \ell_n(\hat \bvec_0^{\oracle})}_2 = O_p(\sqrt{r}). \label{eqn:TS proof 5}
		\end{equation}
		
		Define $T_{S,0} =  n \nabla_{\Mset \cup \Sset} \ell_n(\hat{\bvec}_0^{\oracle})' \bm K_n^{-1} \nabla_{\Mset \cup \Sset} \ell_n(\hat{\bvec}_0^{\oracle})/\hat{\phi}$.
		By \eqref{eqn:TS proof 5} and \eqref{eqn:intermediate lemma 7}, we have
		\begin{align*}
			\hat \phi | T_S - T_{S,0} |
			& = \norm{\sqrt{n} \bm K_n^{-1/2} \nabla_{\Mset \cup \Sset} \ell_n(\hat \bvec_0^{\oracle})}_2^2 \norm{ \bm K_n^{1/2} \Kno^{-1} \bm K_n^{1/2} - \mathbf{I}_{m+s} }_2 \\
			& = O_p(r) O_p\left( \frac{s+m}{\sqrt n} \right)\\
			& = o_p(r)
		\end{align*}
		since $s+m = o(n^{1/3})$.
		Because $1/\hat \phi = O_p(1)$, this implies $T_S = T_{S,0} + o_p(r)$.
		
		Define $T_{S,1} = \norm{\sqrt{n} \bm P_n \bm {K}_n^{-1/2} \nabla_{\Mset \cup \Sset} \ell_n(\bvec^*) - \sqrt{n} \bm K_n^{-1/2} \begin{bmatrix} \Cmat'\\ \bm 0 \end{bmatrix} \bm \Psi^{-1} \bm h_n }_2^2/\hat \phi$.
		Taking \eqref{eqn:TS proof 3} and applying the Cauchy-Schwarz inequality, we find
		\begin{align}
			\hat \phi T_{S,0}
			& = \norm{\sqrt{n} \bm K_n^{-1/2} \nabla_{\Mset \cup \Sset} \ell_n(\hat \bvec_0^{\oracle})}_2^2 \nonumber \\
			& = \norm{\sqrt{n} \bm P_n \bm {K}_n^{-1/2} \nabla_{\Mset \cup \Sset} \ell_n(\bvec^*) - \sqrt{n} \bm K_n^{-1/2} \begin{bmatrix} \Cmat'\\ \bm 0 \end{bmatrix} \bm \Psi^{-1} \bm h_n + o_p(1)}_2^2 \nonumber \\
			& = \hat \phi T_{S,1} + o_p(1) + O_p(\sqrt{r}) \nonumber\\
			& = \hat \phi T_{S,1} + o_p(r), \label{eqn:TS proof 6}
		\end{align}
		with the penultimate equality following from \eqref{eqn:TS proof 4a} and \eqref{eqn:TS proof 4b}.
		This implies $T_{S,0} = T_{S,1} + o_p(r)$ since $1/\hat \phi = O_p(1)$.
		
		We further see that
		\begin{align}
			T_{S,1}
			& = \norm{ \bm K_n^{-1/2} \begin{bmatrix} \Cmat'\\ \bm 0 \end{bmatrix} \bm \Psi^{-1/2} \left( \bm \Psi^{-1/2}\bm \omega_n + \sqrt{n} \bm \Psi^{-1/2} \bm h_n  \right) }_2^2/\hat \phi \nonumber \\
			& = \left( \bm \Psi^{-1/2}\bm \omega_n + \sqrt{n} \bm \Psi^{-1/2} \bm h_n  \right)' \bm \Psi^{-1/2} \bm \Psi \bm \Psi^{-1/2} \left( \bm \Psi^{-1/2}\bm \omega_n + \sqrt{n} \bm \Psi^{-1/2} \bm h_n  \right) \nonumber\\
			& =  \norm{ \bm \Psi^{-1/2}\bm \omega_n + \sqrt{n} \bm \Psi^{-1/2} \bm h_n }_2^2/\hat \phi. \label{eqn:TS proof 7}
		\end{align}
		Thus 
		\begin{equation}
			|T_{S,1} - T_0| = \frac{|\phi^* - \hat \phi|}{\hat \phi \phi^*} \norm{ \bm \Psi^{-1/2}\bm \omega_n + \sqrt{n} \bm \Psi^{-1/2} \bm h_n }_2^2 = o_p(r) \label{eqn:TS proof 8}
		\end{equation}
		by the same arguments we used for \eqref{eqn:TW proof 6}.
		All together, we have $T_S = T_{S,0} + o_p(r) = T_{S,1} + o_p(r) = T_{0} + o_p(r)$, completing this portion of the proof.\\
		
		\noindent \textbf{Part 3.}
		Combining \eqref{eqn:reduced oracle limiting expression} and \eqref{eqn:full oracle limiting expression} from Lemma \ref{lemma:oracle asymptotics} and \eqref{eqn:Kn term expansion for reduced estimator} from the proof of Lemma \ref{lemma:oracle asymptotics}, we have
		\begin{align}
			\sqrt{n}( \hat \bvec_{a,\Mset \cup \Sset}^{\oracle} - \hat \bvec_{0,\Mset \cup \Sset}^{\oracle})
			= & -\sqrt{n} \bm K_n^{-1/2} \bm P_n \bm K_n^{-1/2} \nabla_{\Mset \cup \Sset} \ell_n(\bvec^*) \nonumber \\
			& + \sqrt{n} \bm K_n^{-1} \begin{bmatrix} \Cmat'\\ \bm 0 \end{bmatrix} \bm \Psi^{-1} \bm h_n + o_p(1). \label{eqn:TL proof 1}
		\end{align}
		Applying the bounds from \eqref{eqn:intermediate lemma 3}, \eqref{eqn:TS proof 4a}, and \eqref{eqn:TS proof 4b} to the previous expression, we find
		\begin{equation}
			\norm{ \sqrt{n}( \hat \bvec_{a,\Mset \cup \Sset}^{\oracle} - \hat \bvec_{0,\Mset \cup \Sset}^{\oracle})  }_2 = O_p(\sqrt{r}) .\label{eqn:TL proof 2}
		\end{equation}
		
		We know that $\hat \bvec_{a, (\Mset \cup \Sset)^c}^{\oracle} = \hat \bvec_{0, (\Mset \cup \Sset)^c}^{\oracle} = \bm 0$. 
		Taking a Taylor series expansion and simplifying yields
		\begin{align}
			\ell_n(\hat \bvec_0^{\oracle}) - \ell_n(\hat \bvec_a^{\oracle})
			& = (\hat \bvec_{0, \Mset \cup \Sset}^{\oracle}  - \hat \bvec_{a, \Mset \cup \Sset}^{\oracle})'\nabla_{\Mset \cup \Sset} \ell_n(\hat \bvec_a^{\oracle} ) \nonumber \\
			& \hspace{0.2cm} + \frac{1}{2} (\hat \bvec_{0, \Mset \cup \Sset}^{\oracle}  - \hat \bvec_{a, \Mset \cup \Sset}^{\oracle})'\nabla_{\Mset \cup \Sset}^2 \ell_n(\bar \bvec )  (\hat \bvec_{0, \Mset \cup \Sset}^{\oracle}  - \hat \bvec_{a, \Mset \cup \Sset}^{\oracle})  \nonumber \\
			& = (\hat \bvec_{0, \Mset \cup \Sset}^{\oracle}  - \hat \bvec_{a, \Mset \cup \Sset}^{\oracle})'\nabla_{\Mset \cup \Sset} \ell_n(\hat \bvec_a^{\oracle} ) \nonumber \\
			& \hspace{0.2cm} + \frac{1}{2} (\hat \bvec_{0, \Mset \cup \Sset}^{\oracle}  - \hat \bvec_{a, \Mset \cup \Sset}^{\oracle})'\nabla_{\Mset \cup \Sset}^2 \ell_n(\hat \bvec_a^{\oracle})  (\hat \bvec_{0, \Mset \cup \Sset}^{\oracle}  - \hat \bvec_{a, \Mset \cup \Sset}^{\oracle}) \nonumber \\
			& \hspace{0.2cm} + \frac{1}{2} (\hat \bvec_{0, \Mset \cup \Sset}^{\oracle}  - \hat \bvec_{a, \Mset \cup \Sset}^{\oracle})' \bm R, \label{eqn:TL proof 3}
		\end{align}
		where $\bar \bvec$ lies on the line segment between $\hat \bvec_a^{\oracle}$ and $\hat \bvec_0^{\oracle}$ and
		$\bm R = (\nabla_{\Mset \cup \Sset}^2 \ell_n(\bar \bvec) - \nabla_{\Mset \cup \Sset}^2 \ell_n(\hat \bvec_a^{\oracle}))  (\hat \bvec_{0, \Mset \cup \Sset}^{\oracle}  - \hat \bvec_{a, \Mset \cup \Sset}^{\oracle})$.
		Note that this expression is similar to \eqref{eqn:glm:mvt for gradient at oracle} and \eqref{eqn:glm:mvt for MuS gradient} in the proof of Lemma \ref{lemma:oracle asymptotics}.
		By \eqref{eqn:reduced oracle consistency} and \eqref{eqn:full oracle consistency} in Lemma \ref{lemma:oracle asymptotics}, we have $\Pr(\hat \bvec_0^{\oracle} \in \mathcal{N}_0) \to 1$ and $\Pr(\hat \bvec_a^{\oracle} \in \mathcal{N}_0) \to 1$.
		Therefore $\Pr(\bar \bvec \in \mathcal{N}_0) \to 1$. 
		Applying the mean value theorem as in \eqref{eqn:glm:Rj bound}, we find
		\begin{align*}
			\norm{\bm R}_{\max}
			& \leq \max_{j \in \Mset \cup \Sset} \lambda_{\max}\left\{ \frac{1}{n} \xmat_{\Mset \cup \Sset}' \diag\{ |\xmat_j| \circ |\bm b'''(\xmat \bar {\bar \bvec}^{(j)} ) | \}\xmat_{\Mset \cup \Sset} \right\}\\
			& \hspace{0.2cm} \times \norm{\hat{\bvec}^{\oracle}_{0, \Mset \cup \Sset} - \hat{\bvec}^{\oracle}_{a, \Mset \cup \Sset}}_2^2 
		\end{align*}
		where each $\bar{\bar \bvec}^{(j)}$ lies on the line segment between $\hat \bvec_a^{\oracle}$ and $\bar \bvec$.
		Therefore by \ref{assumption:glm:third derivative bound} and \eqref{eqn:TL proof 2},
		$$
		\norm{\bm R}_{\max} 
		= O(1) \norm{\hat{\bvec}^{\oracle}_{0, \Mset \cup \Sset} - \hat{\bvec}^{\oracle}_{a, \Mset \cup \Sset}}_2^2 
		= O_p\left(\frac{r}{n}\right).
		$$	
		As a consequence, we see that
		\begin{align*}
			\norm{(\hat \bvec_{0, \Mset \cup \Sset}^{\oracle}  - \hat \bvec_{a, \Mset \cup \Sset}^{\oracle})' \bm R}_2
			& \leq \norm{(\hat \bvec_{0, \Mset \cup \Sset}^{\oracle}  - \hat \bvec_{a, \Mset \cup \Sset}^{\oracle})}_2 \sqrt{s+m} \norm{\bm R}_{\max}\\
			& = O_p\left( \sqrt{ \frac{r}{n} }\right)O_p\left( \frac{r\sqrt{s+m}}{n} \right)\\
			& = o_p\left(\frac{\sqrt{r}}{n}\right)
		\end{align*}
		since $r \leq s + m$ and $s+m = o(n^{1/3})$.
		Using a similar argument, one can show that
		$$
		\norm{ (\hat \bvec_{0, \Mset \cup \Sset}^{\oracle}  - \hat \bvec_{a, \Mset \cup \Sset}^{\oracle})'(\nabla_{\Mset \cup \Sset}^2 \ell_n(\hat \bvec_a^{\oracle}) - \bm K_n )  (\hat \bvec_{0, \Mset \cup \Sset}^{\oracle}  - \hat \bvec_{a, \Mset \cup \Sset}^{\oracle}) }_2
		= o_p\left(\frac{\sqrt{r}}{n}\right).
		$$
		As such, \eqref{eqn:TL proof 3} simplifies to
		\begin{align*}
			\ell_n(\hat \bvec_0^{\oracle}) - \ell_n(\hat \bvec_a^{\oracle})
			& = (\hat \bvec_{0, \Mset \cup \Sset}^{\oracle}  - \hat \bvec_{a, \Mset \cup \Sset}^{\oracle})'\nabla_{\Mset \cup \Sset} \ell_n(\hat \bvec_a^{\oracle} ) \nonumber \\
			& \hspace{0.2cm} + \frac{1}{2} (\hat \bvec_{0, \Mset \cup \Sset}^{\oracle}  - \hat \bvec_{a, \Mset \cup \Sset}^{\oracle})' \bm K_n  (\hat \bvec_{0, \Mset \cup \Sset}^{\oracle}  - \hat \bvec_{a, \Mset \cup \Sset}^{\oracle}) \nonumber \\
			& \hspace{0.2cm} + o_p\left(\frac{\sqrt{r}}{n}\right).
		\end{align*}
		By \ref{assumption:full oracle unique}, $\nabla_{\Mset \cup \Sset} \ell_n(\hat \bvec_a^{\oracle} ) = \bm 0$. Therefore the previous expression simplifies further to
		\begin{equation}
			\ell_n(\hat \bvec_0^{\oracle}) - \ell_n(\hat \bvec_a^{\oracle})
			= \frac{1}{2} (\hat \bvec_{0, \Mset \cup \Sset}^{\oracle}  - \hat \bvec_{a, \Mset \cup \Sset}^{\oracle})' \bm K_n  (\hat \bvec_{0, \Mset \cup \Sset}^{\oracle}  - \hat \bvec_{a, \Mset \cup \Sset}^{\oracle})
			+  o_p\left(\frac{\sqrt{r}}{n}\right). \label{eqn:TL proof 4}
		\end{equation}
		
		Multiplying \eqref{eqn:TL proof 1} by $\bm K_n^{1/2}$ and simplifying, we find
		\begin{align}
			\norm{ \sqrt{n} \bm K_n^{1/2}( \hat \bvec_{a,\Mset \cup \Sset}^{\oracle} - \hat \bvec_{0,\Mset \cup \Sset}^{\oracle})  }_2^2
			& = \bigg \lVert -\sqrt{n} \bm P_n \bm K_n^{-1/2} \nabla_{\Mset \cup \Sset} \ell_n(\bvec^*) \nonumber \\
			& \hspace{0.75cm} + \sqrt{n} \bm K_n^{-1/2} \begin{bmatrix} \Cmat'\\ \bm 0 \end{bmatrix} \bm \Psi^{-1} \bm h_n + o_p(1) \bigg \rVert_2^2 \nonumber \\
			& = \bigg \lVert -\sqrt{n} \bm P_n \bm K_n^{-1/2} \nabla_{\Mset \cup \Sset} \ell_n(\bvec^*) \nonumber \\
			& \hspace{0.75cm} + \sqrt{n} \bm K_n^{-1/2} \begin{bmatrix} \Cmat'\\ \bm 0 \end{bmatrix} \bm \Psi^{-1} \bm h_n \bigg \rVert_2^2 + o_p(r) \nonumber\\
			& = \norm{ \bm \Psi^{-1/2} \bm \omega_n + \sqrt{n} \bm \Psi^{-1/2} \bm h_n }_2^2 + o_p(r), \label{eqn:TL proof 5}
		\end{align}
		where the second and third equalities follow from the same arguments used in \eqref{eqn:TS proof 6} and \eqref{eqn:TS proof 7}, respectively.
		Combining \eqref{eqn:TL proof 4} and \eqref{eqn:TL proof 5}, we find
		\begin{align*}
			\hat \phi T_L
			& = 2n( \ell_n(\hat \bvec_0^{\oracle}) - \ell_n(\hat \bvec_a^{\oracle}) )\\
			& = \norm{ \bm \Psi^{-1/2} \bm \omega_n + \sqrt{n} \bm \Psi^{-1/2} \bm h_n }_2^2 + o_p(r) + o_p(\sqrt{r})\\
			& = \norm{ \bm \Psi^{-1/2} \bm \omega_n + \sqrt{n} \bm \Psi^{-1/2} \bm h_n }_2^2 + o_p(r)\\
			& = \hat \phi T_{L,1} + o_p(r)
		\end{align*}
		where $T_{L,1} = \norm{ \bm \Psi^{-1/2} \bm \omega_n + \sqrt{n} \bm \Psi^{-1/2} \bm h_n }_2^2/\hat \phi$.
		Since $1/\hat{\phi} = O_p(1)$, this implies 
		$ T_L = T_{L,1} + o_p(r) $.
		Moreover, we find
		\begin{equation*}
			|T_{L,1} - T_0|
			= \frac{\left| \phi^* - \hat \phi \right|}{\hat \phi \phi^*} \norm{ \bm \Psi^{-1/2} \bm \omega_n + \sqrt{n} \bm \Psi^{-1/2} \bm h_n }_2^2 = o_p(r),
		\end{equation*}
		by the same argument used in \eqref{eqn:TW proof 6}.
		Thus we have $T_L = T_{L,1} + o_p(r) = T_0 + o_p(r)$, completing this portion of the proof.\\
		
		\noindent \textbf{Part 4.}
		Define $\bm \xi_i = \frac{1}{\sqrt{n \phi^*}} \bm \Psi^{-1/2} \begin{bmatrix} \Cmat & \bm 0 \end{bmatrix} \bm K_n^{-1} \xmat_{i,\Mset \cup \Sset}'(y_i  - b'(\xvec_i'\bvec^*))$ for $i = 1, \ldots, n$.
		Because the observations are independent, the $\bm \xi_i$ are independent.
		In addition, one can show that $\E[\bm \xi_i] = \bm 0$ for all $i$ and $\sum_{i = 1}^{n}\Cov(\bm \xi_i) = \mathbf{I}_r$.
		As such, we can apply Lemma S.6 from \cite{Shi2019} (a special case of Theorem 1 from \cite{Bentkus2005}) and conclude that
		\begin{equation}
			\sup_{\mathcal{C}}\left| \Pr\left( \sum_{i=1}^n \bm \xi_i \in \mathcal{C} \right) - \Pr(\bm Z \in \mathcal{C}) \right|
			\leq c_0 r^{1/4} \sum_{i=1}^n \E\left[ \norm{ \bm \xi_i }_2^3\right] \label{eqn:sup C prob deviation}
		\end{equation}
		where $c_0$ is a constant, $\bm Z \sim N(\bm 0, \mathbf{I}_r)$, and the supremum is taken over all convex sets $\mathcal C$ in $\reals^r$. 
		
		Next, we will show that $r^{1/4} \sum_{i=1}^n \E\left[ \norm{ \bm \xi_i }_2^3\right] \to 0$.
		Condition \ref{assumption:glm:expectation bound for chernoff} implies that
		\begin{align*}
			\max_{1 \leq i \leq n} \E \left[ \frac{| y_i  - b'(\xvec_i'\bvec^*) |^3}{6 M}\right]
			&\leq \max_{1 \leq i \leq n} \E \left[ \exp\left( \frac{| y_i  - b'(\xvec_i'\bvec^*) |}{M} \right) - 1 - \frac{| y_i  - b'(\xvec_i'\bvec^*) |}{M} \right]M^2\\
			&\leq \frac{v_0}{2}
		\end{align*}
		for all $n$.
		By extension, $\max_{1 \leq i \leq n} \E \left[| y_i  - b'(\xvec_i'\bvec^*) |^3\right] = O(1)$.
		In addition, we see that
		$
		\norm{ \bm \Psi^{-1/2} \begin{bsmallmatrix} \Cmat & \bm 0 \end{bsmallmatrix} \bm K_n^{-1/2} }_2
		= \lambda_{\max}^{1/2} \left\{  \bm \Psi^{-1/2} \begin{bsmallmatrix} \Cmat & \bm 0 \end{bsmallmatrix} \bm K_n^{-1} \begin{bsmallmatrix} \Cmat' \\ \bm 0' \end{bsmallmatrix} \bm \Psi^{-1/2} \right\}
		= \lambda_{\max}\{ \mathbf{I}_r \} 
		= 1
		$.
		Applying these findings, we derive
		\begin{align*}
			r^{1/4} \sum_{i=1}^n \E\left[ \norm{ \bm \xi_i }_2^3\right]
			& = \frac{r^{1/4}}{(n \phi^*)^{3/2}} \sum_{i=1}^n \E \left[ \norm{ \bm \Psi^{-1/2} \begin{bmatrix} \Cmat & \bm 0 \end{bmatrix} \bm K_n^{-1} \xmat_{i,\Mset \cup \Sset}'(y_i  - b'(\xvec_i'\bvec^*)) }_2^3 \right]\\
			& = \frac{r^{1/4}}{(n \phi^*)^{3/2}} \sum_{i=1}^n \norm{ \bm \Psi^{-1/2} \begin{bmatrix} \Cmat & \bm 0 \end{bmatrix} \bm K_n^{-1} \xmat_{i,\Mset \cup \Sset}'}_2^3 \E\left[ |y_i  - b'(\xvec_i'\bvec^*)|^3 \right] \\
			& = O(1) \frac{r^{1/4}}{(n \phi^*)^{3/2}} \sum_{i=1}^n \norm{ \bm \Psi^{-1/2} \begin{bmatrix} \Cmat & \bm 0 \end{bmatrix} \bm K_n^{-1} \xmat_{i,\Mset \cup \Sset}'}_2^3 \\
			& \leq O(1) \frac{r^{1/4}}{(n \phi^*)^{3/2}} \sum_{i=1}^n \norm{ \bm \Psi^{-1/2} \begin{bmatrix} \Cmat & \bm 0 \end{bmatrix} \bm K_n^{-1/2}}_2^3 \norm{\bm K_n^{-1/2} \xmat_{i,\Mset \cup \Sset}'}_2^3 \\
			& = O(1)  \frac{r^{1/4}}{(n \phi^*)^{3/2}} \sum_{i=1}^n (\xmat_{i, \Mset \cup \Sset}\bm K_n^{-1} \xmat_{i, \Mset \cup \Sset}' )^{3/2}\\
			& = o(1)
		\end{align*}
		where the last equality follows from \ref{assumption:lyapunov condition}.
		Returning to \eqref{eqn:sup C prob deviation}, we now have
		\begin{equation*}
			\sup_{\mathcal{C}}\left| \Pr\left( \sum_{i=1}^n \bm \xi_i \in \mathcal{C} \right) - \Pr(\bm Z \in \mathcal{C}) \right|
			\to 0 \text{ \; as \; } n \to \infty.
		\end{equation*}
		
		For each $x \in \reals$, define the set $\mathcal{C}_x = \{ \bm z \in \reals^r : \norm{\bm z + \sqrt{\frac{n}{\phi^*}}  \bm \Psi^{-1/2} \bm h_n }_2^2 \leq x \}$.
		Because each set $\mathcal{C}_x$ is convex, the previous expression implies that
		\begin{equation*}
			\sup_{x}\left| \Pr\left( \sum_{i=1}^n \bm \xi_i \in \mathcal{C}_x \right) - \Pr(\bm Z \in \mathcal{C}_x) \right|
			\to 0 \text{ \; as \; } n \to \infty.
		\end{equation*}
		We find that
		$
		\frac{1}{\sqrt{\phi^*}} \bm \Psi^{-1/2} \bm \omega_n
		= \frac{1}{\sqrt{n \phi^*}} \bm \Psi^{-1/2} \begin{bmatrix} \Cmat & \bm 0 \end{bmatrix} \bm K_n^{-1} \xmat_{\Mset \cup \Sset}'(\yvec - \bm \mu(\xmat \bvec))
		= \sum_{i=1}^n \bm \xi_i.
		$
		Since $T_0 = \norm{\frac{1}{\sqrt{\phi^*}} \bm \Psi^{-1/2} \bm \omega_n + \sqrt{\frac{n}{\phi^*}} \bm \Psi^{-1/2} \bm h_n }_2^2$, we see that $\sum_{i=1}^n \bm \xi_i \in \mathcal{C}_x $ if and only if $T_0 \leq x$.
		In addition, we know from the definition of a non-central chi-square distribution that $\Pr(\bm Z \in \mathcal{C}_x) = \Pr( \chi^2(r, \nu_n) \leq x)$ for all $x \in \reals$, where $\nu_n = n \bm h_n'\bm \Psi^{-1} \bm h_n/\phi^*$.
		Therefore we conclude that
		\begin{equation}
			\sup_{x}\left| \Pr\left( T_0 \leq x \right) - \Pr( \chi^2(r, \nu_n) \leq x ) \right|
			\to 0 \text{ \; as \; } n \to \infty. \label{eqn:T0 test statistic convergence}
		\end{equation}
		
		For any test statistic $T$ satisfying $T = T_0 + o_p(r)$, \eqref{eqn:T0 test statistic convergence} implies that for all $x \in \reals$, $\epsilon > 0$,
		\begin{align*}
			\Pr( \chi^2(r, \nu_n) \leq x - \epsilon r ) + o(1)
			& \leq \Pr(T_0 \leq x - \epsilon r) + o(1) \\
			& \leq \Pr(T \leq x) \\
			& \leq \Pr(T_0 \leq x + \epsilon r) + o(1)
			\leq \Pr( \chi^2(r, \nu_n) \leq x + \epsilon r ) + o(1)	.
		\end{align*}
		Lemma S.7 of \cite{Shi2019} provides
		$$
		\lim_{\epsilon \to 0^+} \limsup_n | \Pr( \chi^2(r, \nu_n) \leq x - \epsilon r ) - \Pr( \chi^2(r, \nu_n) \leq x + \epsilon r ) | = 0.
		$$
		Combining this with the previous expression yields
		$$
		\sup_{x}\left| \Pr\left( T \leq x \right) - \Pr( \chi^2(r, \nu_n) \leq x ) \right|
		\to 0 \text{ \; as \; } n \to \infty.
		$$
		Since $T_W = T_0 + o_p(r)$, $T_S= T_0 + o_p(r)$, and $T_L = T_0 + o_p(r)$, this completes the proof.
		
	\end{proof}
	
	\subsection{Proofs of asymptotic results for LLA solutions}
	\begin{proof}[Proof of Theorem \ref{thm:two step LLA asymptotics}]
		We will focus on proving the results for $\hat{\bvec}_0^{(2)}$ as the proofs for $\hat{\bvec}_a^{(2)}$ follow the same lines.
		
		Let $\epsilon > 0$.
		Corollary \ref{cor:glm lla convergence prob with lasso} implies that there exists $N_1 \in \mathbb{N}$ such that $\Pr( \hat{\bvec}^{(2)}_0\neq \hat{\bvec}^{\oracle}_0 ) < \epsilon/2$ for all $n > N_1$.
		Likewise, \eqref{eqn:reduced oracle consistency} implies that there exist $M > 0$ and $N_2 \in \mathbb{N}$ such that
		$\Pr\left(\norm{ \hat \bvec_{0,\Mset \cup \Sset}^{\oracle} - \bvec^*_{\Mset \cup \Sset} }_2 > M \sqrt{ \frac{s+m-r}{n} }  \right) < \epsilon/2$ for all $n > N_2$.
		Applying the union bound, we find
		\begin{align*}
			\Pr\bigg(  & \norm{ \hat \bvec_{0,\Mset \cup \Sset}^{(2)} - \bvec^*_{\Mset \cup \Sset} }_2 \leq M \sqrt{ \frac{s+m-r}{n} } \bigg)\\
			& \geq \Pr\left( \left\{ \norm{ \hat \bvec_{0,\Mset \cup \Sset}^{\oracle} - \bvec^*_{\Mset \cup \Sset} }_2 \leq M \sqrt{ \frac{s+m-r}{n} } \right\} \cap \left\{ \hat{\bvec}^{(2)}_0 = \hat{\bvec}^{\oracle}_0 \right\} \right)\\
			& \geq 1 - \Pr\left( \norm{ \hat \bvec_{0,\Mset \cup \Sset}^{\oracle} - \bvec^*_{\Mset \cup \Sset} }_2 > M \sqrt{ \frac{s+m-r}{n} } \right) - \Pr(\hat{\bvec}^{(2)}_0 \neq \hat{\bvec}^{\oracle}_0)\\
			& > 1 - \epsilon
		\end{align*}
		for all $n > \max\{N_1, N_2\}$.
		Therefore \eqref{eqn:reduced lla consistency} holds by definition.
		Using this same strategy of leveraging Corollary \ref{cor:glm lla convergence prob with lasso} and applying the union bound, we can show that $\hat{\bvec}^{\oracle}_{0,(\Mset \cup \Sset)^c} = \bm 0$ implies $\Pr( \hat{\bvec}^{(2)}_{0, (\Mset \cup \Sset)^c} = \bm 0 ) \to 1$ as $n \to \infty$ and that \eqref{eqn:reduced oracle limiting expression} implies \eqref{eqn:reduced lla limiting expression}.
	\end{proof}

	\begin{proof}[Proof of Theorem \ref{thm:lla test statistic distribution}]
		We will focus on the proof for the partial penalized Wald test statistic as the same argument applies to all three partial penalized tests.
		Let $T^{(2)} = T_W(\hat \bvec_a^{(2)})$, the partial penalized Wald test statistic evaluated at the two-step LLA estimator, and $T^{\oracle} = T_W(\hat \bvec_a^{\oracle})$, the Wald test statistic evaluated at the oracle estimator. 
		
		Let $\epsilon > 0$.
		Corollary \ref{cor:glm lla convergence prob with lasso} implies that $\exists_{N \in \mathbb{N}}$ such that $\Pr(\hat \bvec_a^{(2)} \neq \hat \bvec_a^{\oracle}) < \epsilon$ for all $n > N$.
		Applying the union bound, we find that for all $n > N$ and all $x \in \reals$ ,
		\begin{align*}
			\Pr(T^{(2)} \leq x ) 
			& \geq \Pr( \{\hat \bvec_a^{(2)} = \hat \bvec_a^{\oracle}\} \cap \{T^{\oracle} \leq x\} ) \\
			& \geq 1 - \Pr(\hat \bvec_a^{(2)} \neq \hat \bvec_a^{\oracle}) - \Pr( T^{\oracle} > x ) \\
			& > \Pr( T^{\oracle} \leq x ) - \epsilon
		\end{align*}
		Likewise, we find that for all $n > N$ and all $x \in \reals$,
		$ \Pr(T^{(2)} > x ) > \Pr( T^{\oracle} > x ) - \epsilon$,
		which implies
		$\Pr( T^{\oracle} \leq x ) > \Pr(T^{(2)} \leq x ) - \epsilon $.
		Therefore we have
		$
		| \Pr(T^{(2)} \leq x ) - \Pr( T^{\oracle} \leq x ) | < \epsilon
		$
		for all $n > N$ and all $x \in \reals$.
		In other words,
		$x \mapsto  \Pr(T^{(2)} \leq x ) - \Pr( T^{\oracle} \leq x )$ converges uniformly to $0$.
		Since $x \mapsto \Pr(T^{(2)} \leq x ) - \Pr( T^{\oracle} \leq x ) $ is bounded for each $n \in \mathbb{N}$, uniform convergence implies
		$\sup_x| \Pr(T^{(2)} \leq x ) - \Pr( T^{\oracle} \leq x ) | \to 0$ as $n \to \infty$.
		
		Lemma \ref{lemma:oracle test statistic distribution} provides that
		$\sup_x|\Pr( T^{\oracle} \leq x ) - \Pr(\chi^2(r, \nu_n) \leq x) | \to 0$ as $n \to \infty$.
		Therefore by the triangle inequality,
		\begin{align*}
			\sup_x| \Pr(T^{(2)} \leq x ) - \Pr(\chi^2(r, \nu_n) \leq x) | 
			& \leq \sup_x| \Pr(T^{(2)} \leq x ) - \Pr( T^{\oracle} \leq x ) |\\ 
			& \hspace*{0.4cm} + \sup_x|\Pr( T^{\oracle} \leq x ) - \Pr(\chi^2(r, \nu_n) \leq x) |\\
			& \to 0
		\end{align*}
		as $n \to \infty$, completing the proof.
		
	\end{proof}
	
	\subsection{Proof of the supporting lemma}
	\begin{proof}[Proof of Lemma \ref{lemma:intermediate bounds for oracle}]
		Lemma 5.1. of \cite{Shi2019} establishes \eqref{eqn:intermediate lemma 1} - \eqref{eqn:intermediate lemma 4}.
		To prove \eqref{eqn:intermediate lemma 5}, we see that
		$
		\norm{ \bm \Psi^{-1/2} \begin{bsmallmatrix} \Cmat & \bm 0 \end{bsmallmatrix} }_2 
		\leq \norm{ \bm \Psi^{-1/2} \begin{bsmallmatrix} \Cmat & \bm 0 \end{bsmallmatrix} \bm K_n^{-1/2} }_2 \norm{ \bm K_n^{1/2} }_2
		$.
		We know $ \norm{ \bm K_n^{1/2} }_2 = \lambda_{\max}^{1/2}\{ \bm K_n \} = O(1)$ from \eqref{eqn:intermediate lemma 1}.
		In addition, we find
		$$
		\norm{ \bm \Psi^{-1/2} \begin{bsmallmatrix} \Cmat & \bm 0 \end{bsmallmatrix} \bm K_n^{-1/2} }_2
		= \lambda_{\max}^{1/2} \bigg\{  \bm \Psi^{-1/2} \begin{bsmallmatrix} \Cmat & \bm 0 \end{bsmallmatrix} \bm K_n^{-1} \begin{bsmallmatrix} \Cmat' \\ \bm 0' \end{bsmallmatrix} \bm \Psi^{-1/2} \bigg\}
		= \lambda_{\max}\{ \mathbf{I}_r \} 
		= 1.
		$$
		Combining these results yields \eqref{eqn:intermediate lemma 5}.
		
		We now move on to proving \eqref{eqn:intermediate lemma 6}. 
		We will first show that
		\begin{equation}
			\norm{ \bm K_n - \Kna }_2 = O_p\left( \frac{s+m}{\sqrt{n}} \right). \label{eqn:Kna diff bound}
		\end{equation}
		Because $\bm K_n$ and $\Kna$ are symmetric, it is sufficient to prove
		$ 
		\norm{ \bm K_n - \Kna }_{\infty} = O_p\left( \frac{s+m}{\sqrt{n}} \right)
		$
		per Lemma S.8 of \cite{Shi2019}.
		Applying the Cauchy-Schwarz inequality, we find
		\begin{align*}
			\norm{ \bm K_n - \Kna }_{\infty}
			& = \max_{j \in \Mset \cup \Sset} \norm{\frac{1}{n} \xmat_j' \left[ \bm \Sigma(\xmat \bvec^*) - \bm \Sigma(\xmat \hat{\bvec}_a^{\oracle}) \right] \xmat_{\Mset \cup \Sset}}_1\\
			& = \sqrt{s+m} \max_{j \in \Mset \cup \Sset} \norm{\frac{1}{n} \xmat_j' \left[ \bm \Sigma(\xmat \bvec^*) - \bm \Sigma(\xmat \hat{\bvec}_a^{\oracle}) \right] \xmat_{\Mset \cup \Sset}}_2.
		\end{align*}
		
		Let $j \in \Mset \cup \Sset$. By the Fundamental Theorem of Calculus, we have
		\begin{align*}
			\frac{1}{n} & \xmat_j' \left[ \bm \Sigma(\xmat \bvec^*) - \bm \Sigma(\xmat \hat{\bvec}_a^{\oracle}) \right] \xmat_{\Mset \cup \Sset} \\
			& = (\bvec^* - \hat \bvec_a^{\oracle} )'\int_0^1 \frac{1}{n} \xmat' \diag \{ \xmat_j \circ b'''(t\bvec^* + (1-t)\hat{\bvec}_a^{\oracle} )  \} \xmat_{\Mset \cup \Sset} \,dt\\
			& = (\bvec^*_{\Mset \cup \Sset} - \hat \bvec_{a, \Mset \cup \Sset}^{\oracle} )'\int_0^1 \frac{1}{n} \xmat_{\Mset \cup \Sset}' \diag \{ \xmat_j \circ b'''(t\bvec^* + (1-t)\hat{\bvec}_a^{\oracle} )  \} \xmat_{\Mset \cup \Sset} \,dt
		\end{align*}
		where the integral is applied componentwise.
		This implies
		\begin{align}
			\bigg \lVert \frac{1}{n} \xmat_j' \big[ \bm \Sigma&(\xmat \bvec^*) - \bm \Sigma(\xmat \hat{\bvec}_a^{\oracle}) \big] \xmat_{\Mset \cup \Sset} \bigg \rVert_2 \nonumber \\
			& \leq \norm{\bvec^*_{\Mset \cup \Sset} - \hat \bvec_{a, \Mset \cup \Sset}^{\oracle}}_2 \nonumber \\
			& \hspace*{0.25cm} \times \sup_{t \in [0,1] }\norm{\frac{1}{n} \xmat_{\Mset \cup \Sset}' \diag \{ \xmat_j \circ b'''(t\bvec^* + (1-t)\hat{\bvec}_a^{\oracle} )  \} \xmat_{\Mset \cup \Sset}}_2. \label{eqn:int lemma 6.5}
		\end{align}
		Lemma \ref{lemma:oracle asymptotics} implies that $\Pr(\hat{\bvec}_a^{\oracle} \in \mathcal{N}_0) \to 1$.
		Clearly if $\hat{\bvec}_a^{\oracle} \in \mathcal{N}_0$, then for all $t \in [0,1]$, $t\bvec^* + (1-t)\hat{\bvec}_a^{\oracle} \in \mathcal{N}_0$.
		Therefore by \ref{assumption:glm:third derivative bound}, we have
		\begin{align*}
			\sup_{t \in [0,1] }\bigg\lVert\frac{1}{n} \xmat_{\Mset \cup \Sset}' & \diag \{ \xmat_j \circ b'''(t\bvec^* + (1-t)\hat{\bvec}_a^{\oracle} )  \} \xmat_{\Mset \cup \Sset} \bigg\rVert_2 \\
			& \leq \sup_{t \in [0,1] }\norm{\frac{1}{n} \xmat_{\Mset \cup \Sset}' \diag \{ |\xmat_j| \circ |b'''(t\bvec^* + (1-t)\hat{\bvec}_a^{\oracle} ) | \} \xmat_{\Mset \cup \Sset}}_2\\
			& = O_p(1).
		\end{align*}
		Applying this finding and \eqref{eqn:full oracle consistency} from Lemma \ref{lemma:oracle asymptotics} to \eqref{eqn:int lemma 6.5} yields
		$$
		\bigg \lVert \frac{1}{n} \xmat_j' \big[ \bm \Sigma(\xmat \bvec^*) - \bm \Sigma(\xmat \hat{\bvec}_a^{\oracle}) \big] \xmat_{\Mset \cup \Sset} \bigg \rVert_2 = O_p\left(\sqrt{\frac{s+m}{n}}\right)
		$$
		for all $j \in \Mset \cup \Sset$ and, by extension,
		$\norm{ \bm K_n - \Kna }_{\infty} = O_p\left( \frac{s+m}{\sqrt{n}}\right) $, completing our proof of \eqref{eqn:Kna diff bound}.
		
		Having proven \eqref{eqn:Kna diff bound}, we turn to \eqref{eqn:intermediate lemma 6} itself.
		We find that
		\begin{align}
			\norm{ \bm \Psi^{1/2} \left( \begin{bsmallmatrix} \Cmat & \bm 0 \end{bsmallmatrix} \Kna^{-1} \begin{bsmallmatrix} \Cmat' \\ \bm 0 \end{bsmallmatrix} \right)^{-1} \bm \Psi^{1/2} - \mathbf{I}_r }_2
			& = \norm{ \left( \bm \Psi^{-1/2}\begin{bsmallmatrix} \Cmat & \bm 0 \end{bsmallmatrix} \Kna^{-1} \begin{bsmallmatrix} \Cmat' \\ \bm 0 \end{bsmallmatrix} \bm \Psi^{-1/2}\right)^{-1} - \mathbf{I}_r }_2 \nonumber \\
			& \leq \norm{ \left( \bm \Psi^{-1/2}\begin{bsmallmatrix} \Cmat & \bm 0 \end{bsmallmatrix} \Kna^{-1} \begin{bsmallmatrix} \Cmat' \\ \bm 0 \end{bsmallmatrix} \bm \Psi^{-1/2}\right)^{-1} }_2 \nonumber \\
			& \hspace*{0.25cm} \times \norm{  \bm \Psi^{-1/2}\begin{bsmallmatrix} \Cmat & \bm 0 \end{bsmallmatrix} \Kna^{-1} \begin{bsmallmatrix} \Cmat' \\ \bm 0 \end{bsmallmatrix} \bm \Psi^{-1/2} - \mathbf{I}_r }_2 \label{eqn:int lemma 6.1}
		\end{align}
		Taking the second term of this expression, we find
		\begin{align*}
			\norm{  \bm \Psi^{-1/2}\begin{bsmallmatrix} \Cmat & \bm 0 \end{bsmallmatrix} \Kna^{-1} \begin{bsmallmatrix} \Cmat' \\ \bm 0 \end{bsmallmatrix} \bm \Psi^{-1/2} - \mathbf{I}_r }_2
			& = \norm{  \bm \Psi^{-1/2} \left( \begin{bsmallmatrix} \Cmat & \bm 0 \end{bsmallmatrix} \Kna^{-1} \begin{bsmallmatrix} \Cmat' \\ \bm 0 \end{bsmallmatrix}  - \bm \Psi \right) \bm \Psi^{-1/2}}_2 \\
			& = \norm{  \bm \Psi^{-1/2} \begin{bsmallmatrix} \Cmat & \bm 0 \end{bsmallmatrix} \left( \Kna^{-1}   - \bm K_n^{-1} \right)\begin{bsmallmatrix} \Cmat' \\ \bm 0 \end{bsmallmatrix} \bm \Psi^{-1/2}}_2\\
			& \leq \norm{ \bm \Psi ^{-1/2} \begin{bsmallmatrix} \Cmat & \bm 0 \end{bsmallmatrix} }_2^2 \norm{ \Kna^{-1}   - \bm K_n^{-1} }_2 \\
			& = O\left( \norm{ \Kna^{-1}   - \bm K_n^{-1} }_2 \right)
		\end{align*}
		by \eqref{eqn:intermediate lemma 5}.
		From here, we see that
		\begin{align}
			\norm{ \Kna^{-1}   - \bm K_n^{-1} }_2
			& = \norm{\Kna^{-1} (\bm K_n - \Kna )  \bm K_n^{-1} }_2 \nonumber \\
			& \leq \norm{ \Kna^{-1} }_2 \norm{ \bm K_n - \Kna  }_2 \norm{\bm K_n^{-1}}_2 . \label{eqn:int lemma 6.2}
		\end{align}
		
		We will bound each of the terms in \eqref{eqn:int lemma 6.2} in turn.
		We know from \ref{assumption:glm:min hessian} that
		$
		\norm{\bm K_n^{-1}}_2 = \lambda_{\max}\{ \bm K_n^{-1} \} = \lambda_{\min}^{-1} \{ \bm K_n\} = O(1)
		$.
		We find
		\begin{align}
			\lambda_{\min} \{ \Kna \}
			& = \min_{ \norm{\bm v}_2 = 1 } \bm v' \Kna \bm v \nonumber \\
			& = \min_{ \norm{\bm v}_2 = 1 } \bm v'(\bm K_n - \bm K_n + \Kna  )\bm v \nonumber \\
			& \geq \min_{ \norm{\bm v}_2 = 1 } \bm v' \bm K_n \bm v - \sup_{ \norm{\bm v}_2 = 1 } | \bm v'(\bm K_n - \Kna  )\bm v | \nonumber \\
			& = \lambda_{\min} \{ \bm K_n \} - \norm{ \bm K_n - \Kna }_2 \label{eqn:int lemma 6.3}
		\end{align}
		where the final equality follows from Lemma S.5 of \cite{Shi2019}.
		Condition \ref{assumption:glm:min hessian} provides that $\lambda_{\min}\{ \bm K_n \} >c$ for all $n$.
		Since $s + m = o(n^{1/3})$, \eqref{eqn:Kna diff bound} implies that $\norm{\bm K_n - \Kna}_2 = o_p(1) $.
		Together with \eqref{eqn:int lemma 6.3}, these results imply that for sufficiently large $n$,
		$ \lambda_{\min}\{ \Kna \} > c/2$ and therefore $\lambda_{\max}\{ \Kna^{-1} \} < 2/c$ with probability converging to $1$.
		Thus we have $\lambda_{\max}\{ \Kna^{-1} \} = O_p(1)$.
		Applying these findings and \eqref{eqn:Kna diff bound} to \eqref{eqn:int lemma 6.2}, we see that
		\begin{equation}
			\norm{ \Kna^{-1} - \bm K_n^{-1} }_2 = O_p\left( \frac{s+m}{\sqrt{n}} \right) \label{eqn:Kna inv diff bound}
		\end{equation}
		and, by extension,
		\begin{equation}
			\norm{  \bm \Psi^{-1/2}\begin{bsmallmatrix} \Cmat & \bm 0 \end{bsmallmatrix} \Kna^{-1} \begin{bsmallmatrix} \Cmat' \\ \bm 0 \end{bsmallmatrix} \bm \Psi^{-1/2} - \mathbf{I}_r }_2  = O_p\left( \frac{s+m}{\sqrt{n}} \right), \label{eqn:int lemma 6.1a}
		\end{equation}
		giving us a bound for the second term in \eqref{eqn:int lemma 6.1}.
		
		We now turn to the first term in \eqref{eqn:int lemma 6.1}.
		Using the same approach we used to derive \eqref{eqn:int lemma 6.3}, one can show
		\begin{align*}
			\lambda_{\min} \{\bm \Psi^{-1/2}\begin{bsmallmatrix} \Cmat & \bm 0 \end{bsmallmatrix} \Kna^{-1} \begin{bsmallmatrix} \Cmat' \\ \bm 0 \end{bsmallmatrix} \bm \Psi^{-1/2} \}	
			& \geq \lambda_{\min} \{\bm \Psi^{-1/2}\begin{bsmallmatrix} \Cmat & \bm 0 \end{bsmallmatrix} \bm K_n^{-1} \begin{bsmallmatrix} \Cmat' \\ \bm 0 \end{bsmallmatrix} \bm \Psi^{-1/2} \}\\
			& \hspace*{0.25cm} - \norm{ \bm \Psi^{-1/2}\begin{bsmallmatrix} \Cmat & \bm 0 \end{bsmallmatrix} ( \bm K_n^{-1} - \Kna^{-1} ) \begin{bsmallmatrix} \Cmat' \\ \bm 0 \end{bsmallmatrix} \bm \Psi^{-1/2} }_2 \\
			& \geq \lambda_{\min}\{ \mathbf{I}_r \} - \norm{\bm \Psi^{-1/2}\begin{bsmallmatrix} \Cmat & \bm 0 \end{bsmallmatrix}}_2^2 \norm{\bm K_n^{-1} - \Kna^{-1}}_2.
		\end{align*}
		When $s+m = o(n^{1/3})$, \eqref{eqn:intermediate lemma 5} and \eqref{eqn:Kna inv diff bound} provide that 
		$ \norm{\bm \Psi^{-1/2}\begin{bsmallmatrix} \Cmat & \bm 0 \end{bsmallmatrix}}_2^2 \norm{\bm K_n^{-1} - \Kna^{-1}}_2 =  O_p\left( \frac{s+m}{\sqrt{n} }\right) = o_p(1) $. 
		This implies
		$ \liminf_n  \lambda_{\min} \{\bm \Psi^{-1/2}\begin{bsmallmatrix} \Cmat & \bm 0 \end{bsmallmatrix} \Kna^{-1} \begin{bsmallmatrix} \Cmat' \\ \bm 0 \end{bsmallmatrix} \bm \Psi^{-1/2} \} > 0$
		with probability converging to $1$ and, by extension,
		\begin{align}
			\norm{ \left( \bm \Psi^{-1/2}\begin{bsmallmatrix} \Cmat & \bm 0 \end{bsmallmatrix} \Kna^{-1} \begin{bsmallmatrix} \Cmat' \\ \bm 0 \end{bsmallmatrix} \bm \Psi^{-1/2}\right)^{-1} }_2 
			& = \lambda_{\max}\left\{  \left( \bm \Psi^{-1/2}\begin{bsmallmatrix} \Cmat & \bm 0 \end{bsmallmatrix} \Kna^{-1} \begin{bsmallmatrix} \Cmat' \\ \bm 0 \end{bsmallmatrix} \bm \Psi^{-1/2}\right)^{-1} \right\} \nonumber \\
			& = \lambda_{\min}^{-1} \{\bm \Psi^{-1/2}\begin{bsmallmatrix} \Cmat & \bm 0 \end{bsmallmatrix} \Kna^{-1} \begin{bsmallmatrix} \Cmat' \\ \bm 0 \end{bsmallmatrix} \bm \Psi^{-1/2} \} \nonumber \\
			& = O_p(1) . \label{eqn:int lemma 6.1b}
		\end{align}
		Applying the bounds in \eqref{eqn:int lemma 6.1a} and \eqref{eqn:int lemma 6.1b} to the right hand side of \eqref{eqn:int lemma 6.1}, we arrive at \eqref{eqn:intermediate lemma 6}.
		
		We now turn to proving \eqref{eqn:intermediate lemma 7}.
		We see that
		\begin{align*}
			\norm{ \bm K_n^{1/2} \Kno^{-1} \bm K_n^{1/2} - \mathbf{I}_{m+s} }_2 
			& = \norm{ \bm K_n^{1/2} \left( \Kno^{-1} - \bm K_n^{-1} \right) \bm K_n^{1/2} }_2 \\
			& \leq \norm{ \bm K_n^{1/2} }_2^2 \norm{\Kno^{-1} - \bm K_n^{-1}}_2\\
			& = O\left( \norm{\Kno^{-1} - \bm K_n^{-1}}_2 \right)
		\end{align*}
		by \eqref{eqn:intermediate lemma 1}.
		By the same argument we used to prove \eqref{eqn:Kna inv diff bound}, we can show $\norm{\Kno^{-1} - \bm K_n^{-1}}_2 = O_p\left(\frac{s+m}{\sqrt n}\right)$, completing the proof of \eqref{eqn:intermediate lemma 7}.
	\end{proof}

	\section{Discussion of Technical Conditions} \label{sec:technical conditions}
	In this appendix we verify that \ref{assumption:glm:self concordant} holds in linear, logistic, and Poisson regression models.
	Recall that the canonical density of an exponential family takes the form
	\begin{equation*}
		p(y|\theta, \phi) = \exp \left(\frac{y \theta - b(\theta)}{\phi}\right)c(y)\text{,} 
	\end{equation*}
	where $\theta$ is the canonical parameter. We verify that \ref{assumption:glm:self concordant} holds below:
	\begin{itemize}
		\item Linear regression: Here $b(\theta) = \theta^2/2$, so $b''(\theta) = 1$ and $b'''(\theta) = 0$ and \ref{assumption:glm:self concordant} clearly holds.
		\item Logistic regression: Here $b(\theta) = \log(e^\theta + 1)$. One can show that $b''(\theta) = (e^\theta + 2 + e^{-\theta})^{-1}$ and that $b'''(\theta) = b''(\theta)(1 - 2(e^{-\theta} +1)^{-1})$. As such, we find
		$ |b'''(\theta)| = b''(\theta)|(1 - 2(e^{-\theta} +1)^{-1})| \leq b''(\theta) $ and \ref{assumption:glm:self concordant} holds.
		\item Poisson regression: Here $b(\theta) = e^\theta$, so $b'''(\theta) = b''(\theta) = e^\theta$ and \ref{assumption:glm:self concordant} clearly holds.
	\end{itemize}

\end{document}